\documentclass[11pt
]{article}

\usepackage[margin=1in]{geometry}

\usepackage{amsthm, amssymb, bm, amsmath,mathrsfs,cancel}

\usepackage{soul, xcolor}
\usepackage{tikz-cd}
\usepackage{schemata}
\usepackage{caption}
\usepackage{bbm}
\usepackage{aligned-overset}
\usepackage{xcite}
\usepackage{graphicx}

\usepackage{accents}

\usepackage{imakeidx}
\makeindex[columns=3,intoc, title=Index of notation ]

\usepackage[
colorlinks=true
,pdfpagemode=UseNone,urlcolor=blue,linkcolor=blue,citecolor=blue]{hyperref}

\usepackage[colorinlistoftodos]{todonotes}

\newtheorem{theorem}{Theorem}[]
\newtheorem{definition}{Definition}[section]
\newtheorem{remark}[definition]{Remark}

\newtheorem{proposition}[definition]{Proposition}
\newtheorem{lemma}[definition]{Lemma}
\newtheorem{corollary}[definition]{Corollary}
\newtheorem{theoremofothers}[definition]{Theorem}

\newcommand \Cm { \mathbb{C}}
\newcommand \Dm { \mathbb{D}}
\newcommand \Rm { \mathbb{R}}
\newcommand \Nm { \mathbb{N}}
\newcommand \Sm { \mathbb{S}}
\newcommand \Zm { \mathbb{Z}}
\newcommand \zbar { {\bar{z}}}
\renewcommand{\P}{\mathcal{P}}
\renewcommand{\d}{\mathrm{d}}

\let \Re \undefined
\DeclareMathOperator{\Re}{Re}

\DeclareMathOperator{\Diff}{Diff}
\DeclareMathOperator{\supp}{supp}

\renewcommand \L {{\cal L}}
\newcommand \D {{\cal D}}
\newcommand \A {{\cal A}}
\newcommand \F {{\cal F}}
\newcommand \G {{\cal G}}
\newcommand \Gh {{\cal G}_{\mathrm h}}
\newcommand \Ghbar {\overline{{\cal G}_{\mathrm h}}}

\renewcommand \H {{\cal H}}
\renewcommand \S {{\cal S}}
\newcommand \T {{\cal T}}

\newcommand \x { {\mathrm x}}

\newcommand \sfa {{\sf{a}}}

\newcommand\Stt {S_{\mathrm{tt}}}

\DeclareMathOperator{\tr}{tr}

\newcommand{\Aphg}{\mathcal{A}_{\mathrm{phg}}}

\newcommand{\OSM}{{}^0\overline{SM}}
\newcommand{\OSD}{{}^0\overline{S\mathbb{D}}}

\newcommand{\ev}{{\mathrm{ev}}}

\newcommand \Cev {C_{\mathrm{ev}}^\infty}

\newcommand \pih {\pi_\mathrm{h}}

\newcommand{\bSM}{{}^b\overline{ S^*M}}

\newcommand \muh {\mu_{\mathrm h}}

\newcommand \hide[1] {}

\renewcommand \tr {\mathrm{tr}_g}
\newcommand \itt {\mathrm{itt}}
\renewcommand \div {\delta_g}

\usepackage[shortlabels]{enumitem} %
\usepackage{mathtools}
\mathtoolsset{showonlyrefs} %
\allowdisplaybreaks %

\title{Tensor Tomography on Asymptotically Hyperbolic Surfaces}
\author{Nikolas Eptaminitakis\thanks{Institut f\"{u}r Differentialgeometrie, Leibniz Universit\"{a}t Hannover, Welfengarten 1, 30167 Hannover, Germany; email: nikolaos.eptaminitakis@math.uni-hannover.de} \and Fran\c{c}ois Monard\thanks{Department of Mathematics, University of California, Santa Cruz CA 95064; email:fmonard@ucsc.edu} \and Yuzhou Joey Zou\thanks{Department of Mathematics and Statistics, Oakland University, Rochester MI 48309; \newline email: yzou@oakland.edu}}
\date{}

\numberwithin{equation}{section}

\begin{document}
\maketitle

\begin{abstract}
    We initiate a study of the inversion of the geodesic X-ray transform $I_m$ over symmetric $m$-tensor fields on asymptotically hyperbolic surfaces. This operator has a non-trivial kernel whenever $m\ge 1$. To propose a gauge representative to be reconstructed from X-ray data, we first prove a ``tt-potential-conformal'' decomposition theorem for $m$-tensor fields (where ``tt'' stands for transverse traceless), previously used  in integral geometry on compact Riemannian manifolds with boundary in \cite{Sharafudtinov2007}, \cite{Dairbekov2011}. The proof is based on elliptic decompositions of the Guillemin-Kazhdan operators $\eta_\pm$ (\cite{Guillemin1980}) and leverages in the current setting the 0-calculus of Mazzeo-Melrose \cite{Mazzeo1987,Mazzeo1991}. Iterating this decomposition gives rise to an ``iterated-tt''  representative modulo $\ker I_m$ for a tensor field, which is distinct from the often-used solenoidal representative. 

    In the case of the Poincar\'e disk, we show that the X-ray transform of a tensor in iterated-tt form splits into components that are orthogonal relative to a specific $L^2$ structure in data space. We then provide a full picture of the data space decomposition, in particular a range characterization of $I_{m}$ for every $m$ in terms of moment conditions and spectral decay. Finally, we give explicit  approaches for the reconstruction of tensors in iterated-tt form from their X-ray transform or its normal operator, using specific knowledge of geodesically invariant distributions with one-sided Fourier content, whose properties are analyzed in detail. 
\end{abstract}

\tableofcontents

\section{Introduction} \label{sec:intro}

We consider the tensor tomography problem on Asymptotically Hyperbolic surfaces. A Riemannian metric $g$ defined in the interior $M^{\circ}$ of a compact manifold with boundary $M$ is called Asymptotically Hyperbolic (AH) if for any boundary defining function (``bdf'' throughout) $\rho$ for $\partial M$, $\overline{g}\coloneqq \rho^2g$ extends to a smooth Riemannian metric on $M$ and in addition one has $\| \d \rho\big|_{\partial M}\|_{\rho^2g}^2=1$. Such metrics are complete  and their sectional curvatures approach $-1$ as $\rho\to 0$ (\cite{Mazzeo1986}). The most important example is the Poincaré model of hyperbolic space. Given a non-trapping AH manifold $(M^{\circ},g)$ (that is, a manifold carrying an AH metric such that no unit speed geodesic spends infinite time in any compact subset of $M^{\circ}$), the space of oriented unit speed geodesics through $M^{\circ}$ can be given the structure of a manifold $\G \cong T^* (\partial M)$. Then for a symmetric $m$-tensor field on $M^\circ$, we can define the geodesic X-ray transform $I_m\colon C_c^\infty(M^\circ; S^m(T^* M^\circ))\to C_c^\infty(\G)$, given by  
\begin{align}
    I_m f (\gamma) = \int_{\Rm} f_{\gamma(t)} (\dot\gamma(t)^{\otimes m})\ dt, \qquad \gamma\in \G,
    \label{eq:Imf}
\end{align}
where $\gamma$ is in unit-speed parameterization. The tensor tomography problem consists in assessing which part of a given symmetric $m$-tensor $f$ can be reconstructed from $I_m f$, and how to reconstruct it. %
In the AH setting, such a problem is related to the non-linear inverse problem of reconstructing an AH metric from knowledge of its renormalized boundary distance data, see \cite{Graham2019,Lefeuvre2020,humbert2026}. 

Prior literature on this question is extensive in the case of convex and non-trapping manifolds with boundary, covering both its applications (to medical imaging and seismology, among others) and its theoretical developments. Known results include injectivity (up to natural kernel) and stability statements \cite{Sharafudtinov1994,Paternain2023,Paternain2019}, range characterizations, and reconstruction procedures \cite{Monard2015a,Derevtsov2015,Kazantsev2004,Krishnan2019}.

More recently, the X-ray transform over tensor fields has gained attention in the case of non-compact spaces including AH ones \cite{Graham2019}, Cartan-Hadamard manifolds \cite{Lehtonen2017} and asymptotically conic spaces \cite{Jia2022}. Such results usually follow more special cases where the background geometry is a symmetric space \cite{Michel1984}, or the integrand is a function or a one-form \cite{Helgason1999,Berenstein1994,Lehtonen2016,Eptaminitakis2024,Eptaminitakis2022,Eptaminitakis2021,Vasy2024,Hoop2024}.

The transform \eqref{eq:Imf} with $m=2$ can also be related to a linearization of entanglement entropy functionals in the Anti-de Sitter (AdS)/Conformal Field Theory (CFT) correspondence. 
The article \cite{Cao2020} proposes numerical approaches to recover a perturbation of a spacelike slice of the $\text{AdS}_3$ bulk metric from hyperbolic X-ray data, as well as to determine whether certain boundary measurements are in the range of the hyperbolic X-ray transform acting on symmetric tensors of order 2. The approach there is based on minimizing misfit functionals via constrained optimization and points to the absence of analytic inversions and range characterizations in the literature, which, as we discuss below, are addressed in detail in the present work.

For the problem we consider, the most recent result to date is \cite[Corollary 1.2]{Graham2019}, stating a unique characterization of $\ker I_m$ on a non-positively curved AH manifold: for $m\ge 0$ and $f\in \rho^{1-m} C^\infty(M, S^m(T^* M))$, %
we have that $I_m f =0$ if and only if $f=\d ^s q$ for $q\in \rho ^{2-m}C^\infty(M, S^{m-1}(T^* M))$, where $\d^{s} $ denotes the symmetrized {total} covariant derivative (in short, ``$f$ is a potential tensor").
While this kernel characterization result is an important starting point to envision what part of a tensor field is reconstructible in these geometries, it does not address reconstruction nor range characterization. Indeed, the method of energy identities used to obtain the above result is not suited for these purposes. The present article sheds some light on these questions in the following ways: 

\medskip
$\bullet$ We first address the choice of a tensor representative modulo potential tensors, coined ``iterated tt\footnote{short for {\bf t}ransverse-{\bf t}racefree} representative" below, to be reconstructed from X-ray data, see Corollary \ref{cor:gauge}. 
On convex, non-trapping compact surfaces with boundary, such a representative arises from iteratively applying  a {``tt-potential-conformal''} decomposition \cite[Theorem 3.3]{Sharafudtinov2007}, and was shown in \cite{Monard2015a,Kay2025} to lend itself to reconstruction approaches with desirable features. These include: the reconstruction is systematic at all tensor orders; iterated-tt tensors decompose into components whose X-ray transforms are in direct sum; the reconstruction of the tt components can be templated using special invariant distributions with one-sided fiber-harmonic content. These reasons make it preferable over the {\em solenoidal} representative (which arises from Sharafutdinov's potential-solenoidal decomposition \cite[Theorem 3.3.2]{Sharafudtinov1994}, often used in Fourier-based Euclidean X-ray inversions), or the {\em normal gauge} representative, used in \cite{Stefanov2017} to deal with the local X-ray problem near a convex boundary point. 

A first salient feature of this article is thus to generalize the tt-potential-conformal decomposition result \cite[Theorem 1.5]{Dairbekov2011} to oriented, simply connected AH surfaces, given in Theorem \ref{thm:gauge_maybe} below. Such a result requires deriving new elliptic decompositions for the Guillemin-Kazhdan operators $\eta_\pm$ (first defined in \cite{Guillemin1980}) in the AH context, and we show that this can be treated using the 0-calculus of Mazzeo-Melrose \cite{Mazzeo1987,Mazzeo1991}. The X-ray transform is then shown to be {\em injective} over such representatives in situations where solenoidal injectivity (that is, injectivity up to potential tensor) is already known, e.g., when the results in \cite{Graham2019} apply.

\medskip
$\bullet$ Second, we give a template for the reconstruction of an iterated-tt even tensor from its X-ray transform, emphasizing that some steps require results which exist in the literature of simple surfaces, but are still open in the AH case. We then show that all steps of this template can be successfully and explicitly carried out when $(M^\circ,g)$ is the Poincar\'e disk. Moreover, in this case the iterated-tt representative of a symmetric tensor field splits into components whose X-ray transforms live on {\em orthogonal} subspaces of an $L^2$-type space on $\G$. This provides insights into a full range decomposition and characterization of the X-ray transform over tensor fields of arbitrary order, see Corollary~\ref{thm:range} and Theorem~\ref{thm:moments} below. Finally, we provide two approaches to explicitly reconstruct this representative from either its X-ray transform (see Theorem \ref{thm:reconsXray} below), or the associated normal operator of the latter (see Theorem \ref{thm:recons_normal} below). The approach used to prove Theorem \ref{thm:reconsXray} mirrors the template laid out in \cite{Monard2015a,Kay2025} for the Euclidean disk. In this template, the reconstruction of the tt components requires the construction of appropriate geodesically invariant distributions with one-sided fiberwise Fourier content, to be paired with {X-ray} data. This geometry-specific construction is generally an open question, and is resolved here for the Poincar\'e disk. 

\medskip
We should mention that both features above are inherently two-dimensional. The first one may admit an analogue on AH manifolds of general dimension by considering the higher-dimensional analogues of $\eta_{\pm}$, though the details are more technical. The second step leverages crucially the notion of fiberwise holomorphicity relative to the circle action on the tangent fibers, along with ideas regarding the algebraic properties of fiberwise holomorphic invariant distributions, which have yet to be generalized to higher-dimensional settings.

Another important feature of this work lies in the fact that since the original manifold $M^\circ$ is non-compact, addressing ``boundary behavior'' involves a choice of extension of $SM^\circ$ and $S^*M^\circ$ (the unit sphere and cosphere bundle respectively), all the way to $\partial M$. As is observed below, for AH surfaces the compactification compatible with fiberwise Fourier decomposition ($\OSM$ below) differs from the extension resolving geodesic transport equations (${}^b \overline{S^* M}$ below). Then the reconstruction of tensor fields requires appropriate pairings between objects that are well-behaved (in the sense of being polyhomogeneous conormal on them) on each of the two: to address fiberwise Fourier content, it is best to consider tensor fields that extend naturally to $\OSM$, while the geodesic vector field and thus relevant geodesically invariant distributions are well-behaved on ${}^b \overline{S^* M}$. 

Several future directions are of interest. One would be to develop an explicit reconstruction method for the solenoidal part of a symmetric tensor on  hyperbolic space of arbitrary dimension. In another direction, it would be desirable to generalize the reconstruction results beyond the constant curvature setting. For the reconstruction from X-ray data (Theorem \ref{thm:reconsXray}), it is reasonable to expect that our method would apply in the AH surface case under suitable geometric assumptions, such as an analogue of simplicity (see Section \ref{sssec:template}). On the other hand, our reconstruction procedure from the normal operator (Theorem \ref{thm:recons_normal}) relies on the orthogonality of the X-ray transforms of the iterated tt components (Theorem \ref{thm:dataDecomp}), and it is unclear if one should expect an analogous result to hold in the non-constant curvature setting.

We now state the main results and provide an outline of the remainder of the article at the end.

\section{Statement of the main results}

\subsection{Results on AH Spaces}
\label{sec:results_on_AH_spaces}

\subsubsection{The unit sphere and cosphere bundles and their compactifications.}\label{sub:the_unit_sphere_and_cosphere_bundles_and_their_compactifications_}

Let $(M^{\circ},g)$ be an $n$-dimensional AH manifold and $\rho$ a bdf. The metric $g$ induces a conformal class of metrics on $\partial M$, given by $[\rho^{2}g\big|_{T\partial M}]$ and called the \emph{conformal infinity}. As shown in \cite{Graham1991}, for each choice $h_0$ of  representative  in the conformal infinity there exists a unique bdf ${\tilde x}$ called a \emph{geodesic bdf} such that $\| \d \tilde x\|_{\tilde x^{2}g}=1$ in a neighborhood of $\partial M$ and $\tilde x^{2}g\big|_{T\partial M}=h_0$.
Via the flow of its $\tilde{x}^{2}g$-gradient, it induces a product decomposition of a collar neighborhood of $\partial M$ in terms of which
\begin{equation}\label{eq:normal_form}
    g=\frac{\d {\tilde x}^{2}+h_{\tilde x}}{{\tilde x}^{2}},
\end{equation}
where $h_{\tilde x}$ is a smooth 1-parameter family of metrics on $\partial M$.
We can then choose coordinates $(y^{2},\dots,y^{n})$   for $\partial M$ near a point there and obtain a coordinate system $(u^{1},\dots, u^{n})=(\tilde x,y^{2},\dots,y^{n})$ for $M$.
Throughout the paper we use the Einstein summation convention, assuming that Latin (resp. Greek) indices run between 1 and $n$ (resp. 2 and $n$).

\paragraph{A model for the manifold of geodesics $\G$.} We will follow  \cite{Graham2019} to parametrize the geodesic flow of $(M^{\circ},g)$.
The statements in this section are elaborated there.
Recall that the $b$-tangent bundle $^{b}TM$ is  a smooth rank $n$-bundle over $M$ characterized by the property that its smooth sections are the vector fields on $M$ that are tangent to the boundary. 
In coordinates near a point in $\partial M$ as above, its smooth sections are spanned over $C^{\infty}(M)$ by ${\tilde x}\partial_{\tilde x}$, $\partial_{y^2},\dots ,\partial_{y^n}$.
The sections of the dual bundle ${}^bT^*M$ are locally spanned by $\frac{\d {\tilde x}}{{\tilde x}},\d y^{2},\dots, \d y^n$.
Over $M^{\circ},$  $TM^{\circ}$ and $T^{*}M^{\circ}$ are naturally identified with $ ^{b}TM\big|_{M^\circ}$ and $ ^{b}T^*M\big|_{M^\circ}$ respectively. \\
The unit cosphere bundle $S^{*}M^\circ$ is given by
\begin{equation}
    S^{*}M^\circ\coloneqq \{(u,\zeta)\in T^*M^{\circ}:g(\zeta,\zeta)=1\}.
\end{equation}
It can be viewed as a subset of $^{b}T^*M$, in which case its closure is a non-compact submanifold with boundary of $^{b}T^*M$ denoted by $^{b}\overline{S^*M}$, and expressed  near the boundary as
\begin{equation}\label{eq:bSM_coords}
    ^{b}\overline{S^*M}= \Big\{(u,\overline{\zeta})=\big(\tilde x,y,\bar{\xi}\frac{\d \tilde x}{\tilde x}+\eta_\alpha \d y^{\alpha}\big):\bar{\xi}^{2}+{\tilde x}^2h_{\tilde x}^{\alpha\beta}\eta_\alpha\eta_{\beta}=1\Big\},
\end{equation}
where the metric is given by \eqref{eq:normal_form}.\footnote{Throughout, lower (resp. upper) indices indicate components of a metric (resp. cometric).} 
See Fig.~\ref{fig:bsm}. Its boundary has two connected components, given by 
\begin{equation}
    \partial_\pm \bSM= \Big\{(u,\overline{\zeta})=\big(\tilde x,y,\bar{\xi}\frac{\d \tilde x}{\tilde x}+\eta_\alpha \d y^{\alpha}\big): {\tilde x =0},\  \bar\xi=\pm 1\Big\},
\end{equation}
and $\tilde x$ is a bdf for $\partial\, \bSM\coloneqq \partial_+\bSM\cup \partial_-\bSM$. The sets $\partial_\pm \bSM\subset {}^bT^*M$ are canonical, independent of $g$.

Write $\mathcal{X}$ for the generator of the geodesic flow on  $S^{*}M^{\circ}$. There exists a smooth vector field $\overline{\mathcal{X}}$ on $ \bSM$ which is transversal to $\partial\, \bSM$ and satisfies $\mathcal{X}=\tilde x\overline{\mathcal{X}}$; hence the integral curves of $\mathcal{X}$ and $\overline{\mathcal{X}}$ agree up to reparametrization (and project to unparameterized geodesics on $M^{\circ}$).
Moreover, one has $\pm\d \tilde x (\overline{\mathcal{X}})>0$ on $\partial_\pm \bSM$ and thus the latter can be viewed as the incoming/outgoing boundary\footnote{In this paper we will use the opposite from \cite{Graham2019} convention for the incoming/outgoing boundary, but we are in agreement with the convention used in \cite{Eptaminitakis2024}.}.
Assuming that $(M^{\circ},g)$ is non-trapping, 
one can parameterize the space of geodesics on $M^{\circ}$ as
\begin{equation}\label{eq:calG}
    \mathcal{G}\coloneqq \partial_+\bSM,
\end{equation}
which can be identified with $T^{*}\partial M$ via $ \frac{\d \tilde x}{\tilde x}+\eta_\alpha \d y^{\alpha}\leftrightarrow \eta_\alpha \d y^{\alpha}$, with the identification only depending on the representative in the conformal infinity.
In what follows it will be convenient to consider the radial compactification of $\mathcal{G}$, denoted by $\overline{\G}$, that is, compactify it as a compact manifold with boundary with bdf
\begin{equation}\label{eq:mu_defn}
    \mu_{\tilde x}(y,\eta) \coloneqq \langle \eta\rangle_{\tilde x^2g}^{-1}=(1+ \tilde x^{-2}g^{\alpha\beta}\eta_\alpha\eta_\beta)^{-1/2}.
\end{equation}
If $\rho$ is any bdf, $\rho^{2}g\big|_{T\partial M}$ determines a geodesic bdf and identification of $ \partial_+\bSM$ with $T^*\partial M$ as above. 
So $\rho$ determines a bdf for $\partial \G$, which we denote with $\mu_\rho$. 

\paragraph{The 0-sphere bundle.} Below we will consider the X-ray transform of symmetric tensor fields on $M^{\circ}$, and it is well known that these can be identified in a natural way with restrictions of polynomials in $TM^{\circ}$ to the unit sphere bundle, given by
\begin{equation}
    SM^\circ \coloneqq \{(u,v)\in TM^{\circ}:g(v,v)=1\}.
\end{equation}
It turns out that a natural way to extend this correspondence to the boundary is by viewing $SM^\circ$ as a subset of the 0-tangent bundle, as we now explain.
The 0-tangent bundle ${}^0TM$ (see \cite{Mazzeo1987}) is a rank $n$-vector bundle over $M$ whose smooth sections are spanned over $C^\infty(M)$ by vector fields vanishing at $\partial M$, that is, if $u=(u^1,\dots ,u^n)$ are local coordinates\footnote{Here we do not not assume that one of the coordinates is a bdf. This will be convenient later, as we would sometimes like our coordinates to be isothermal when considering AH surfaces. 
} near $p\in \partial M$ and $\rho$ a bdf, the smooth sections of ${}^0TM$ are spanned over $C^\infty$ by $\rho\partial_{u^j}$, $j=1,\dots,n$.
The dual bundle of ${}^0TM$ is denoted by ${}^0T^*M$, and its local sections are spanned over $C^{\infty}$ by $\frac{\d u^j}{\rho}$, $j=1,\dots, n$.
As before, the  bundles ${}^0TM$, ${}^0T^*M$ are canonically identified over $M^\circ$ with $TM^{\circ}$ and $T^*M^{\circ}$ respectively.
In this way, smooth tensor fields of any order on $M^\circ$ can be canonically identified with smooth (over $M^\circ$) sections  of corresponding  tensor products of ${}^0TM\big|_{M^\circ}$ and ${}^0T^*M\big|_{M^\circ}$.

Consider an {AH metric} $g$ on $M^\circ$. 
 By means of the canonical identification $\iota:{}^0TM\big|_{M^\circ}\to TM^\circ$,  $g$ induces in a natural way an inner product  on the fibers of $ {}^0TM$, still denoted by $g$, which extends to be smooth and non-degenerate all the way up to $\partial M$.
Indeed, using local coordinates near $\partial M$ to parametrize $\OSM$ as 
\begin{equation}\label{eq:coordinates0TM}
     (u,v)\mapsto \big((u^1,\dots,u^n),v^j \rho\partial_{u^j}\big)\in {}^0TM,
 \end{equation} 
  we have 
\begin{equation}
    {g}\big( v^j \rho\partial_{u^j}, v^k \rho\partial_{u^k}\big)= \rho^2 {g}( v^j \partial_{u^j},v^k \partial_{u^k})= \overline{g}_{jk} v^jv^k,
\end{equation}
and $\overline{g}$ extends to a smooth non-degenerate metric all the way to $\partial M$ by assumption.
Since $g$ is  a smooth metric on ${}^0TM$, the unit sphere bundle it defines is a smooth embedded submanifold with boundary of ${}^0TM$.
We denote it by $\OSM$, written locally in coordinates \eqref{eq:coordinates0TM} as
\begin{equation}
    \OSM=\{(u, v^j \rho\partial_{u^j})\in {}^0TM:\overline{g}_{jk} v^jv^k=1\}.
    \label{eq:OSM}
\end{equation}
An alternative way to describe it is as the closure of $SM^{\circ}$ when viewed as a subset of $^{0}TM$. We mention that $\OSM$ is in a natural way an edge space in the sense of \cite{Mazzeo1991}, see also Remark \ref{rem:edge}.
Note that $\OSM$ is \emph{not} the same space as the unit sphere bundle of $\bar{g}$ inside $TM$.

\subsubsection{Mapping properties of the X-ray transform on \texorpdfstring{$L^2$}{L 2} and \texorpdfstring{$\Aphg$}{phg} spaces}

On a non-trapping AH manifold $(M^\circ,g)$, the geodesic X-ray transform can be generally thought of as a map $I \colon C_c^\infty(SM^\circ)\to C_c^\infty(\G)$, given by 
\begin{align}
    I f(y,\eta) = \int_{\Rm} f(\gamma_{y,\eta}(t), \dot\gamma_{y,\eta}(t)) \ \d t, \qquad (y,\eta) \in \G.
    \label{eq:IH}
\end{align}
More generally, we want to consider Fr\'echet and Hilbert functional settings that incorporate boundary behavior. With the polyhomogeneous spaces $\Aphg^{E}$ and the various notations related to index sets defined in  Appendix \ref{sec:phg}, we first prove the following, which is analogous to \cite[Theorem 2.1]{Mazzeo2021} in the compact case: 
\begin{proposition}\label{prop:phg_mapping_I} Let $(M^\circ,g)$ be a non-trapping AH manifold and let $E$ be an index set with $\Re(E)>0$. Then the operator \eqref{eq:IH} extends as a map
    \begin{align}\label{eq:Iphg_mapping}
	I\colon \Aphg^E (\OSM)\to \Aphg^E(\overline{\G}).
    \end{align}  
Moreover, for every $\alpha>0$ and bdf $\rho$ on $M$, there exists a function $R_\alpha\in C^\infty(\overline{\G})$ such that 
\begin{equation}\label{eq:Ixalpha}
      I\rho^{\alpha}(\gamma)=B\Big(\frac{\alpha}{2},\frac12\Big)\mu_\rho^{\alpha}(\gamma)+\mu_\rho^{\alpha+1}R_{\alpha}(\gamma), \qquad \gamma \in \overline{\mathcal{G}},
  \end{equation}  
   where $B$ is the Beta function. 
   Moreover, there exist $\mu',\alpha'>0$ such that for all $\mu_\rho\leq \mu'$ and $0<\alpha\leq \alpha'$, $\big|\frac{\mu_{\rho} R_\alpha}{B(\alpha/2,1/2))}\big|\leq 1/2$. 
\end{proposition}

\begin{remark}\label{rmk:phg_hyp}
    In the case of the Poincar\'e disk and relative to the bdf's $x$ and $\muh =\mu_x$ defined in Eq.\ \eqref{eq:bdfs} below, a direct calculation shows that $R_{\alpha}=0$ for all $\alpha>0$. 

\end{remark}

In what follows, we also need to address the X-ray transform in $L^2$-type settings and derive further mapping properties in that context. To this end, we will use as reference measures $\d V_g$, $\d\Sigma^{2n-1}$ (See Section \ref{ssub:measures_}) and $\d\Sigma_\partial = \d y\d \eta$ (in the notation of Sec.  \ref{sub:the_unit_sphere_and_cosphere_bundles_and_their_compactifications_}) on $M$, $\OSM$ and $\G$ respectively, and will use the following shorthand notation: 
\begin{align}
    L^2(M) \coloneqq  L^2(M, \d V_g), \quad L^2(\OSM) \coloneqq  L^2(\OSM, \d\Sigma^{2n-1}), \quad L^2(\G) \coloneqq  L^2(\G, \d\Sigma_\partial). 
    \label{eq:notation}
\end{align}
We will also use the measures above without explicitly denoting them to define $L^2$ sections of Hermitian vector bundles over $M$ and $\OSM$. Among them, the most important  will be sections of the (complexified) $m$-th symmetric tensor product of ${}^0T^*M$ over $M$; on this bundle, $g$ induces a canonical smooth (up to $\partial M$) Hermitian fiber metric, as explained in Section~\ref{ssub:symmetric_tensors}, and it will be always assumed that this Hermitian fiber metric is used to define $L^2$ inner products and Sobolev spaces.

Recall from \cite[Lemma 3.8]{Graham2019} that for every $\beta>\frac{1}{2}$, the operator \eqref{eq:IH} extends as a bounded operator $I\colon |\log \rho|^{-\beta} L^2(\OSM)\to L^2(\G)$ (where $\rho$ is a bdf as usual). %
We start with some further observations regarding forward mapping properties. %

\begin{lemma}\label{lem:surjective}\footnote{We would like to thank C Robin Graham for pointing out an inaccuracy in an earlier formulation of this lemma.}
    Let $(M^\circ,g)$ be a non-trapping AH manifold.
    \begin{enumerate}[(a)]
	\item\label{item_surj_1} For every $\delta>0$, the operator \eqref{eq:IH} viewed as a densely defined map on $C_c^\infty(S M^\circ)\subset L^2(\OSM)$ with values in $\mu_{\rho}^{\delta} L^2(\G)$ extends to a bounded and surjective map
	    \begin{align}
		I \colon \rho^{\delta} L^2(\OSM) \longrightarrow \mu_{\rho}^{\delta} L^2(\G). %
		\label{eq:IHgamma}
	    \end{align}
	\item\label{item_surj_2} The operator \eqref{eq:IH} viewed as a densely defined map $I \colon C_c^\infty (SM^\circ) \subset  L^2(\OSM) \longrightarrow L^2(\G)$ is unbounded.
     
     \item\label{item_surj_3} There exists a bounded, positive weight $w\in C^\infty(SM^\circ)$ such that \eqref{eq:IH} extends to a bounded and surjective operator $wL^2(\OSM)\subset L^2(\OSM)\to L^2 (\G)$.

    \end{enumerate}
\end{lemma}

\noindent The statement of Lemma \ref{lem:surjective}, together with the inclusion $\Aphg^E(\OSM)\subset \rho^{\delta} L^2(\OSM)$ for all $\delta$ such that $\Re(E)>\delta+ \frac{n-1}{2}$, imply that for polyhomogeneous spaces the condition $I(\Aphg^E(\OSM)) \subset L^2(\G)$ holds provided $\Re(E)>\frac{n-1}{2}$.
This is somewhat more restrictive than the results in Section \ref{sub:representatives_modulo_gauge_for_symmetric_tensor_fields} below.

\subsubsection{Representatives modulo gauge for symmetric tensor fields}\label{sub:representatives_modulo_gauge_for_symmetric_tensor_fields}

Proposition \ref{prop:phg_mapping_I} implies that  if $m\ge 0$ and $E$ is a smooth index set with $\Re(E)>0$, we can define the X-ray transform $I_m \colon \Aphg^E(M;{{S}}^m({}^0T^*M))\to \Aphg^E(\overline{\G})$ over tensor fields with boundary behavior as  $I_m = I \circ \ell_m$; here $\ell_m$, defined in \eqref{eq:lm}, is a map used  to realize smooth sections of ${{S}}^m({}^0T^*M)$ as smooth functions on $\OSM$. This operator has ``potential'' tensors as a natural kernel, as the next result indicates.
\noindent The $0$-Sobolev spaces $\rho^\delta H_0^{s}$ are defined in Appendix \ref{sec:the_0_calculus}.

\begin{lemma}\label{lm:kernel} Let $m\geq 1$. On a non-trapping AH manifold, provided $q\in \Aphg^{E}(M;S^{m-1}({}^{0}T^*M))$  with  $\Re(E)>0$ or $q\in \rho^\delta H_0^{s}(M;S^{m-1}({}^{0}T^*M))$ with $s\geq 1$ and $\delta>0$ we have $I_m \d^{s}q=0$ a.e., where $\d^{s}$ stands for the symmetrized covariant derivative, see Section \ref{sec:OSM}. %
\end{lemma}

We now particularize to AH surfaces ($n=2$), and address tensor field decompositions that lead to gauge representatives modulo $\ker I_m$ which, as explained below, have good reconstructibility properties. Here and below, $\div$ and $\tr$ denote the divergence and trace operators on tensor fields respectively, and $L$ is the symmetrized product with the metric defined in \eqref{eq:L}. For an integer $m\ge 2$ and a smoothness class ${\cal F} \in\{ \rho^\delta L^2, C^\infty, \dots\}$, we also denote
\begin{align}
\begin{split}
    {\cal F} (M; \Stt^m({}^0 T^* M)) &= \{u\in {\cal F} (M; S^m({}^0 T^* M)),\ \div u= 0,\ \tr(u) = 0\}, \quad m\ge 2,  \\
    {\cal F} (M; \Stt^1({}^0 T^* M)) &= \{u\in {\cal F} (M;{}^0 T^* M),\ \div u = \d u= 0\},    
    \end{split}
    \label{eq:ttspaces}
\end{align}
where the derivatives are to be understood in a distributional sense when they do not make sense classically. 
For one-forms, the decomposition considered is nothing but the usual Hodge decomposition. For one-forms in $L^2(M)$, the latter was proved in \cite{Mazzeo1986}; here we include a statement for weighted spaces.

\begin{lemma}\label{lem:oneforms} 
    Let $(M^\circ,g)$ be an oriented, {simply connected} AH surface, and fix $s\in \mathbb{R}$ and $\delta\in (-1/2,1/2)$. Then each $f\in \rho^\delta H^s_0(M;{}^0T^*M)$ decomposes uniquely in the form
    \begin{align}
        f = \d q_0 + \star\d f_s + \tilde{f}_1, 
        \label{eq:oneform_decomp}
    \end{align}
    where $q_0, f_s \in \rho^\delta H^{s+1}_0(M)$ and $\tilde{f}_1 \in \rho^\delta H^s_0(M;\Stt^1({}^0T^*M))$. 
    Here, $\star$ denotes the Hodge star on one-forms. Moreover, there exists $C>0$ independent of $f$ such that     \begin{equation}\label{eq:cont_dep_oneform}
	\| q_0\|_{\rho^\delta H^{s+1}_0(M)}+ \| f_s\|_{\rho^\delta H^{s+1}_0(M)} + \|\tilde{f}_1\|_{\rho^\delta H^s_0(M)} \leq C\| f\|_{\rho^\delta H^s_0(M)}.
    \end{equation}
    The conclusions above also hold if $f\in \Aphg^E(M;{}^0T^*M)$ with $\Re(E)>0$, in which case
    $q_0, f_s\in \Aphg^G(M)$ with $\Re {(G)}\geq \min\{1,\inf(E)\}^{-}$, and $\tilde{f}_1 \in \Aphg^{G\cup E}(M;{}^0T^*M)$.    
\end{lemma}

For  symmetric tensor fields of order $m\geq 2$, we now establish a decomposition  analogous to \cite[Theorem 1.5]{Dairbekov2011} in the case of oriented, non-trapping AH surfaces. For AH manifolds of general dimension, such a result was proved in \cite[Theorem 1.3]{Gicquaud2010} in the case $m=2$ and we generalize it here to general tensor order on surfaces.

\begin{theorem}\label{thm:gauge_maybe}
    Let $(M^\circ,g)$ be an oriented, {simply connected} AH surface, $s\in \Rm$, $m\ge 2$ and $\delta\in (-m+1/2,m-1/2)$. Then each $f\in \rho^\delta H^s_0(M;S^m({}^0T^*M))$ decomposes uniquely in the form 
    \begin{align}
	f = \d^s q + L \lambda + \tilde f,
    \label{eq:tensorDecomp}
    \end{align}
    where $q\in \rho^\delta H^{s+1}_0(M;S^{m-1}({}^0T^*M))$ satisfies $\tr(q)=0$, $\lambda\in \rho^\delta H^s_0(M;S^{m-2}({}^0T^*M))$ and $\tilde{f}\in \rho^\delta H^s_0(M;\Stt^m({}^0T^*M))$. Moreover, there exists $C>0$ independent of $f$ such that
    \begin{equation}\label{eq:cont_dep_maybe}
	\| q\|_{\rho^\delta H^{s+1}_0(M)}+ \| \lambda\|_{\rho^\delta H^{s}_0(M)} + \|\tilde{f}\|_{\rho^\delta H^s_0(M)} \leq C\| f\|_{\rho^\delta H^s_0(M)}.
    \end{equation}
    The conclusions above also hold if $f\in \Aphg^E(M;{{S}}^m({}^0T^*M))$ with $\Re(E)>-m+1$, in which case
    $q\in \Aphg^G(M;{{S}}^{m-1}({}^0T^*M))$ with $\Re {(G)}\geq \min\{m,\inf(E)\}^{-}$, and $L\lambda,$ $ \tilde{f} \in \Aphg^{G\cup E}(M;{{S}}^m({}^0T^*M))$.
\end{theorem}
\noindent

\begin{remark}[The case of higher boundary decay]
    For $\delta>0$ and $s\in \Rm$, denote $x^{\delta -} H_0^s := \cap_{0<\delta'<\delta} x^{\delta'} H_0^s$. 
    Although Lemma \ref{lem:oneforms} and Theorem \ref{thm:gauge_maybe} appear to have an upper threshold of boundary decay (of the form $\delta <\delta_0$), decompositions \eqref{eq:oneform_decomp} and \eqref{eq:tensorDecomp} are still possible if $f\in \rho^\delta H_0^s$ for $\delta\ge \delta_0$, however in this case, the summands in \eqref{eq:oneform_decomp} and \eqref{eq:tensorDecomp} can at best belong to $\rho^{\delta_0-}H_0^\bullet$ (with continuity estimates of the form \eqref{eq:cont_dep_oneform} or \eqref{eq:cont_dep_maybe} for every $\delta<\delta_0$) and cannot be guaranteed to decay faster.
\end{remark}

\paragraph{The iterated-tt representative.} For any $m$-tensor field, using Theorem \ref{thm:gauge_maybe} iteratively $\lfloor m/2\rfloor$ times, followed by Lemma \ref{lem:oneforms} if $m$ is odd, and using that $[L,\d ^s]=0$, we obtain the following decomposition, which we state without proof. 

\begin{corollary}[Iterated-tt decomposition] \label{cor:gauge} 
    Let $(M^\circ,g)$ be an oriented, simply connected AH surface, and fix $s\in \Rm,$ $m\ge 1$, and $\delta$ such that $|\delta|< 3/2- (m\mod 2)$. Then each tensor field $f\in \rho^\delta H^s_0(M;S^{m}({}^0T^*M))$ decomposes uniquely in the form 
    \begin{align}
        f = \d^s q + f^{\itt}, 
        \label{eq:gauge_decomp}
    \end{align}
    where $q\in \rho^\delta H^{s+1}_0(M;S^{m-1}({}^0T^*M))$ and $f^{\itt} \in \rho^\delta H^s_0(M;S^{m}({}^0T^*M))$ takes the following form: 
    \begin{itemize}
        \item If $m=2n$ for some $n>0$, 
        \begin{align}
	       f^{\itt} \coloneqq  \sum_{k=0}^n L^{n-k} \tilde f_{2k},
	       \label{eq:fitt_even}
        \end{align}
        where $\tilde f_0 \in \rho^\delta H_0^s(M)$, and for $1\le k\le n$ we have $\tilde f_{2k} \in \rho^\delta H^s_0(M;\Stt^{2k}({}^0T^*M))$. 
        \item If $m=2n+1$ for some $n\ge 0$,
        \begin{align}
	           f^{\itt} \coloneqq L^n(\star\d f_s) +  \sum_{k=0}^n L^{n-k} \tilde f_{2k+1}, 
	       \label{eq:fitt_odd}
        \end{align}
        where $f_s \in \rho^\delta H_0^{s+1}(M)$, and for $0\le k\le n$ we have $\tilde f_{2k+1} \in \rho^\delta H^s_0(M;\Stt^{2k+1}({}^0T^*M))$. 
    \end{itemize}    
    Moreover, there exists $C>0$ independent of $f$ such that
    \begin{equation}\label{eq:cont_dep_cor}
	\| q\|_{\rho^\delta H^{s+1}_0(M)}+ \|f^{\itt}\|_{\rho^\delta H^s_0(M)} \leq C\| f\|_{\rho^\delta H^s_0(M)}.
    \end{equation}
    Decomposition \eqref{eq:gauge_decomp} also holds if $f\in \Aphg^E(M;{{S}}^{m}({}^0T^*M))$ with $\Re(E)>-1+(m\mod 2)$, in which case $q$ has regularity $\Aphg^G$ with $\Re {(G)}\geq \min\{2-(m\mod 2),\inf(E)\}^{-}$, and $f^{\itt}$ has regularity $\Aphg^{G\cup E}$. 
\end{corollary}

\hide{
\todo[inline]{If $\delta\ge 2-\mod(m,2)$, maybe the same decomposition holds, where some of the lower tt components can be taken to be zero.}
}
Considering \eqref{eq:gauge_decomp}, \eqref{eq:fitt_even} and \eqref{eq:fitt_odd} when $\delta>0$,  Lemma \ref{lm:kernel} implies that $I_{m} \d^s q = 0$, i.e. $I_m f = I_m f^{\itt}$ and thus, using the fact that $I_{m} \circ L^q = I_{m-2q}$ for any $q\le \lfloor m/2\rfloor$, we have
\begin{align}
    \begin{split}
        I_{2n} f &= I_{2n} \Big( \sum_{k=0}^n L^{n-k} \tilde f_{2k} \Big) = \sum_{k=0}^n I_{2k} \tilde f_{2k} \qquad ({m}\ \text{even}), \\
        I_{2n+1} f &= I_{2n+1} \Big( L^n (\star\d f_s) + \sum_{k=0}^n L^{n-k} \tilde f_{2k+1} \Big) = I_1 (\star\d f_s) + \sum_{k=0}^n I_{2k+1} \tilde f_{2k+1} \qquad (m\ \text{odd}).
    \end{split}
    \label{eq:Xray_decomp}
\end{align}
We call the tensor $f^{\itt}$ appearing in \eqref{eq:gauge_decomp} the {\em iterated-tt} representative of $f$. While in the next section will propose a method for reconstructing $f^{\itt}$ from $I_{m}f$ on AH surfaces (and carry it out in the case of the Poincaré disk in Section~\ref{sec:intro_recons}), a first observation is to show that the X-ray transform is injective over it in some situations. 

\begin{proposition}[Injectivity]\label{prop:inj_on_repr_maybe2}
    Let $(M^\circ,g)$ be an oriented, non-trapping, non-positively curved AH surface and $m\ge 1$. Let $f\in \Aphg^E(M;{{S}}^{m}({}^0T^*M))$ with $\Re(E)>0$, and consider the decomposition \eqref{eq:gauge_decomp}. If $I_{m} f=0$, then $f^{\itt} = 0$. In particular, $f$ is a potential tensor. 
\end{proposition}

\subsubsection{A template to reconstruct iterated-tt tensors from X-ray data }\label{sssec:template}
Given the injectivity result Proposition \ref{prop:inj_on_repr_maybe2}, we now give a general step-wise scheme to reconstruct an $m$-tensor in iterated-tt form, along with the results needed, which are believed to be generally true yet not proved in the AH context. Writing $f^\itt$ as in \eqref{eq:fitt_even} or \eqref{eq:fitt_odd} depending on the parity of $m$, a template which has been successful in prior literature on simple surfaces (see, e.g., \cite{Monard2015a,Krishnan2019,Kay2025}) is as follows: 
\begin{enumerate}
    \item First reconstruct the top tt component $\tilde{f}_{m}$ by pairing the data $I_m f$ with special functions on $\G$ which, when extended to $SM^\circ$ as geodesically invariant distributions, are $L^2(SM^\circ)$-orthogonal to all tensor components of order less than $m$. On simple surfaces with boundary, this can be done using vertical Fourier analysis considerations (as defined Section \ref{sec:dim2} below), see, e.g. \cite[Lemma 7]{Krishnan2019}. 
    The construction of such distributions hinges on the landmark surjectivity theorem \cite[Theorem 1.4]{Pestov2005} for the adjoint of $I_0$, which is a cornerstone result in the study of 2D geometric inverse problems on simple Riemannian surfaces, see \cite{Paternain2023}. It is an open question to find analogues to either result in the AH context. In the case of the Poincar\'e disk,  we bypass the overall approach and show how to explicitly construct such invariant distributions, in Section \ref{sec:xipm} below. 

    Once $\tilde{f}_{m}$ is reconstructed, remove $I_{m} \tilde{f}_{m}$ from the data and repeat with the next tt component, until $\tilde{f}_2$ or $\tilde{f}_1$ included, depending on the parity of $m$.

    \item When all tt components are recovered, it remains to reconstruct $\tilde{f}_0$ from $I_0 \tilde{f}_0$, or $f_s$ from $I_1 (\star\d f_s)$. On simple Riemannian surfaces, pointwise reconstruction formulas exist (see, e.g., \cite[Theorem 5.1]{Pestov2004}), and in AH cases, only for $I_0$ on the Poincar\'e disk (see references in Remark \ref{rem:f0recons} below), leaving the case of simple AH surfaces open. In geometries where Singular Value Decompositions are available (e.g. rotationally-invariant metrics), SVD-based inversions can also be envisioned. 
\end{enumerate}

In the next section, we give avenues for explicitly solving both steps above in the case of the Poincar\'e disk, along with new additional results on reconstruction formulas for
\begin{align}
    I_\perp := I_1 \circ (-\star\d),
    \label{eq:Iperp}
\end{align}
data space decompositions, and range characterization of the hyperbolic X-ray transform over tensor fields of arbitrary order.

\subsection{Results on the Poincar\'e disk} \label{sec:intro_Poincare}

We now consider $M = \Dm \coloneqq  \{z\in \Cm,\ |z|\le 1\}$ and equip $\Dm^\circ$ with the Poincar\'e metric $g(z) = c^{-2}(z) |\d z|^2$, where $c(z) = \frac{1-|z|^2}{2}$. Recall from \cite{Eptaminitakis2024} that in this case, the set 
{$\G$ defined in \eqref{eq:calG} becomes }
 \begin{align}
     \Gh = (\Rm/2\pi\Zm)_\beta\times \Rm_a, 
     \label{eq:Gh}
 \end{align}
 where the geodesic corresponding to $(\beta,a)$ is given by \eqref{eq:hypgeo}.
Moreover, by \cite[Lemma 3.9]{Eptaminitakis2024}, one has the correspondence $(\beta,a)\leftrightarrow (y,-\eta dy)\in T^*\partial \Dm$ in the notation of Section \eqref{sub:the_unit_sphere_and_cosphere_bundles_and_their_compactifications_}, so $\d\Sigma_\partial = \d\beta\d a$.
Throughout,  $L^2(\Dm)$ is with respect to the hyperbolic volume form $\d V_H$.
Below, we will use the following bdfs for $\Dm$ and for the compactification $\overline{\Gh}$: 
\begin{align}
    x(z) \coloneqq  \frac{1-|z|^2}{1+|z|^2}, \quad z\in \Dm;\qquad \muh(\beta,a) \coloneqq  \frac{1}{\sqrt{1+a^2}}, \quad (\beta,a)\in \Gh.
    \label{eq:bdfs}
\end{align}
Using the results in \cite[Sec. 3.2.3]{Eptaminitakis2024} one can check that $\muh=\mu_x$.

In what follows, we will fix $m\ge 0$, $\delta>0$ (so that Proposition \ref{prop:inj_on_repr_maybe2} holds), and we will focus on the range characterization and inversion of $I_{m}$ on $x^\delta L^2(\Dm; S^{m} ({}^0 T^*\Dm))$, when a symmetric tensor  is decomposed as in \eqref{eq:gauge_decomp}. In light of \eqref{eq:Xray_decomp}, an important step is to explain how each summand of \eqref{eq:fitt_even} or \eqref{eq:fitt_odd} can be reconstructed separately. The case $m=0$ was addressed in \cite{Eptaminitakis2024}.  

In Section \ref{sec:intro_range}, we will explain how the summands in \eqref{eq:Xray_decomp} can be located in the data space $L^2(\Gh)$. Notably, they are $L^2(\Gh)$-orthogonal and, as $n$ runs through $\Nm_0$, they flag $L^2_\pm(\Gh)$ (defined in Eq. \eqref{eq:L2pm}). 

We then give explicit avenues for tensor field reconstructions in Section \ref{sec:intro_recons}, both from X-ray transform data and from the ``normal'' operator $N_{m}$ defined there. 

\subsubsection{Data space decomposition and range characterization} \label{sec:intro_range}

Recall that $\Gh$ is a 2-to-1 cover of the set of unoriented geodesics, with deck transformation the ``antipodal scattering relation'' 
\begin{align}\label{eq:antipodal_sc}
    \S_A\colon \Gh\to \Gh, \qquad (\beta,a) \mapsto (\beta+\pi+2\tan^{-1}a, -a).
\end{align}
The geodesic $\gamma_{\beta,a}$ corresponding to $(\beta,a)\in \Gh$, whose expression is given in Eq.\ \eqref{eq:hypgeo} below, is the same as $\gamma_{\S_A(\beta,a)}$ as a set, parametrized in the opposite direction. Then the space $L^2(\Gh)$ admits an orthogonal splitting 
\begin{align}
    L^2(\Gh) = L_+^2(\Gh) \stackrel{\perp}{\oplus} L_-^2(\Gh), \qquad L^2_{\pm}(\Gh) \coloneqq  L^2(\Gh) \cap \ker (Id \mp \S_A^*), 
\label{eq:L2pm}
\end{align} 
and the X-ray transform over $m$-tensor fields belongs to $L_{+/-}^2$ for $m$ even/odd. 

To understand how the data spaces decompose, a starting point is to recall that from results in \cite{Eptaminitakis2024} (see also Sec. \ref{sec:dataDecomp} below), the $L^2_+(\Gh)$-closure of $I_0 ( x^{\frac{1}{2}} L^2(\Dm))$ is understood in terms of the special basis of $L_+^2(\Gh)$ defined Eq.\ \eqref{eq:basis1} below, and this helps ``locate'' the term $I_0 f_0$ in \eqref{eq:Xray_decomp}. To address the other terms, we need a similar story for the transform $I_\perp$, and we need to study the X-ray transform over tt tensors. On the latter, since elements of $x^\delta L^2 (\Dm; \Stt^{k}({}^0 T^* \Dm))$ take the form $f(z) \d z^k + g(z)\d \zbar^k$ with $f$ holomorphic and $g$ antiholomorphic (see Lemma \ref{lm:conv} and Remark \ref{rm:holodiff} below), studying their image under the X-ray transform hinges on studying
\begin{align}
    I_{p,q} \coloneqq  I_q [z^p \d z^q] , \quad p\ge 0, \quad q\ge 1.
    \label{eq:Ipq}
\end{align}
We will see in Proposition \ref{prop:Ipq_ortho} below that $I_{p,q}\in L^2(\Gh)$, and hence we may define $\widehat{I_{p,q}} = I_{p,q}/\|I_{p,q}\|_{L^2(\Gh)}$. To formulate the decomposition for $L^2_-(\Gh)$, we introduce the space 
$H^{1,0}_w(\Dm) := \overline{\Cev(\Dm)}^{H^{1,0}_w}$, where the $H^{1,0}_w$ norm is defined by
\begin{align}
    \|u\|_{H^{1,0}_w}^2:=\int_{\Sm^1}\int_0^1 x(1-x^2)|\partial_xu|^2+ \frac{x}{1-x^2}|\partial_\omega u|^2+x|u|^2 \d x\d\omega,
    \label{eq:H10w}
\end{align}
$\omega = \arg(z)$, and $\Cev(\Dm)$ denotes the functions in $C^\infty(\Dm)$ whose Taylor expansion at $\partial \Dm$ with respect to the distinguished bdf $x$ in \eqref{eq:bdfs} only contains even powers. One can show that $\|u\|_{H^{1,0}_w}^2 = (\L_0^Hu,u)_{L^2(\Dm,\ x^3\d V_H)}$, where $\L_0^H$ is a  distinguished second order \emph{wedge} differential operator defined in \cite[Section 2.4]{Eptaminitakis2024}, hence the ``$w$" subscript. This definition of $H^{1,0}_w(\Dm)$ coincides with that introduced in \cite[Section 2.5]{Eptaminitakis2024} as the domain of $(\L_0^H)^{1/2}$. To compare this space to the 0-Sobolev scale, we state the following lemma, whose proof is relegated to Appendix \ref{app:0Sob}: 
\begin{lemma}\label{lem:H10w}
    For every $\epsilon>0$, the following inclusions hold:
    \begin{align}
        x^{1/2}H_0^1(\Dm) \subsetneq xH_w^{1,0}(\Dm) \subsetneq x^{1/2-\epsilon} H_0^1(\Dm). 
        \label{eq:H10comparison}
    \end{align}
\end{lemma}
The following result gives key data space decompositions, see also Figure \ref{fig:data} for a schematic.

\begin{theorem}[Data space decomposition]\label{thm:dataDecomp}
    We have the following decomposition
    \begin{align}
	L^2_+(\Gh) &= \overline{I_0 (x^{1/2} L^2(\Dm))}^{L^2} \stackrel{\perp}{\oplus} \mathrm{span}_{L^2} \left(\widehat{I_{p,2q}}, p\ge 0, q\ge 1\right) \stackrel{\perp}{\oplus} \mathrm{span}_{L^2} \left(\widehat{\overline{I_{p,2q}}}, p\ge 0, q\ge 1\right), 
	\label{eq:dataDecomp_even}  \\
    L^2_-(\Gh) &= \overline{I_\perp (xH^{1,0}_w(\Dm))}^{L^2} \stackrel{\perp}{\oplus} \mathrm{span}_{L^2} \left(\widehat{I_{p,2q+1}}, p\ge 0, q\ge 0\right) \stackrel{\perp}{\oplus} \mathrm{span}_{L^2} \left(\widehat{\overline{I_{p,2q+1}}}, p\ge 0, q\ge 0\right),
    \label{eq:dataDecomp_odd}
    \end{align}
    where in each equality above, the second and third families are $L^2(\Gh)$-orthonormal.   %
\end{theorem}

We now discuss the regularity and invertibility of the X-ray transform on tt tensors of arbitrary order. For an orthonormal family $\{\phi_p\}_{p\ge 0}$ in $L^2(\Gh)$ and $s\ge 0$, we denote 
\begin{align*}
    h^s(\phi_p,\ p\ge 0) = \Big\{u \in \text{span}_{L^2}(\phi_p,\ p\ge 0), \quad u = \sum_{p\ge 0} u_p \phi_p, \quad \sum_{p\ge 0} |u_p|^2 (p+1)^{2s} <\infty\Big\},
\end{align*}
which is a Hilbert space with norm $\|u\|_s = \big(\sum_{p\ge 0} |u_p|^2 (p+1)^{2s}\big)^{1/2}$. In the case $s=0$ we write $\ell^{2}(\phi_{p},\ p\geq 0)$. Then we have the following result. 

\begin{proposition}\label{prop:IStt_2k}
    For $k>0$ and $\delta\in (0,k-1/2)$,
     the operator 
    \begin{align*}
	I_{k} \colon x^\delta L^2 (\Dm, \Stt^{k} ({}^0 T^* \Dm)) \longrightarrow h^\delta ( \widehat{I_{p,k}},\ p\ge 0) \oplus h^\delta ( \widehat{\overline{I_{p,k}}},\ p\ge 0)
    \end{align*}
    is a homeomorphism. 
\end{proposition}

Combining Corollary \ref{cor:gauge}, Theorem \ref{thm:dataDecomp} and Proposition \ref{prop:IStt_2k}, we obtain the following, which in particular states that the X-ray transforms of the components of the decompositions \eqref{eq:fitt_even}-\eqref{eq:fitt_odd} are orthogonal.

\begin{corollary}[Range decomposition]\label{thm:range}
    For any $n\ge 0$, we have the following range decompositions, orthogonal in $L^2(\Gh)$: 
    \begin{align}
    \begin{split}
        	I_{2n}\left( x^\delta L^2(\Dm, S^{2n} ({}^0 T^* \Dm))\right) &= I_0 (x^\delta L^2(\Dm)) \stackrel{\perp}{\oplus} \bigoplus_{k=1}^n I_{2k} (x^\delta L^2(\Dm, \Stt^{2k} ({}^0 T^* \Dm))),\ \delta\in (0,3/2),\\
            I_{2n+1}\left( x^\delta L^2(\Dm, S^{2n+1} ({}^0 T^* \Dm))\right) &= I_\perp (x^\delta H^1_0(\Dm)) \stackrel{\perp}{\oplus} \bigoplus_{k=0}^n I_{2k+1} (x^\delta L^2(\Dm, \Stt^{2k+1} ({}^0 T^* \Dm))),\ \delta\in (0,1/2).
    \end{split}
	\label{eq:L2decomp}
    \end{align}    
\end{corollary}

\begin{remark}
    Combining Lemma~\ref{lem:oneforms}, Theorem \ref{thm:gauge_maybe} and \eqref{eq:L2decomp}, one may also write decompositions analogous to \eqref{eq:L2decomp} in $\Aphg^E$ spaces provided that ${\Re}(E)>1/2$. %
\end{remark}

\paragraph{Range characterization.} In light of Theorem \ref{thm:dataDecomp}, upon defining the orthogonal projections
\begin{align}
    \begin{split}
        \Pi_0&\colon L^2_+ (\Gh)\to \overline{\left( I_0 (x^{1/2}L^2(\Dm)) \right)}^{L^2(\Gh)}, \qquad \Pi_\perp \colon L^2_- (\Gh)\to \overline{\left( I_\perp (xH^{1,0}_w(\Dm)) \right)}^{L^2(\Gh)}, \\
        \Pi_{k} &\colon L^2 (\Gh) \to \ell^2 ( \widehat{I_{p,k}},\ p\ge 0) \oplus\ell^2 (\widehat{\overline{I_{p,k}}},\ p\ge 0), \qquad k\ge 1,        
    \end{split}
    \label{eq:orthoProj}    
\end{align}
we have the decompositions of the identity: 
\begin{align*}
    Id_{L^2_+(\Gh)} = \sum_{q\ge 0} \Pi_{2q}, \qquad Id_{L^2_-(\Gh)} = \Pi_\perp + \sum_{q\ge 0} \Pi_{2q+1}.
\end{align*}

Then we have the following range characterization.

\begin{theorem}[Range characterization] \label{thm:moments}
    Let $m$ be a natural integer, $\delta\in (0,3/2- (m\mod 2))$, and $\sigma = (-)^m$. Then a function $u\in L_\sigma^2(\Gh)$ belongs to the range of $I_{m}$ over $x^{\delta}L^2(\Dm, S^{m} ({}^0 T^* \Dm))$ if and only if
    \begin{enumerate}[$(i)$]
    	\item\label{range_item1} For every $q>0$, $\Pi_{m+2q} u = 0$.
        \item\label{range_item2} For every $0\le q < \lceil m/2 \rceil$, $\Pi_{m-2q} u \in h^{\delta} ( \widehat{I_{p,m-2q}},\ p\ge 0) \oplus h^{\delta} ( \widehat{\overline{I_{p,m-2q}}},\ p\ge 0)$.
        \item\label{range_item3} $\Pi_0 u \in I_0 (x^\delta L^2(\Dm))$ (case $m$ even), or $\Pi_\perp u \in I_\perp (x^\delta H^1_0(\Dm))$ (case $m$ odd).
    \end{enumerate}
\end{theorem}

\begin{remark}
    Condition \ref{range_item1} can be viewed as a generalization of the Helgason-Ludwig-Gelfand-Graev conditions and could be formulated equivalently as the vanishing moments condition
    \begin{enumerate}
        \item[(\textit{i}\,$'$)] $\qquad(h,I_{p,m+2q})_{L^2(\Gh)} = (h, \overline{I_{p,m+2q}})_{L^2(\Gh)} = 0\quad$\textit{ for every }$p\ge 0$ \textit{ and } $q>0$.
    \end{enumerate}
    Conditions \ref{range_item2} and \ref{range_item3} are concerned with spectral decay in certain distinguished bases.     
\end{remark}

\begin{remark}[On condition \ref{range_item3}] When $m$ is even, the condition $\Pi_0 h \in I_0 (x^\delta L^2(\Dm)))$ can be rewritten as 
\begin{center}
    $\Pi_0 h = \sum_{n=0}^\infty \sum_{k=0}^n a_{n,k} \sigma_{n,k}^{\delta-1/2} \psi_{n,k}^{\delta-1/2,H}$, where the $a_{n,k}$'s satisfy $\sum_{n=0}^\infty \sum_{k=0}^n |a_{n,k}|^2 <\infty$,
\end{center}
 the basis functions $\psi_{n,k}^{\gamma,H}$ for $\gamma>-1$ are defined in \eqref{eq:basis1}, and $\sigma_{n,k}^{\gamma}\in \Rm$ are defined in \cite[Eq. (21)]{Eptaminitakis2024}. %
Such an expansion is not $L^2_+(\Gh)$-orthogonal unless $\delta = 1/2$. For the latter value, this same condition \ref{range_item3} can also we rewritten somewhat more compactly and intrinsically, as the following equivalent condition
    \begin{enumerate}
	\item[(iii\,$'$)] $\Pi_0 h \in H_{T,D,+}^{1/2} (\Ghbar)$, 
    \end{enumerate}   
    where the space $H_{T,D,+}^{1/2} (\Ghbar)$ is defined in \cite[Sec. 2.7]{Eptaminitakis2024}. 

    When $m$ is odd, regarding characterizations of $I_\perp (x^\delta H^1_0(\Dm))$, Proposition \ref{prop:Ipq_ortho} helps understand its orthocomplement in $L^2_-(\Gh)$, though the spectral decay and the basis where to formulate sharp regularity remains unclear. On the other hand, Equation \eqref{eq:IperpRange} below gives a sharp range characterization of $I_\perp (x H^{1,0}_w(\Dm))$, where the space $x H^{1,0}_w(\Dm)$ satisfies the inclusions \eqref{eq:H10comparison} relative to the $0$-Sobolev scale.   
\end{remark}

\subsubsection{Reconstruction} \label{sec:intro_recons}
 
Because the range characterization is explicit, the ``only if'' part of Theorem \ref{thm:moments} provides an SVD-based reconstruction of each component of the iterated-tt representative of a symmetric tensor. In this section, we propose alternative routes towards reconstruction of the higher-order tensor components (written as $\{\tilde f_{k}\}_{1\le k\le m}$ and $f_s$ in the notation of Corollary \ref{cor:gauge}, where $k\equiv m\mod 2$), from $I_{m}f$ via integral kernels and from the normal operator $N_{m}$. Interestingly, both approaches involve inverting triangular systems, the former from higher to lower tensor modes, the latter from lower to higher. 

\paragraph{Reconstruction from X-Ray Data.} 
We have the following:

\begin{theorem}[Reconstruction of the tt components from X-ray data]\label{thm:reconsXray}
    Fix $m\ge 1$, $\delta\in (0,m-1/2)$, and suppose $f\in x^{\delta}L^2(\Dm, S^m ({}^0 T^* \Dm))$ takes the form
    \begin{description}
        \item[($m\ge 2$)] $f = Lf_{m-2} + \tilde{f}_{m}$, where $f_{m-2}\in x^{\delta}L^2(\Dm, S^{m-2} ({}^0 T^* \Dm))$ and $\tilde{f}_m \in x^{\delta}L^2(\Dm, \Stt^m ({}^0 T^* \Dm))$,
        \item[($m=1$)] $f = \star\d f_s + \tilde{f}_1$, where $f_s\in x^\delta H^1_0(\Dm)$ and $\tilde{f}_1 \in x^{\delta}L^2(\Dm, {}^0 T^* \Dm)$.
    \end{description}     
    Writing $\tilde{f}_m = f_{m,+}\d z^{m} + f_{m,-}\d\zbar^{m}$, $\tilde{f}_{m}$ is reconstructed from $\D=I_{m} f$ via the formulas
    \begin{align}
	    f_{m,+} = \int_{\Gh} {\cal D} (\beta,a)\ G_{m}(\beta,a; z)\ \d\beta\d a, \qquad f_{m,-} = \int_{\Gh} {\cal D} (\beta,a)\ \overline{G_{m}(\beta,a; z)}\ \d\beta\d a,
	\label{eq:reconsfk2}
    \end{align}
    where 
    \begin{align*}
	G_{m}(\beta,a; z) \coloneqq  \frac{2^m (-1)^m e^{-im(\beta+\tan^{-1}a)}}{8\pi^2(2m-2)! (1+a^2)^{m/2}} \sum_{p=0}^\infty \frac{(p+2m-1)!}{p!} \left(\frac{ia ze^{-i(\beta+\tan^{-1}a)}}{\sqrt{1+a^2}}\right)^{p}.
    \end{align*}
\end{theorem}

Theorem \ref{thm:reconsXray} is the basis for reconstructing all tt components of an arbitrary tensor in iterated-tt form, in decreasing tensor order.
As explained in Section \ref{sssec:template}, this reduces the problem to the reconstruction of $\tilde f_0$ from $I_0 \tilde f_0$ (even case) or $f_s$ from $I_1(\star \d f_s)$ (odd case).
For the former, there exists extensive prior work, see Section \ref{sec:reconsXray}.
To our knowledge, the latter has not been addressed in the literature and
 can be achieved by means of the following:

\begin{theorem}\label{thm:recons_m=1}
   For $f_s\in x^\delta H^1_0(\Dm)$, $\delta >0$ we have  a.e. \begin{equation}\label{eq:PU_inversion}
        f_s=\frac{1}{8\pi}I_0^\sharp(A_{+} ^{H})^{*}\mathcal{H}_+ A_-^{H}I_1\star \d f_s.
    \end{equation}
    The operators $A_+^H$, $\mathcal{H}_+$, $A_-^H$ are defined in Section \ref{sec:dataSpace} and $I_0^\sharp$ in \eqref{eq:backproj}.
\end{theorem}
Alternatively, one could reconstruct $f_s$ using an SVD-based approach, see Proposition~\ref{prop:IperpSVD}.

\begin{remark}
    Eq. \eqref{eq:PU_inversion} mirrors exactly the Pestov-Uhlmann inversion formula  on simple surfaces of constant curvature (see \cite{Pestov2004,Monard2015}). 
\end{remark}

\paragraph{Reconstruction from the normal operator.} 
Let $m\ge 0$ be an integer. The \emph{normal operator} associated with $I_m=I\ell_m$ is
\begin{equation}
    N_m\coloneqq \ell_m^*I^\sharp I \ell_m,
\end{equation}
where, briefly, $I^\sharp$ is extension by constancy along the geodesic flow and $\ell_m^*$ is the $L^2$ adjoint of $\ell_m$, see Section~\ref{ssub:reconstruction_from_the_normal_operator_proof_of_theorem_ref_thm_recons_normal} for details. There we also explain that for $\delta>0$, $N_m:x^{\delta}L^2(\Dm)\to x^{-\delta}L^2(\Dm)$ is bounded.

An inversion formula  for $N_0$ acting on compactly supported elements was derived by Berenstein and Casadio Tarabusi in \cite{Berenstein1991}. As a corollary of the SVD's proved in \cite{Eptaminitakis2024} and in Section \ref{sec:Iperp} below, we also obtain an inversion formula for $N_0$.
It has a different form from theirs, so we record it here. 
\begin{theorem}\label{thm:inversionN0}
    On $x^{1/2}L^2(\Dm)$, $\frac{1}{16\pi^2}(-\star \d N_1\star \d)  N_0= {Id}.$
\end{theorem}

Below we derive inversion formulas for $N_m$ for any $m\geq 1$ (modulo kernel).
For any smoothness class $\mathcal{F}\in \{C^r,\dot{C}^\infty,\Aphg^E,x^\delta L^2\}$ we will use  operators\footnote{On a manifold with boundary $X$, $\dot{C}^{\infty}(X)$ stands for functions that vanish at $\partial X$ to infinite order, with all derivatives.}
\begin{equation}
    Q_{m,k}: \mathcal{F}(\Dm;S^{m}(^{0}T^*\Dm))\to \mathcal{F}(\Dm;S^{m}(^{0}T^*\Dm)),  \quad -m\leq k\leq m\quad k\equiv m\mod 2 ,
\end{equation}
which project to the $\mathcal{F}$-span of   $c^{-m}\d z^{\frac{m+k}{2}}\d \bar z^{\frac{m-k}{2}}$. %
We use the notation $\mathcal A_m$ for the real section of  $\mathrm{Aut}(S^m({}^0T^{*}\Dm))$ 
 defined by
 \begin{equation}\label{eq:Akt2}
     \mathcal A_m \Big(\frac{\d z^{k}\d \bar{z}^{m-k} }{c^{m}}\Big) = 2^m \binom{m}{k}^{-1} \Big(\frac{\d z^{k}\d \bar{z}^{m-k} }{c^{m}}\Big), \qquad 0\le k\le m. 
\end{equation}

\begin{theorem}[Reconstruction from the normal operator]\label{thm:recons_normal}
Fix an integer $m\geq 0$ and $\delta\in (0,1/2)$.
and suppose that $f\in x^{\delta}L^2(\Dm;S^{m}(^0T^*\Dm))$.
Denote $\mathsf{D}\coloneqq N_m f\in x^{-\delta}L^2(\Dm;S^{m}(^0T^*\Dm))$. We have:
\begin{itemize}
    \item \textbf{$m$ even}:
Decompose $f$ as in \eqref{eq:gauge_decomp}.
We have the inversion formula
\begin{align}
    {\tilde{f}}_0 &= R_0\P_0\ell_m \mathcal{A}_{m}\mathsf{D},\label{eq:normal_inversion}\\
    \begin{split}
      L^{\frac{m-k'}{2}}{\tilde{f}}_{k'}
        &= \frac{Q_{m,k'}+Q_{m,-k'}}{4\pi  B\big(\frac{1}{2},k'\big)}\mathcal{A}_{m}\mathsf{D}\\ 
        &\qquad-  \sum _{\substack{0\leq k< k'\\ k \text{ even}}}  \frac{Q_{m,k'}\mathcal{A}_{m}N_m Q_{m,k}+Q_{m,-k'}{\mathcal{A}_{m}}N_m Q_{m,-k}}{4\pi  B\big(\frac{1}{2},k'\big)}L^{\frac{m-k}2}\tilde{f}_k,\qquad\label{eq:normal_inversion2}
        \end{split}
\end{align}    
where  $ 0<k'\leq m$ is even,  $R_0$ is the Berenstein-Casadio Tarabusi inverse of $N_0$ (\cite{Berenstein1991}), and $\mathcal{P}_0$ is defined in \eqref{eq:projector}.

\item 
\textbf{$m$ odd:} Decompose $f$ as in \eqref{eq:fitt_odd}.
We have the inversion formula 
\begin{align}
        f_s &= -\frac1{32\pi^2} N_0\star \d \ell_1^{-1} (\P_1+\P_{-1})\ell_m \A_m \mathsf{D}\label{eq:normal_inversion_odd}\\
\begin{split}
  L^{\frac{m-k'}2}\tilde{f}_{k'} &=
  \frac{Q_{m,k'}+Q_{m,- k'}}{4\pi B\big(\frac{1}2,k'\big)} \mathcal A_m \mathsf{D} -\frac{Q_{m, k'}+Q_{m,-k'}}{4\pi B\big(\frac{1}2,k'\big)}\mathcal A_mN_m L^{\frac{m-1}{2}} \star \d f_s\\*
  &\qquad -\sum _{\substack{1\leq  k< k'\\ k \text{ odd}}}\frac{Q_{m, k'}\mathcal A_mN_m Q_{m, k}+Q_{m,- k'}\mathcal A_mN_m Q_{m,- k}}{4\pi B\big(\frac{1}2,k'\big)}L^{\frac{m-k}2}\tilde{f}_k ,\end{split}\label{eq:normal_inversion2odd}
\end{align}
where $1\leq k'\leq m$ is odd, and $\mathcal{P}_{\pm 1}$ are defined in \eqref{eq:projector}.
\end{itemize}

\end{theorem}

\begin{remark}
    The equalities  \eqref{eq:normal_inversion}, \eqref{eq:normal_inversion2odd} a priori hold in $x^{-\delta}L^2(\Dm)$.
\end{remark}
\noindent Note that formulas (\ref{eq:normal_inversion}-\ref{eq:normal_inversion2odd}) are triangular, in the sense that one first reconstructs the lowest order components and proceeds iteratively to the higher order ones.
Further, note that in the case $m=1$, using \eqref{eq:normal_inversion_odd} and \eqref{eq:normal_inversion2odd} we can fully invert $N_1$ up to potential tensors.
We mention that inversion formulas for $N_1$  have also been derived by C Robin Graham (\cite{Graham}), in the spirit of \cite{Berenstein1991}.

\paragraph{Outline.} The remainder of the article is organized as follows. Section \ref{sec:AHspaces} covers the proofs of results that hold generally on {simply connected, oriented} AH surfaces, including: preliminaries in Section \ref{sec:OSM}; mapping properties of X-ray transforms (proofs of Proposition \ref{prop:phg_mapping_I}, Lemmas \ref{lem:surjective} and \ref{lm:kernel}) in Section \ref{sec:mappingppties}; the tensor field decompositions (proof of Theorem \ref{thm:gauge_maybe}) in Section \ref{sec:gauge_rep}. 

Section \ref{sec:Poincare} covers proofs of results specific to the Poincar\'e disk, namely:  preliminaries in Section \ref{sec:prelim_poinc}; a study of the transform $I_\perp$ in Section \ref{sec:Iperp}; the data space decomposition (Theorem \ref{thm:dataDecomp}) in Section \ref{sec:dataDecomp}; the range characterization (Theorem \ref{thm:moments}) in Section \ref{sec:rangeCharac}; finally,  reconstruction aspects in Section \ref{sec:recons}.

Appendix \ref{sub:the_0_calculus} covers background material on polyhomogeneous conormal distributions (Section \ref{sec:phg}) and the 0-calculus of Mazzeo-Melrose (Section \ref{sec:the_0_calculus}) and Appendix \ref{app:0Sob} includes the proof of Lemma \ref{lem:H10w}.
\section{Proofs on AH Spaces} \label{sec:AHspaces}

\subsection{Preliminaries}
 \label{sec:OSM}

In what follows $M$ is an $n$-dimensional compact manifold with boundary and $\rho$ is a fixed bdf.

\subsubsection{Symmetric Tensors}\label{ssub:symmetric_tensors}
We discuss certain symmetric tensors which are well adapted to the AH geometry and operators acting on them.
Our normalization conventions will generally follow \cite{Dairbekov2011}.
If  $V$ is a finite dimensional vector space and $m\geq 1$, the symmetrization operator is defined on decomposable elements of $V^{\otimes m}$ as
\begin{align}
        \sigma (v_1\otimes  \ldots \otimes v_m) \coloneqq  \frac{1}{m!}\sum_{\tau\in \mathrm{S}_m} v_{\tau(1)}\otimes \ldots \otimes v_{\tau(m)}, \quad v_j\in V,
    \label{eq:symmetrization}
\end{align}
and extended by linearity. Above, $\mathrm{S}_m$ denotes the $m$-th order permutation group.  The symmetric product  of $v\in S^m(V) $, $w\in S^{m'}(V)$ is defined via $vw=\sigma(v\otimes w)\in S^{m+m'}(V)$; here, $S^{m}(V)$ stands for the $m$-th symmetric power of $V$. Accordingly, for any smooth real or complex vector bundle  $E\to M$ we write   $C^\infty(M;{{S}}^m(E))$ for the space of smooth, symmetric $m$-tensors. The various vector bundles over $M$ will generally be assumed complexified, without explicitly denoting it.

We are especially interested in local sections of the  bundle ${{S}}^m({}^0T^*M)$, which in terms of coordinates $(u^1,\dots, u^n)$  take the form 
\begin{equation}\label{eq:0-tensor}
   f= \sum_{i_k\in \{1,\dots,n\}} f_{i_1\dots i_m}(u)\frac{\d u^{i_1}}{\rho}\otimes \dots \otimes \frac{\d u^{i_m}}{\rho},
\end{equation}
with $f_{i_1\dots i_m}=f_{\tau (i_1)\dots \tau(i_m)}$ for every permutation $\tau\in \mathrm{S}_m$. For each $u\in M$, the metric $g$ defines an inner product on the fiber ${}^0T_u^{*}M$ and we obtain an operator
\begin{equation}
    L:S^m({}^0T_{u}^*M)\to S^{m+2}({}^0T_u^*M),\quad f\mapsto \sigma(g\otimes f).\label{eq:fiberwise}
\end{equation}
Moreover, since $g\in C^{\infty}(M,S^2({}^0T^*M))$, $L$ induces in a natural way  an operator
\begin{align}
    L:C^{\infty}(M,S^m({}^0T^*M))\to C^{\infty}(M,S^{m+2}({}^0T^*M)).\label{eq:L}
\end{align}
It will be clear from context whether $L$ refers to \eqref{eq:fiberwise} or \eqref{eq:L}. Since  $\bar{g}=\rho^{2}g$ extends to a smooth metric on $M$ and $g^{ij}=\rho^{2}\bar{g}^{ij}$, it follows from the local expression \eqref{eq:0-tensor} that \eqref{eq:L} also preserves various  lower regularity spaces, such as  $\Aphg^{E}$, $C^r$, $\rho^{\delta}L^2$ etc.  

Since $g$  is  a smooth fiber metric on ${}^0TM\to M$, it induces  a smooth Hermitian inner product on (the complexification of) any tensor power of it. 
With respect to this Hermitian inner product, denoted   by $\langle \cdot,\cdot\rangle_g$, the inner product of $f,$ $f'\in S^m({}^0 T^*_{u}M)$  takes the form\footnote{Here $\bar{\cdot}$ over $g$ stands for the compactified metric and over $f$ for conjugation. We hope that this will not cause confusion.}
\begin{equation}\label{eq:metric_on_fibers}
    \langle f,f'\rangle _g=\sum _{i_k ,j_\ell\in  \{1,\dots,n\}}\overline{g}^{i_1j_1}\dots \overline{g}^{i_mj_m}f_{i_1\dots i_m}\overline{f'_{j_1\dots j_m}},\quad  {g}=\bar{g}_{ij}\frac{d u^{i}}{\rho} \frac{d u^{j}}{\rho}.
\end{equation}
The $\langle\cdot,\cdot \rangle_g$-adjoint of the map \eqref{eq:fiberwise} is the $g$-trace
\begin{equation}
    \tr:S^{m+2}({}^0T^*_uM))\to S^m({}^0T^*_uM)
\end{equation}
with respect to any two indices of its argument.
It induces a map $C^{\infty}(M,S^{m+2}({}^0T^*M))\to C^{\infty}(M,S^m({}^0T^*M))$, still denoted by $\tr$, which is the $L^2$-adjoint of \eqref{eq:L}.

It is easy to check that the Christoffel symbols of an AH metric are $O(\rho^{-1})$ in any coordinate system near $\partial M$.
Using this, 
one checks that the Levi-Civita connection induces by extension from the interior a map 
\begin{equation}
\begin{aligned}
         \nabla:C^{\infty}(M;{}^{0}T^{*}M)&\to C^\infty(M;{}^{0}T^{*}M\otimes {}^{0}T^{*}M),
\end{aligned}
\end{equation}  which  satisfies the Leibniz rule 
\begin{equation}
    \nabla (f\omega)=\d f \otimes  \omega +f\nabla \omega, \qquad \omega \in C^\infty(M;{}^0T^*M),\quad f\in  C^\infty(M),
\end{equation}
where $\d$ is the exterior derivative viewed as a map $C^\infty(M)\to C^\infty(M;{}^0T^*M) $.
$\nabla$ extends naturally to an operator $C^{\infty}(M;{}^{0}T^*M^{\otimes m})\to C^{\infty}(M;{}^{0}T^*M^{\otimes m+1})$ inducing an operator
\begin{equation}
    \d^{s}:C^{\infty}(M;S^{m}({}^{0}T^*M))\to C^\infty(M;S^{m+1}({}^{0}T^{*}M)),\qquad \d^{s}\coloneqq \sigma \circ \nabla.
\end{equation}
A local coordinate computation shows that 
\begin{equation}
    \d^{s}\in \Diff_0^{1}(M;S^{m}({}^{0}T^*M),S^{m+1}({}^{0}T^*M)),\label{eq:diff0}    
\end{equation}
(see Appendix \ref{sec:the_0_calculus}), so we obtain more generally that for any index set $E$,
\begin{equation}\label{eq:ds_phg}
    \d^{s}:\Aphg^{E}(M;S^{m}({}^{0}T^*M))\to \Aphg^{E}(M;S^{m+1}({}^{0}T^{*}M)).
\end{equation}
The $L^2$-formal adjoint of $-\d^{s}$ is the divergence operator 
\begin{equation}
    \div =\tr \nabla : C^{\infty}(M;S^{m+1}({}^{0}T^*M))\to C^\infty(M;S^{m}({}^{0}T^{*}M)).
\end{equation}

\subsubsection{Measures}\label{ssub:measures_}
We can make sense of weighted $L^2$ spaces of symmetric tensors using the inner product $\langle \cdot,\cdot \rangle_g$: for a measurable section $f$ we have 
\begin{equation}
    f\in \rho^\delta L^2(M;{{S}}^m({}^0T^*M))\overset{\text{def.}}\iff\int _M\rho^{-2\delta}\langle f,{f}\rangle_g\d V_g<\infty, \quad \delta \in \Rm.
\end{equation}
As we will see in the next section, symmetric tensors on ${}^0T^*M$ are in natural correspondence with certain functions in $\OSM$.
The natural Liouville measure on $T^*M$ pulls back via the metric to a measure on $TM$, whose local expression is 
\begin{equation}
    \d\lambda_{TM}=\det g\, \d u^1\dots\d u^{n} \d \tilde{v}^1\dots  \d \tilde{v}^n,
\end{equation}
where we write a vector as $\sum_{j=1}^{n}\tilde{v}^j\partial_{u^j}$.
In the coordinates  on ${}^0TM$, where a 0-vector is written as   $\sum_{j=1}^n v^j\rho \partial_{u^j}$,
 it takes the form 
\begin{equation}\label{eq:DSigma}
    \d\lambda_{{}^0TM}=\sqrt{\det g}\sqrt{\det \bar{g}}\, \d u^1\cdots \d u^{n} \d {v}^1\cdots  \d {v}^n,\quad \bar{g}=\rho^{2}g .
\end{equation}
Thus $\rho^{n}\d\lambda_{{}^0TM}$ is a smooth measure on ${}^0TM$.
Since $\OSM $ is a smooth submanifold of the latter, $\rho^{n}\d\lambda_{{}^0TM}$ induces a smooth measure on $\OSM$.
This implies that $\d\lambda_{{}^0TM}$ induces a smooth measure on $(\OSM)^\circ$ which we denote by $\d \Sigma^{2n-1}$, and which is singular at $\partial (\OSM)$.
Since in our coordinates $\OSM$ is given as the set where $\bar{g}(v,v)=1$, we have a local decomposition of the form $\d \Sigma^{2n-1}=\d V_g(u)\d\mu(v)$, where $\d \mu$ is a smooth measure on the fibers of $\OSM$ (which is the restriction to these fibers of the Lebesgue measure induced by $\bar g$). 

In some places we will use $\bSM$ instead of $\OSM$, equipped with the Liouville measure $|\mu|$.
 Since the metric $g$ induces a measure-preserving diffeomorphism between $(S^{*}M^{\circ},|\mu|)$ and $(SM^{\circ},\d \Sigma^{2n-1})$, we will also use the notation $\d \Sigma^{2n-1}$  for $\bSM$, without creating confusion.
By \cite[Lemma 2.2]{Graham2019}, $\rho \d \Sigma^{2n-1}$ extends to a smooth measure on $\bSM$.

\subsubsection{Symmetric Tensors as Functions on \texorpdfstring{$\OSM$}{0SM}}
There is a natural correspondence between smooth symmetric 0-tensor fields and smooth functions on $\OSM$.
We write 
\begin{equation}\label{eq:pi_m2}
	\pi_m:SM^\circ \to (TM^\circ)^{\otimes m} ,\quad  (u,v)\mapsto (u,v\otimes \dots\otimes v).
\end{equation}
Since the inclusion $\OSM\hookrightarrow {}^0TM$ is smooth, we have:
\begin{lemma}The map \eqref{eq:pi_m2} extends to a smooth map $\pi_m:\OSM \to {}^0TM^{\otimes m} $.\label{lm:pim}
\end{lemma}
\noindent Since ${}^0T^*M$ is the dual bundle of $ {}^0TM$, a smooth section of ${{S}}^m({}^0T^*M)$ is in particular a smooth map $ {}^0TM^{\otimes m} \to \Cm$.
Therefore, with Lemma \ref{lm:pim} we obtain a linear map 
\begin{equation}\label{eq:lm}
    \begin{aligned}
	\ell_m \colon {C}^\infty(M;{{S}}^m({}^0T^*M)) \to C^\infty(\OSM), \qquad f \mapsto \pi_m^* f.
    \end{aligned}	
\end{equation}
Another way to see the mapping property \eqref{eq:lm} is with local expressions.
Using local coordinates \eqref{eq:coordinates0TM} and if $f$ is given locally by \eqref{eq:0-tensor}, we obtain in coordinates \eqref{eq:coordinates0TM}
\begin{equation}\label{eq:lmf}
	\ell_m f(u,v)=  \sum_{i_k\in \{1,\dots,n\}} f_{i_1\dots i_m}(u){v^{i_1}} \cdots  v^{i_m}, \qquad \sum_{i,j=1}^n\overline{g}_{ij}v^iv^j=1.
\end{equation}

\begin{lemma}
The map $\ell_m$ in \eqref{eq:lm} satisfies
\begin{align}
		\ell_m:\dot {C}^\infty(M;{{S}}^m({}^0T^*M)) &\to \dot{C}^\infty(\OSM).
\end{align}
Moreover, for any smooth index set $E$ and $\delta\in \Rm$, it extends to a map
\begin{align}
		\ell_m:\Aphg^E(M;{{S}}^m({}^0T^*M)) &\to  \Aphg^E(\OSM),\\
 		\ell_m: \rho^\delta L^2(M;{{S}}^m({}^0T^*M)) &\to \rho^\delta L^2(\OSM)\label{eq:Hs}.
 \end{align} 
The map \eqref{eq:Hs} is a homeomorphism onto its image.
\end{lemma}

\begin{proof}
The first two statements are clear by \eqref{eq:lmf}. 
Eq.\ \eqref{eq:Hs} essentially follows from the arguments in \cite{Dairbekov2011}, so we only provide a brief outline.
Let $(V^{n}, g_0)$ be an inner product space with unit sphere $S = \{v\in V,\ g_0( v, v) = 1\}$. For any $m\ge  0$, consider the map $\tilde \ell_m \colon S^m (V) \to C^{\infty}(S)$ defined by $\tilde \ell_m f (v) = f(v,\dots,v)$. 
We claim that it is a homeomorphism with respect to the topologies induced by the inner product $\langle\cdot ,\cdot\rangle_{g_0}$ (defined analogously to \eqref{eq:metric_on_fibers}) and the $L^2$ inner product on $C^{\infty }(S)$.
To see this, by \cite[Lemma 2.3]{Dairbekov2011}  decompose $v$ into $g_0$-orthogonal summands as
$ v=\sum_{k=0}^{\lfloor m/2\rfloor} L^{k}\tilde{v}_{m-2k},$ where $ \mathrm{tr}_{g_0}  \tilde{v}_{m-2k}=0.$
Here $L$ and $\mathrm{tr}_{g_0}=L^*$ denote respectively symmetrized multiplication by $g_0$ and $g_0$-trace, as before.
The images $\tilde \ell_m L^{k}\tilde{v}_{m-2k}=\tilde\ell_{m-2k}\tilde v_{m-2k} $ are mutually $L^2$-orthogonal for different $k$'s as spherical harmonics of different degrees.
Then the claim follows upon noticing that for every $ k$, $ \tilde{\ell}_{m-2k} : S^{m-2k}(V)\cap \ker \mathrm{tr}_{g_0}\to L^2(S)$ is an isometry onto its image up to multiplication by a dimensional constant (by \cite[Lemma 2.4]{Dairbekov2011}), as well as the fact that  $|L^k\tilde v_{m-2k}|_{g_0}^{2}=C_{m,n,k} |\tilde v_{m-2k}|^{2}_{g_0}$ for some   dimensional constant $C_{m,n,k}>0$.
The latter fact is a consequence of \cite[Lemma 2.2]{Dairbekov2011}.

As we saw, $g$ defines a smooth metric (up to $\partial M$)   on the smooth vector bundle $^0T^*M\to M$ and its symmetric tensor products. At each $u\in M$, the unit sphere in $^0T^*_uM$ is $\OSM\big|_u$. Thus the arguments above apply to each $({}^0T_u^*M,g_u)$, $u\in M$, so for each such $u$ and $w\in S^m({}^0T_u^*M)$
\begin{equation}\label{eq:inner_products}
    \frac{1}{C}\int_{{}^0\overline{SM}\big|_{u}} |\ell_m w|^{2}\d \mu(v)\leq |w|_{g}^2\leq C\int_{{}^0\overline{SM}\big|_{u}} |\ell_m w|^{2}\d \mu(v),
\end{equation}
with $C>0$ independent of $u$, $w$.
Since locally $\d \Sigma^{2n-1}=\d V_g \d \mu$, for $f\in \dot{C}^{\infty}(M;{{S}}^m({}^0T^*M))$ replace $w$ with  $\rho^{-\delta }f$ in  \eqref{eq:inner_products} and integrate both sides with respect to $\d V_g$ to obtain by density \eqref{eq:Hs}, as well as the last statement.
\end{proof}

\subsubsection{Geodesic Vector Field and Santal\'o's formula}\label{ssec:Santalo}
An equation relating the natural kernel of $I_m$ with the natural kernel of $I$ is the following:
\begin{equation}
    \label{eq:intertw}
    \ell_{m+1} \circ \d^sv = X\circ \ell_m  v  \quad \text{for all} \quad v\in \Aphg^{E}(M;S^{m}(M,{}^0T^*M)),
\end{equation}
where $E$ is a smooth index set.
Both sides make sense as elements of $\Aphg^{E}(\OSM)$: indeed, for the left hand side this follows from  \eqref{eq:ds_phg} and  \eqref{eq:Hs}. For the right hand side, from \eqref{eq:Hs} and the fact that $X$ extends to a smooth vector field on $\OSM$ that is tangent to $\partial \OSM$ (this follows from a direct computation which we omit,  see \eqref{eq:Xconf} below for the 2-dimensional case).
In the interior the equality is classical (see, e.g. \cite[Sec. 10]{Dairbekov2011}), so it holds everywhere.

We close this section by mentioning Santal\'o's formula, which we need later. One has
\begin{align}
    \int_{\OSM} f \d \Sigma^{2n-1} = \int_{\G} If \ \d\Sigma_\partial, \quad f\in C_c^\infty(SM^\circ).
    \label{eq:Santalo}
\end{align}
This is the content of \cite[Lemma 3.6]{Graham2019}. There, the left hand side is expressed as an integral over the cosphere bundle $S^{*}M^{\circ}$ with the Liouville measure $|\mu|$, but, as explained at the end of Section \ref{ssub:measures_}, we can equivalently write \eqref{eq:Santalo} and often we will use the two points of view interchangeably.
Eq.  \eqref{eq:Santalo} implies that $I$ extends to a bounded operator $L^1(\OSM)\to L^1(\G)$, and \eqref{eq:Santalo} extends to functions in $L^1(\OSM)$ (or ($L^1(\bSM)$). In particular, it holds for functions in $\rho^\alpha  C^\infty_c(\bSM)$ if $\alpha>0$, by the remarks at the end of Section \ref{ssub:measures_}.

\subsection{Mapping properties: proof of Proposition \ref{prop:phg_mapping_I} and Lemmas \ref{lem:surjective} and \ref{lm:kernel}} \label{sec:mappingppties}

We now prove basic mapping properties of $I$ on some classes of polyhomogeneous and weighted $L^2$ spaces.
As it will be used later in the section, we denote 
\begin{align}
    \pi_\G\colon SM^{\circ}\to \G
    \label{eq:piG}
\end{align}
the  smooth map sending $(u,v)\in SM^{\circ}$ to the unique $\gamma\in \G$ such that there exists $t\in \Rm$ satisfying $(u,v) = (\gamma(t),\dot{\gamma}(t))$.
It is occasionally useful to view it as a smooth map $S^*M^\circ \to \partial_+\bSM=\G$ which extends smoothly on $\bSM$ and equals the identity on $\partial_+\bSM$, by \cite[Corollary 2.5]{Graham2019}.

\begin{proof}[Proof of Proposition \ref{prop:phg_mapping_I}]
Assume that $g$ is written in normal form as in \eqref{eq:normal_form}.
We use coordinates $(u^{1},\dots, u^{n})=({\tilde x},y^{2},\dots,y^{n})$ as in Section \ref{sub:the_unit_sphere_and_cosphere_bundles_and_their_compactifications_} ($\tilde x$ is a geodesic bdf) and write an element of $\OSM$ as $({\tilde x},y, a {\tilde x}\partial_{\tilde x}+b^\alpha {\tilde x}\partial_{y^\alpha})$, where $a^{2}+h_{\tilde x,\alpha\beta}b^\alpha b^\beta=1$.
Then ${\tilde x}$ is a bdf for $\OSM$, and in a neighborhood of any point $(u,v)\in \OSM$ near $\partial {}\OSM$, $n-1$ of the $(a,b)$ together with $({\tilde x},y)$ furnish smooth coordinates for $\OSM$.
Also let $({\tilde x},y,\bar{\xi}\frac{\d {\tilde x}}{{\tilde x}}+\eta_\alpha \d y^{\alpha})\in \bSM$.
The correspondence induced by the metric between the interiors of $\bSM$ and $\OSM$ yields
    $a=\bar{\xi}$, $ b^\alpha ={\tilde x} h^{\alpha\beta}_{\tilde x} \eta_\beta$.
Hence given $f\in \Aphg^E (\OSM)$, for every $N\geq 1$ there exist smooth $f_{j\ell}$ and $R_N\in \mathcal{A}^{N}(\OSM)$ (see Appendix \ref{sec:phg}) such that %
\begin{equation}\label{eq:phg_f}
    \begin{aligned}
        f&= \sum _{\substack{(s_j,p_j)\in E\\ \Re s_j\leq N}}\sum _{\ell=0}^{p_j} f_{j\ell}\left(y,
        {a,b}
        \right) \tilde {x}^{s_j}\log^{\ell}(\tilde {x})+R_N(\tilde {x},y,a,b)\\ 
        &= \sum_{\substack{(s_j,p_j)\in E\\ \Re s_j\leq N}}\sum _{\ell=0}^{p_j} f_{j\ell}\left(y,
    {\bar{\xi},{\tilde x}h_{{\tilde x}}^{-1}\eta}
    \right) {\tilde x}^{s_j}\log^{\ell}({\tilde x})+R_N({\tilde x},y, \bar{\xi},{\tilde x}h_{{\tilde x}}^{-1}\eta).
    \end{aligned}    
\end{equation}

We will consider the X-ray transform $If(y_*,\eta_*)$, where $(y_*,\eta_*)\in \partial_+\bSM\eqqcolon\G$, and with $f$ viewed as a function on $\bSM$. Viewed as coordinate functions, $(y_*,\eta_*)=(y,\eta)\big|_{\partial_+\bSM}$, Recall that we radially compactified $\G$ into a manifold with boundary $\overline{\G}$ with bdf $\mu_{\tilde{x}}$. 
It will be convenient to set $\delta \coloneqq |\eta_*|_{h_0}^{-1}$ for large $|\eta_*|_{h_0}$, so that $\delta=\mu_{{\tilde x}}(1+O(\mu_{{\tilde x}}))$ is a bdf for $\partial \overline{\G}$.
See Fig.~\ref{fig:bsm}. 

We first claim that if $\psi\in C_c^\infty(\bSM)$, then $I(\psi f)\in C_c^\infty (\G)$.
Recall from \cite{Graham2019} that the geodesic vector field satisfies $\mathcal X=\tilde x \overline{\mathcal X}$, where $\overline{\mathcal X}$ is smooth and transverse to $\partial \bSM$. Hence for each $(u,\zeta)\in S^*M^\circ$ the respective flows satisfy $\varphi_t(u,\zeta) = \overline \varphi_\tau (u,\zeta)$, where $\frac{\d t}{\d \tau}=\frac{1}{\tilde x \circ \overline\varphi_\tau(u,\zeta)}$. 
The flow of $\overline{\mathcal X}$ is incomplete and extends smoothly to $\bSM$.
For each  $(y_*,\eta_*)\in \partial_+\bSM $ it is defined for $0\leq \tau\leq \tau_+(y_*,\eta_*)<\infty$, where $\tau_+(y_*,\eta_*)$ is smooth and positive. Thus
\begin{equation}
    I(\psi f)(y_*,\eta_*) = \int _0^{\tau_+(y_*,\eta_*)}(\tilde x^{-1}\psi f)\circ \overline\varphi _\tau (y_*,\eta_*)\d \tau.
\end{equation}
Since $\overline\varphi _\tau (y_*,\eta_*)$ is smooth and in particular $\tilde x\circ \overline\varphi _\tau (y_*,\eta_*)= \tau\big (\tau_+(y_*,\eta_*)-\tau\big) h(y_*,\eta_*,\tau)$ with $h$ smooth and positive (by the transversality of $\overline{\mathcal X}$), it is straightforward to check using \eqref{eq:phg_f} that $I(\psi f)\in C^\infty (\bSM) $. Here the requirement $\Re (E)>0$ is needed for integrability.
The compact support property follows from  $\supp(I(\psi f))\subset \pi_\G(\supp \psi f)$.
In particular, it implies that $\delta$ is bounded below by a positive constant on $\supp(I(\psi f))$ and thus the behavior of $f$ when $|\eta|_{h_{\tilde x}}$ is bounded does not contribute to the  behavior of $If(y_*,\eta_*)$ as $\delta\to 0$, i.e., to the behavior of $If$ on ``short geodesics''.
Hence to examine this behavior is enough to assume that $f$ is supported in the set where $|\eta|_{h_{\tilde x}}\gg 1$.

\begin{figure}[ht]\centering
  \includegraphics[scale=.7]{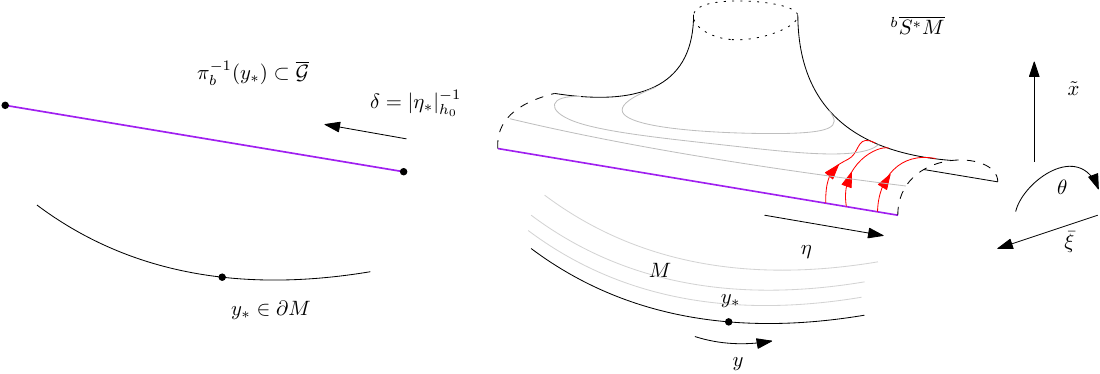}
  
  \captionsetup{width=0.8\textwidth}
  \caption{The case $n=2$. On the left, $\overline{\G}$, whose compactified fibers are parameterized  by a point $y_*\in \partial M$. We write $\pi_b$ for the natural projection $\bSM\to M$. On the right, $\bSM$ (with the $y$ variable suppressed) sitting over $M$. The gray curves are level sets of $\tilde{x}$. Some integral curves of $\mathcal{Y}$ (defined in \eqref{eq:Y}) corresponding to initial conditions with $\delta$ small are shown in red.}\label{fig:bsm} 
\end{figure}

Following \cite[§2.2]{Graham2019}, for $|\eta|_{h_{\tilde x}}\gg 1$ we use smooth coordinates $(\theta,y,\eta)$ for $\ \bSM$, where
\begin{equation}\label{eq:theta_defn}
    \bar{\xi}=\cos(\theta),\qquad {\tilde x}|\eta|_{h_{\tilde x}}=\sin(\theta).
\end{equation}
The vector field 
\begin{align}
	\mathcal{Y}\coloneqq \frac{1}{{\tilde x}|\eta|_{h_{\tilde x}}}\mathcal{X}=\frac{1}{\sin(\theta)}\mathcal{X}
	\label{eq:Y}
\end{align}
is smooth on $\bSM$ in the neighborhood of interest and transversal to its boundary, hence it has an incomplete flow.
Its  integral curves $\tilde{\varphi}_s$ agree with the integral curves $\varphi_t$ of $\mathcal{X}$ up to reparametrization, and for each $(u,\zeta)\in S^*M^{\circ}$ one has $\varphi_t(u,\zeta)=\tilde\varphi_s(u,\zeta)$ with $\d t/\d s=\frac{1}{({\tilde x}|\eta|_{h_{\tilde x}})\circ \tilde{\varphi}_s}=\frac{1}{\sin(\theta)\circ \tilde{\varphi}_s}$.

It follows from the proof of \cite[ Lemma 2.8]{Graham2019} that if $\delta$ is sufficiently small, the integral curve of $\mathcal{Y}$ starting at $(y_*,\eta_*)\in \partial_+\bSM$
is defined for $s\in [0,s_+(y_*,\eta_*)]$, where $s_+(y_*,\eta_*)$ is defined implicitly by the requirement $\tilde\varphi_{s_+(y_*,\eta_*)}=\pi$ and satisfies $s_+(y_*,\eta_*)=\pi +O(\delta)$.
Moreover, for small $\delta$, $\theta$ can be used as a parameter along the flow with $\frac{\d \theta}{\d s}=1+O(\delta)$ and 
\begin{align}\label{eq:coordinates_bSM}
\begin{split}
    {\tilde x}(\theta) &= \delta\sin(\theta)(1+O(\delta)), \quad y(\theta)=y_*+O(\delta) , \\
    \eta(\theta) &= \delta^{-1}(\hat{\eta}_{*}+O(\delta)) ,\quad |\eta(\theta)|_{h_{\tilde x}}=\delta^{-1}(1+O(\delta)),    
\end{split}    
\end{align}
where $\hat\eta_*=\delta \eta_*=\eta_*/|\eta_*|_{h_0}$ and the error terms are smooth in all variables (including $\delta)$.

We apply $I$ to \eqref{eq:phg_f} and analyze the terms of the sum; we want to show an analogous expansion for $If(y_*,\eta_*)=If(y_*,\hat\eta_*,\delta)$ in terms of $\delta$. For $\delta\ll 1$, recalling that $s_j>0$, and writing $\hat\eta_*^{\sharp}=h_0^{-1}\hat\eta_*$, 
\begin{align}
       &I\left(f_{j\ell}\left(y,
    {\bar{\xi},{\tilde x}h_{{\tilde x}}^{-1}\eta}
    \right) {\tilde x}^{s_j}\log^{\ell}({\tilde x})\right)(y_*,\eta_*)\\* 
    &\quad=\int _0^{s_+(y_*,\eta_*)}\!\!\!\!\!f_{j,\ell}\big(y(s),\cos(\theta(s)), {\tilde x}{(s)}h_{{\tilde x}(s)}^{-1}\eta(s)\big)\left(\frac{\sin(\theta(s))}{|\eta(s)|_{h_{\tilde x}}}\right)^{s_j}\log^{\ell}\left(\frac{\sin(\theta(s))}{|\eta(s)|_{h_{\tilde x}}}\right)\frac{\d s}{\sin(\theta(s))}\\
    &\quad=\int _0^{\pi }f_{j,\ell}\big(y(\theta),\cos(\theta), \sin(\theta)(\hat\eta_*^{\sharp}+O(\delta))\big)\left(\frac{\sin(\theta)}{|\eta(\theta)|_{h_{\tilde x}}}\right)^{s_j}\log^{\ell}\left(\frac{\sin(\theta)}{|\eta(\theta)|_{h_{\tilde x}}}\right)\frac{\d \theta}{\sin(\theta)(1+O(\delta))}\\
    &\quad=\delta^{s_j}\int _0^{\pi }f_{j,\ell}\big(y_*+O(\delta),\cos(\theta), \sin(\theta)(\hat\eta_*^{\sharp}+O(\delta))\big)\\* 
    &\hspace{1 in}\times {\sin^{s_j-1}(\theta)}\log^{\ell}\Big({\delta\sin(\theta)(1+O(\delta))}\Big)(1+O(\delta)){\d \theta} \\ 
&\quad=\sum_{m=0}^{\ell}A_{j,m}(y_*,\hat{\eta}_*,\delta)\delta^{s_j}\log^{m}(\delta),\label{eq:Ixlogx}
\end{align}
where
\begin{align}
    A_{j,m}(y_*,\hat{\eta}_*,\delta)=& \binom{\ell}{m}\int _0^{\pi }\left(f_{j,\ell}\big(y_*,\cos(\theta), \sin(\theta)\hat\eta_*^{\sharp}\big)+O(\delta)\right){\sin^{s_j-1}(\theta)}{(1+O(\delta))}\\*
&\hspace{2 in} \times\Big(\log(\sin(\theta))+\log(1+O(\delta))\Big)^{\ell-m}\d \theta,\\
        =&   \binom{\ell}{m}\int _0^{\pi }f_{j,\ell}\big(y_*,\cos   (\theta), \sin(\theta)\hat\eta_*^{\sharp}\big){\sin^{s_j-1}(\theta)}
\log^{\ell-m}(\sin(\theta))\d \theta+O(\delta).
\end{align}

Hence we have taken care of the terms of the expansion in \eqref{eq:phg_f}. We will have shown \eqref{eq:Iphg_mapping} if we can show that $I(\mathcal{A}^N(\OSM))\subset \mathcal{A}^N(\overline{\G})$.
To do so, let $N\geq 1$ and $R_N(\tilde x,y,a,b)\in \mathcal{A}^N(\OSM)$.
As before,
\begin{equation}\label{eq:IRN}
    IR_N(y_*,\hat{\eta}_*,\delta)=\int_0^{\pi} R_N(F(\delta,\theta, y_*,\hat{\eta}_*))\frac{\d \theta}{\sin(\theta)(1+{O}(\delta))},
\end{equation}
where
\begin{equation}
\begin{aligned}
         F(\delta,\theta, y_*,\hat{\eta}_*)= &\big(\delta\sin(\theta)(1+O(\delta)),y_*+O(\delta),\cos(\theta), \sin(\theta)(\hat\eta^{\sharp}_*+O(\delta)\big)   \big)\in \OSM.\\
\end{aligned}
\end{equation}
We claim that if ${r}\in \mathcal{A}^{N}(\OSM)$ for some $N\in \Rm$, then for any $j\geq 0$ and multiindex $J$ 
\begin{equation}\label{eq:conormality}
    (\delta \partial_\delta )^{j}\partial_{(y_*,\hat{\eta}_*)}^{J}\Big( {r}(F(\delta,\theta, y_*,\hat{\eta}_*))\Big)=\sum _{k=1}^{M}F_k (\delta,\theta, y_*,\hat{\eta}_*){r}_k(F(\delta,\theta, y_*,\hat{\eta}_*)),
\end{equation}
for some $M\geq 0$, where $F_k$ are smooth in all variables and ${r}_k\in \mathcal{A}^{N}(\OSM)$.
When we write $\partial_{(y_*,\hat{\eta}_*)}^{J}$ we abuse notation somewhat to indicate that we choose $n-2$ of $\hat{\eta}_{*j}=\frac{\eta_j}{|\eta|_{h_0}}$ such that, together with $y^\alpha$ and $\delta$, they form a coordinate system near a point in $\partial \overline \G$.
By induction it is easy to see that it suffices to show \eqref{eq:conormality} for $j+|J|=1$.
For $j=1,$ we have  by the chain rule
\begin{align}
     \delta \partial_\delta \left( {r}(F(\delta,\theta, y_*,\hat{\eta}_*))\right) 
        &=({\tilde x} \partial_{\tilde x} {r}) \big|_{F}(1+O(\delta))+\sum_{\alpha=2}^{n}\partial_{y^\alpha }{r}\big|_{F}  G_\alpha +\sum_{{\alpha}=2}^{n}\partial_{b^{\alpha}}  {r} \big|_{F}\tilde{G}_{\alpha},
\end{align}
where $G_\alpha$, $\tilde{G}_{ \alpha}$ are smooth functions on $\bSM$, down to $\delta=0$.
Now the claim follows for $J=0$ and $j=1$ by the observation that ${\tilde x}\partial_{\tilde x} {r}, \partial_{y}{r}, \partial_b {r}\in \mathcal{A}^{N}(\OSM)$.
The argument for $|J|=1$ is similar.
Using \eqref{eq:conormality}, we see that if $N\geq 1$
\begin{align}
    |(\delta \partial_\delta )^{j}\partial_{(y_*,\hat{\eta}_*)}^{J}IR_N|\overset{\eqref{eq:IRN}}{=}\left|\int _0^\pi (\delta \partial_\delta )^{j}\partial_{(y_*,\hat{\eta}_*)}^{J}
    \left(\frac{R_N(F(\delta,\theta, y_*,\hat{\eta}_*))}{1+O(\delta)}\right)\frac{\d\theta}{\sin(\theta)}\right|\\ 
    \overset{\eqref{eq:conormality}}{\leq} C\int_0^{\pi } {\tilde x}^{N}\circ F(\delta,\theta,y_*,\hat \eta_*)\frac{\d \theta}{\sin(\theta)}\overset{\eqref{eq:coordinates_bSM}}{\leq} C\int_0^\pi \delta^{N}\sin^{N-1}(\theta)\d \theta\leq C\delta^{N}.
\end{align}
This shows that $I(\mathcal{A}^N(\OSM))\subset \mathcal{A}^N(\overline{\G})$, finishing the proof of \eqref{eq:Iphg_mapping}.

To show \eqref{eq:Ixalpha}, let $\rho $ be a bdf;  the proof above applies with $\tilde x$  the geodesic bdf corresponding to $\rho^{2}g\big|_{T\partial M}$. Moreover, $\mu_\rho=\mu_{\tilde{x}}$ (by definition), so  $\delta=\mu_{\rho}+O(\mu_{\rho}^2) $.
 Using \eqref{eq:Ixlogx} with $s_j=\alpha>0$, $\ell=0$, and $f_{j\ell}=1$, we find
\begin{equation}
    I\rho^{\alpha}(y_*,\hat\eta_*,\mu_\rho)=\mu_\rho^{\alpha}\int _0^\pi \sin^{\alpha-1}(\theta)G(\alpha,\theta,y_*,\hat\eta_*,\mu_\rho)\d \theta,
\end{equation}
where $G$ is smooth in all variables, down to $\alpha=0$, with Taylor expansion in $\mu_{\rho}$ at $0$ of the form 
\begin{equation}
    G(\alpha,\theta,y_*,\hat\eta_*,\mu_\rho)=1+ O(\mu_\rho)\quad \text{as }\mu_\rho\to 0.
\end{equation}
For every fixed $\alpha>0$, \eqref{eq:Ixalpha} holds with $R_\alpha=\mu_\rho^{-1}
\int_0^\pi\sin^{\alpha-1}(\theta) (G-1)\d \theta$.
For the last claim,  if $\mu_\rho$ and $\alpha$ are sufficiently small we have
\begin{align}
        \mu_\rho^{-\alpha}\left|I\rho^{\alpha}-\mu_\rho^{\alpha}B(\alpha/2,1/2)\right|\leq \int _0^\pi \sin^{\alpha-1}(\theta)|
        G-1|\d\theta
        \leq \frac12\int _0^\pi \sin^{\alpha-1}(\theta)\d\theta\leq \frac12B(\alpha/2,1/2).
\end{align}
\end{proof}

We now turn to Lemma \ref{lem:surjective}.  Note the following pointwise inequality as a direct consequence of Cauchy-Schwarz' inequality: 
\begin{align}
    |I (fg)|^{2}(\gamma) \le I (|f|^2)(\gamma) \cdot I (|g|^2) (\gamma), \qquad \gamma\in \G.
    \label{eq:CS}
\end{align}
As an immediate consequence of \eqref{eq:Ixalpha}, we have that for any bdf $\rho$ on $M$ giving rise to a bdf $\mu_{\rho}$ on $\overline{\G}$, and for every $\alpha>0$, there exist constants $C_{\alpha,1},C_{\alpha,2}>0$ such that 
\begin{align}
    C_{1,\alpha} \mu_{\rho}^\alpha(\gamma) \le I \rho^\alpha (\gamma) \le C_{2,\alpha} \mu_{\rho}^\alpha(\gamma), \qquad \gamma\in \G. 
    \label{eq:IxAH}
\end{align}
\noindent For part \ref{item_surj_3} we will also need the following preparatory lemma.  As usual, $\tilde x$ is a geodesic bdf.

\begin{lemma}\label{lem:w_fct}Let $(M^\circ,g)$ be a non-trapping AH manifold and fix $\alpha>0$. There exists $w\in \tilde x^\alpha {C}^\infty(\bSM)$ which is bounded, everywhere positive in $S^*M^\circ$,  and satisfies $I(w^2)\in L^\infty (\G)$ and $1/(I(w^2))\in L^\infty(\G)$.
\end{lemma}
\begin{proof}
Consider a geodesic bdf  $\tilde x$, and fix coordinates $(\tilde x, y, \bar \xi, \eta )$ in a neighborhood of  $\partial {}^bT^{*}M$ in ${}^bT^{*}M$ where $\tilde x<2\epsilon$ for some $\epsilon >0$, so that $\bSM$ is represented as in \eqref{eq:bSM_coords}. 
Then, for $\tilde x\leq \epsilon$ we let $w_0\coloneqq  \tilde x \sqrt{1+|\eta|^2_{h_{\tilde x}}}$
and we extend it smoothly on $\bSM$, requiring that $w_0\geq \epsilon $ for $\tilde x\geq \epsilon $.
Then we let $w\coloneqq w_0^\alpha$.
We have $w\in\tilde x^\alpha C^\infty (\bSM)$, 
 $w>0 $ in $S^*M^{\circ}$ and 
 $w =0 $ on $\partial \bSM$.
 Moreover, $w$ is bounded: indeed, this is clear if $\tilde x\geq \epsilon$, and otherwise we observe that $\tilde x^{2}|\eta|^2_{h_{\tilde x}}\leq 1$ on $\bSM$.
 Since $w\in \tilde x ^\alpha{C}^\infty(\bSM)$,   $I(w^{2})\in C^{\infty}(\G)$ (this can be justified as in the proof of Proposition \ref{prop:phg_mapping_I}), in particular it is bounded over any compact set.
Since $w>0$ on $S^*M^\circ$, $I(w^2)>0$ everywhere and thus $I(w^2)$ is bounded below by a positive constant on any compact set.
Following the notation in the proof of Proposition~\ref{prop:phg_mapping_I}, for $(y_*,\eta_*)=(y_*,\hat \eta_*,\delta)\in \G$ with $\delta \ll 1$ we have, as in the computation \eqref{eq:Ixlogx} and using \eqref{eq:coordinates_bSM},
\begin{align}
    &I(w^2)(y_*,\hat \eta_*,\delta)=\int_0^{\pi }\big(\sin (\theta)(1+O(\delta))\big)^{2\alpha}\frac{\d \theta}{\sin(\theta)(1+O(\delta))}\quad =B(\alpha,1/2)+O(\delta).
\end{align}
This shows that $I(w^{2})$ is bounded above and below by a positive constant in a neighborhood of $\partial \overline {\G}$, hence everywhere.
\end{proof}

\begin{proof}[Proof of Lemma \ref{lem:surjective}]
    Proof of \ref{item_surj_1}.
    To prove boundedness, we compute, for fixed $\delta>0$ and $f\in C_c^\infty(SM^\circ)$:
    \begin{align*}
	\int_{\G} |I f|^2 \mu_{\rho}^{-2\delta} \d\Sigma_\partial &\stackrel{\eqref{eq:CS}}{\le} \int_{\G} I (\rho^{2\delta}) \cdot I (|f|^2 \rho^{-2\delta}) \mu_{\rho}^{-2\delta}\ \d\Sigma_\partial \\
	&\stackrel{\eqref{eq:IxAH}}{\le} C_{2,2\delta} \int_{\G} I (|f|^2 \rho^{-2\delta})\ \d\Sigma_\partial \stackrel{\eqref{eq:Santalo}}{=} C_{2,2\delta} \int_{SM^{\circ}} |f|^2 \rho^{-2\delta} \d\Sigma^{2n-1},
    \end{align*}
    hence \eqref{eq:IHgamma} is bounded with operator norm no greater than $\sqrt{C_{2,2\delta}}$ .
    
    To prove surjectivity, notice that we have  
    \begin{align}
	I \left( \rho^\alpha \frac{h}{I \rho^\alpha}\circ \pi_\G \right) = h, \qquad \alpha>0, \qquad h \in C_c^\infty (\G).
	\label{eq:rightinv}
    \end{align}
    Now fix $\delta>0$, and if $h\in L^2(\G, \mu^{-2\delta}\ \d\Sigma_\partial)$, we compute, fixing $\alpha>\delta$,
    \begin{equation}
	\int_{SM^{\circ}} \left| \rho^\alpha \frac{h}{I \rho^\alpha}\circ \pi_\G \right|^2 \rho^{-2\delta} \d\Sigma^{2n-1} \stackrel{\eqref{eq:Santalo}}{=} \int_{\G} \frac{|h|^2}{|I \rho^\alpha|^2} I (\rho^{2(\alpha-\delta)})\d\Sigma_\partial \stackrel{\eqref{eq:IxAH}}{\le} \frac{C_{2,2(\alpha-\delta)}}{C_{1,\alpha}^2} \int_{\G} |h|^2 \mu_{\rho}^{-2\delta}\ \d \Sigma_\partial.\label{eq:L2SMbd}
    \end{equation}
    Here the use of \eqref{eq:Santalo} is justified by the remarks at the end of Section \ref{ssec:Santalo}.
    Hence we have shown the existence of a bounded right-inverse for the operator \eqref{eq:IHgamma}.

    Proof of \ref{item_surj_2}. To show  unboundedness, we will show that for some well-chosen family $\{f_\alpha\}_{\alpha>0}$, the ratio $\|If_\alpha\|^2_{L^2(\G)}/\|f_\alpha\|^2_{L^2(SM^\circ)}$ can be made arbitrarily large. As a consequence of the statement below \eqref{eq:Ixalpha}, there exists $\mu'>0$ and $\alpha'>0$ such that for every $\alpha\in (0,\alpha')$, 
    \begin{align}
	\frac{1}{2} B\left( \alpha/2, 1/2 \right) \mu_{\rho}^\alpha(\gamma) \le I \rho^\alpha (\gamma) \le \frac{3}{2} B\left( \alpha/2, 1/2 \right) \mu_{\rho}^\alpha(\gamma) , \quad \gamma \in \G\cap \mu_{\rho}^{-1}( (0,\mu')).
	\label{eq:IxAH2}
    \end{align}
    Now fix $0\not\equiv\phi\in C_c^\infty(\G)$  with support contained in a  nonempty compact set  $K\subset \G\cap \mu_\rho^{-1} ( (0,\mu'))$, and set $f_\alpha \coloneqq  \rho^\alpha(\phi\circ\pi_\G)$.
For such an $f_\alpha$, 
    \begin{align*}
	\|If_\alpha\|^2_{L^2(\G)} = \int_{K} |\phi|^2 |I \rho^\alpha|^2 \d\Sigma_\partial \stackrel{\eqref{eq:IxAH2}}{\ge} \frac{1}{4} B(\alpha/2,1/2)^2 \int_{K} |\phi|^2 \mu_\rho^{2\alpha}\d\Sigma_\partial.
    \end{align*}
    On the other hand,
    \begin{align*}
	\|f_\alpha\|^2_{L^2(SM^\circ)} = \int_{SM^\circ} \rho^{2\alpha}|\phi\circ\pi_\G|^2\d\Sigma^3 &\stackrel{\eqref{eq:Santalo}}{=} \int_{K} I(\rho^{2\alpha}(|\phi|^2\circ\pi_\G))\d\Sigma_\partial \\
& \stackrel{\eqref{eq:IxAH2}}{\le} \frac{3}{2} B(\alpha,1/2)\int_{K}|\phi|^2 \mu_\rho^{2\alpha}\d\Sigma_\partial.
    \end{align*}
    Thus, 
    \begin{align*}
	\frac{\|I f_\alpha\|^2_{L^2(\G)}}{\|f_\alpha\|^2_{L^2(SM^\circ)}} \ge \frac{1}{6} \frac{B(\alpha/2,1/2)^2}{B(\alpha,1/2)} = \frac{2}{3} \frac{1}{\alpha}(1+o(1)) \quad \text{as } \alpha\to 0,
    \end{align*}
    hence the result.

Proof of  \ref{item_surj_3}.
Let $w$ be as in Lemma \ref{lem:w_fct} for a fixed $\alpha>0$: it can be viewed as a smooth function on either $S^*M^\circ$ or $SM^\circ$, by means of the musical isomorphisms.
Since $w$ is bounded and vanishes at $\partial \bSM$, $wL^2(SM^\circ) \subsetneq L^2(SM^\circ)$.
The boundedness statement follows from the estimate below, which is true for any $f\in C_c^\infty(SM^\circ)$ and does not require $1/(I[w^2])\in L^\infty(\G)$:
    \begin{align*}
        \int_\G |If|^2 \d\Sigma_\partial \stackrel{\text{C.-S.}}{\le} \int_\G I[(f/w)^2] I [w^2] \d\Sigma^{2n-1} &\le \|I[w^2]\|_\infty \int_\G I[(f/w)^2] \d\Sigma^{2n-1} \\
        &\!\!\!\stackrel{\eqref{eq:Santalo}}{\le} \|I[w^2]\|_\infty  \|f/w\|^2_{L^2(SM^\circ)}.
    \end{align*}

    For surjectivity we work with $\bSM$, slightly modifying the argument in the proof of part \ref{item_surj_1}.
    We now require that $1/(Iw^2)\in L^\infty(\G)$, and define the operator $B\colon C_c^\infty(\G) \to \tilde x^{2\alpha}{C}_c^\infty({}^b \overline{S^*M})$ by $Bh := w^2 \cdot\left(\frac{h}{I[w^2]}\circ \pi_G\right)$. We first note that $B$ extends by density to a bounded operator $B\colon L^2(\G)\to w L^2(S^*M^\circ)$, as follows from the estimate
    \begin{align*}
        \|w^{-1}Bh\|^2_{L^2(S^*M^\circ)} = \int_{S^*M^\circ} w^2 \left( \frac{h}{I[w^2]}\circ\pi_\G\right)^2 \d\Sigma^{2n-1} \stackrel{\eqref{eq:Santalo}}{=} \int_{\G} \frac{h^2}{I[w^2]} \d\Sigma_\partial \le \left\|\frac{1}{I [w^2]}\right\|_{\infty} \|h\|^2_{L^2(\G)},
    \end{align*}
    valid for every $h\in C_c^\infty(\G)$. Moreover, the identity $IBh = h$, originally true for all $h\in C_c^\infty(\G)$, extends to $h\in L^2(\G)$ by density. Hence we obtain surjectivity of $I\colon w L^2(S^*M^\circ)\to L^2(\G)$, which is equivalent to surjectivity of $I\colon w L^2(SM^\circ)\to L^2(\G)$.
\end{proof}

We end this section with the proof of Lemma \ref{lm:kernel} on the kernel of $I_m$. 

\begin{proof}[Proof of Lemma \ref{lm:kernel}]
First, for the polyhomogeneous case, let $q\in  \Aphg^{E}(M;S^{m-1}({}^{0}T^*M))$ with ${\Re E>0} $ and $f=\d^s q \in   \Aphg^{E}(M;S^{m}({}^{0}T^*M))$.
In coordinates $u^j$ near $\partial M$, $f$ takes the form \eqref{eq:0-tensor}, where $f_{i_1\dots i_m}\in \Aphg^{E}$.
Thus for any unit speed geodesic $\gamma(t)$ within the coordinate neighborhood
\begin{equation}
    I_m f (\gamma) =\sum_{i_j\in \{1,\dots,n\}}\int_{-\infty}^{\infty}f_{i_1\dots i_m}(\gamma(t))\frac{\dot\gamma^{i_1}(t)}{\rho(\gamma(t))}\dots \frac{\dot\gamma^{i_m}(t)}{\rho(\gamma(t))}\d t.
\end{equation}
Since $\gamma$ is unit speed, $\dot \gamma^{j}(t)/\rho(t)\in L^\infty(\Rm)$ for all $j$.
Then we know e.g. from \cite{Mazzeo1986} or \cite[Lemma 2.3]{Graham2019} that $\rho\circ \gamma(t)=O(e^{-|t|})$ as $t\to \pm \infty$. Moreover,
since $\Re(E)>0$, there exists $\epsilon>0 $ such that $\rho^{-\epsilon}|f_{i_1\dots i_m}|$ is bounded. Putting these together, the integrand is estimated by $Ce^{-\epsilon|t|}\in L^1(\Rm)$, so by dominated convergence  we have, for any unit speed $\gamma$,
\begin{equation}
\begin{aligned}
         I_m \d^{s}q=\lim _{T\to\infty}\int_{-T}^{T}\d ^{s} q(\gamma(t),\dot\gamma(t)^{\otimes m})\ \d t=\lim _{T\to\infty}\int_{-T}^{T}\frac{\d}{\d \tau}q(\gamma(\tau),\dot\gamma(\tau)^{\otimes (m-1)})\big|_{\tau = t}\ \d t\quad \\
        =\lim_{T\to \infty}(q(\gamma(T),\dot\gamma(T)^{\otimes (m-1)})-q(\gamma(-T),\dot\gamma(-T)^{\otimes (m-1)}))=0.
\end{aligned}
\end{equation}
The second equality follows from e.g., \cite[Eq.\ 3.3.17]{Sharafudtinov1994} and the last due to the fact that $q(\gamma(\pm T),\dot\gamma(\pm T)^{\otimes (m-1)})\leq C e^{-\epsilon T}$ as before. 
The general case follows upon covering $\gamma$ with finitely many charts.

If  $q\in \rho^{\delta }H_0^{s}(M;S^{m}({}^0T^*M))$ with $s\geq 1$ and $\delta>0$, 
the result follows by density of $ C^{\infty}_c(M^{\circ })$ in $\rho^{\delta }H_0^{s}(M)$, $\eqref{eq:diff0}$, and the continuity property \eqref{eq:IHgamma}.
\end{proof}

\subsection{Gauge representatives: proofs of Theorem \ref{thm:gauge_maybe} and Proposition \ref{prop:inj_on_repr_maybe2}} \label{sec:gauge_rep}

\subsubsection{Vertical Fourier Analysis in dimension 2} \label{sec:dim2}

Much of the material in Section \ref{sec:OSM} simplifies in  dimension 2  using isothermal coordinates. 
As mentioned in \cite[p. 2871]{Graham2019}, if $(M^{\circ},g)$ is a non-trapping, oriented  AH surface, it is simply connected.
By the existence of global isothermal coordinates for $(M,\rho ^{2}g)$, where $\rho $ is a bdf (see \cite[Theorem 3.4.16]{Paternain2023}) one sees that, upon renaming $\rho$ if necessary, the metric takes the form
 \begin{equation}\label{eq:isothermal}
    g=\frac{ (\d u^1)^2+(\d u^2)^2}{\rho^2}
\end{equation}
in terms of a coordinate chart identifying $M$ with $\Dm$.
In terms of \eqref{eq:isothermal}, a $0$-vector $v = v^1\rho \partial _{u^1}+v^2 \rho \partial_{u^2} \in {}^0 TM$ belongs to $\OSM$ (as defined in \eqref{eq:OSM}) if and only if $(v^1)^2+(v^2)^2=1$. 
This gives rise to the following parametrization of $\OSM$ 
\begin{equation}\label{eq:osm_param}
    \Dm\times (\Rm/2\pi\Zm) \ni (u^1,u^2,\theta)\mapsto \big((u^1,u^2),\cos(\theta) \rho\partial_{u^1}+\sin(\theta) \rho\partial_{u^2}\big)
    =\big(z,2\rho\Re (e^{i\theta}\partial_{z})\big), %
\end{equation} %
where $z\coloneqq u^1+iu^2$ is a complex coordinate with respect to the complex structure determined by $g$ and the orientation.
In terms of this parametrization, $\d \Sigma^3= \rho^{-2} \, \d u^1\d u^2 \d \theta$.

    Recall the canonical framing $\{X,X_\perp, V\}$ of $T(SM^{\circ})$, where $X$ is the generator of the geodesic flow, $V$ is the generator of the circle action on the tangent fibers uniquely determined by the metric and orientation, and $X_\perp\coloneqq [X,V]$. In the AH situation, we have the following.

\begin{proposition}
    If $(M^\circ ,g)$ is an oriented AH surface, then the canonical framing $\{X,X_\perp,V\}$ of $T(SM^\circ)$ extends to be smooth on $\OSM$ and tangent to the fibers of the natural projection $ \pi:\OSM\to M$ at $\partial\, \OSM$.
\end{proposition}
\begin{proof} In terms of \eqref{eq:osm_param}, the fiberwise counterclockwise rotation by angle $t$, $\rho_t\colon SM^\circ \to SM^\circ$ is nothing but the map $(u^1,u^2,\theta)\mapsto (u^1,u^2,\theta+t)$, with generator $V$ given by $V = \partial_\theta$.  
    This shows that $V$ extends to a smooth vector field on ${}^0\overline{SM}$ and tangential to its boundary. As in \cite[Lemma 3.5.6]{Paternain2023}, the vector fields $X,X_\perp$ take expressions
    \begin{align}
        X &= \cos(\theta)\rho\partial_{u^1}+\sin(\theta)\rho\partial_{u^2}+(\partial_{u^1}\rho\sin(\theta)-\partial_{u^2}\rho \cos(\theta))\partial_\theta, \label{eq:Xconf}\\*
        X_\perp &= \sin(\theta)\rho\partial_{u^1}-\cos(\theta)\rho\partial_{u^2}-(\partial_{u^1}\rho \cos(\theta)+\partial_{u^2}\rho \sin(\theta))\partial_\theta. \label{eq:Xperpconf}
    \end{align}
    When $\rho =0$, both vector fields are annihilated by $\d \pi:T\OSM\to TM$, which shows the claim.
\end{proof}

\begin{remark}\label{rem:edge}
    Denoting $\pi_\partial\coloneqq \pi\big|_{\partial\,\OSM}$, the
    space $\OSM$ is a manifold with boundary whose boundary has the structure of a fiber bundle $\partial\, \OSM\overset{\pi_\partial}{\longrightarrow} \partial M$, in other words an {\em edge space} in the sense of \cite{Mazzeo1991}. Recall that on an edge space, the space of edge vector fields consists of those smooth vector fields which are tangent to the fibers of the fiber bundle at the boundary. As an immediate corollary of the previous proposition we obtain that the vector fields $X,$ $X_\perp$, $V$ generate over $C^\infty(\OSM)$ the space of edge vector fields on $\OSM$. The Sasaki metric on $SM^\circ$ (for which $\{X,X_\perp,V\}$ is by definition orthonormal) therefore extends to an edge metric on $\OSM$.    
\end{remark}

Below we use the shorthand $\mathcal{F}$ to denote a fixed regularity space on either $M$ or $\OSM$, with $\mathcal{F}\in \{C^r,\dot{C}^\infty,\Aphg^E,\rho^\delta L^2\}$. We consider the Fourier decomposition
\begin{equation}
    \mathcal{F}(\OSM)=\bigoplus_{k\in \Zm}\Omega_k^{\mathcal{F}}, \qquad \text{where} \qquad \Omega_k^{\mathcal{F}}\coloneqq  \{ f\in \mathcal{F}(\OSM)\colon Vf = ik f \},
\end{equation}
with projectors 
\begin{equation}\label{eq:projector}
  \P_k:\mathcal{F}(\OSD)\to \Omega^{\mathcal{F}}_{k},\qquad  \P_kf(z,\theta)\coloneqq \frac{ e^{ik\theta}}{2\pi}\int_0^{2\pi}e^{-ik\phi}f(z,\phi)\d \phi, \quad k\in \Zm.
\end{equation}
Note that $\P_k^{2}=\P_k$ and $\P_{k'}\P_k=0$ for $k'\neq k$. A typical element $\hat{f}_k \in \Omega_k^{\mathcal{F}}$ in terms of a chart \eqref{eq:osm_param} will often be written as 
\begin{equation}\label{eq:hat_notation}
    \hat{f}_k \coloneqq e^{ik\theta} f_k, \quad f_k\in \F(M).
\end{equation}
Now fix $m\in \Nm_0$. Observing that a  frame of ${{S}}^m({}^0T^*M)$ is given by $\{\rho^{-m} {\sigma( \d z^{\otimes k} \otimes \d\bar{z}^{\otimes (m-k)} )}\}_{0\le k\le m}$ in the chart \eqref{eq:osm_param}, and that $\ell_m (\rho^{-m} \sigma( \d z^{\otimes k} \otimes \d\bar{z}^{\otimes (m-k)}) = e^{-i(m-2k)\theta}$, we deduce that the map \eqref{eq:lm} is an isomorphism
\begin{align}\label{eq:isomorphism_lm}
    \ell_m:\mathcal{F}(M;{{S}}^m({}^0T^*M))&\to \bigoplus_{j=0}^{m}\Omega_{-m+2j}^{\mathcal{F}}.
\end{align}
As in the compact case, from \eqref{eq:Xconf}-\eqref{eq:Xperpconf} it follows that if $\mathcal{F}\in \{\dot{C}^\infty,\Aphg^E\}$,
\begin{equation}
	X, X_\perp: \Omega_k^{\mathcal{F}}\to \Omega_{k-1}^{\mathcal{F}}\oplus \Omega_{k+1}^{\mathcal{F}}.
\end{equation}
Similar statements hold in finite regularity spaces. Moreover, defining the {\em Guillemin-Kazhdan operators} 
\begin{align}
\eta_\pm \coloneqq  \frac{1}{2}(X\pm iX_\perp),
\label{eq:GK}
\end{align}
one has $\eta_\pm:\Omega_k^{\mathcal{F}}\to \Omega_{k\pm 1}^{\mathcal{F}}$ for every $k\in \Zm$, $X=\eta_++\eta_-$,  and $X_\perp = -i \eta_++i\eta_-$.
Finally, we briefly discuss the following characterization of spaces $\F(M; \Stt^m ({}^0 T^* M))$ defined in \eqref{eq:ttspaces}. Following the discussion in, e.g., \cite{Paternain2015}, for $m\ge 2$ and $\F$ one of the regularity classes above, defining $\Theta_m^\F = \F(M; S^m ({}^0 T^* M)) \cap \ker \tr$ of trace free symmetric $m$-tensors and setting $\Theta_1^\F = \F(M; {}^0 T^* M)$, then the map 
\begin{align}
    \ell_m \colon \Theta_m^\F \longrightarrow \Omega_{-m}^\F \oplus \Omega_m^\F
    \label{eq:lmtraceless}
\end{align}
is an isomorphism. Via these isomorphisms, the divergence $\delta\colon \Theta^\F_m\to \Theta^\F_{m-1}$ acts by\footnote{This can be extracted from \cite{Paternain2015} as follows: in \cite[Appendix B]{Paternain2015}, the display $X_- (e_m + e_{-m}) = \eta_- e_m + \eta_+ e_{-m}$ appears, where $e_{\pm m}$ plays the role of $\hat{f}_{\pm m}$ above, and where $X_-$ is an operator satisfying in \cite[Section 3]{Paternain2015} the relation $X_- = \frac{1}{2} \ell_{m-1} \delta \ell_m^{-1}$ on $\Theta_m$ in the special case where $n=2$. } 
\begin{align*}
    \ell_{m-1} \delta \ell_m^{-1} (\hat{f}_{-m} + \hat{f}_m) = 2 (\eta_+ \hat{f}_{-m} + \eta_- \hat{f}_m), \qquad (\hat{f}_{-m}, \hat{f}_m)\in \Omega_{-m}^\F\times \Omega_m^\F.
\end{align*}
As a result, since by definition $\F(M; \Stt^m ({}^0 T^* M)) = \Theta_m^\F \cap \ker \delta$, we see that 
\begin{align}
    \ell_m \colon \F(M; \Stt^m ({}^0 T^* M)) \longrightarrow (\Omega_{-m}^\F \cap \ker \eta_+) \oplus (\Omega_{m}^\F \cap \ker \eta_-)
    \label{eq:ttdecom}
\end{align}
is an isomorphism for all $m\ge 1$. The same statement also holds with $\mathcal{F}=\rho^{\delta}L^2$, where the differential operators act a priori in a weak sense; as is known however, the elements on both sides are smooth over $M^{\circ}$. We also note that, on $\dot C^\infty(M)$, $\Aphg^E(M)$ or $\rho^\delta H^s_0(M)$,
\begin{align}
    \ell_1 \circ \d = X\circ \ell_0, \qquad \ell_1 \circ \star\d = - X_\perp \circ \ell_0.
    \label{eq:XXperp_intertwiners}
\end{align}

\subsubsection{Elliptic decomposition for the Guillemin-Kazhdan operators} \label{sec:GK}

The main result of this section is an elliptic decomposition for the $\eta_\pm$ operators, adapted to AH geometry. We will extensively use terminology and results from Appendix \ref{sec:the_0_calculus}. We assume throughout that $M$ is simply connected and oriented, so that there exists a global isothermal  chart. In such a chart, it follows from \eqref{eq:Xconf} that $\eta_\pm$ defined in \eqref{eq:GK} take the expression
\begin{equation}
    \eta_+=e^{i\theta}(\rho \partial_{{z}}-i\partial_z\rho\, \partial _\theta), \qquad 
    \eta_-=e^{-i\theta}(\rho \partial_{\bar{z}}+i\partial_{\bar{z}}\rho\, \partial _\theta).
\end{equation}
Therefore,
\begin{equation}\label{eq:eta_pm_isoth}
	\begin{aligned}
\eta_+(f(z)e^{ik\theta})=e^{i(k+1)\theta}(\rho \partial_{{z}}+k\partial_{{z}}\rho)f=e^{i(k+1)\theta}\rho^{-k+1}\partial_{{z}}(\rho^{k}f),\\%
    \eta_-(f(z)e^{ik\theta})=e^{i(k-1)\theta}(\rho \partial_{\bar{z}}-k\partial_{\bar{z}}\rho)f=e^{i(k-1)\theta}\rho^{k+1}\partial_{\bar{z}}(\rho^{-k}f).%
\end{aligned}	
\end{equation}  

\begin{proposition} \label{prop:decomp}
    Suppose that $(M^\circ,g)$ is an oriented, simply connected AH surface. Fix an integer $k$ with $\pm k\geq 1$, and consider  $f_k\in \rho^{\delta}H_0^s(M)$ with $s\in \Rm$ and  $\delta \in (-|k|+1/2,|k|-1/2)$.
    There exist unique  $g_{k}\in \rho^{\delta}H_0^s(M) $ and $v_{k\mp 1}\in \rho^{\delta}H_0^{s+1}(M)$ such that $\eta_\mp (g_{k}e^{ik\theta})=0$  and
    \begin{equation}\label{eq:decomposition}
        f_ke^{ik\theta}=g_ke^{ik\theta}+\eta_\pm (v_{k\mp 1}e^{i(k\mp 1)\theta}),
    \end{equation}
    with the estimate
    \begin{equation}\label{eq:estimate_v_k_f}
       \|v_{k\mp 1}\|_{\rho^\delta H_0^{s+1}(M)}+\|g_k\|_{\rho^\delta H_0^{s}(M)} \lesssim \|f_k\|_{\rho^\delta H_0^{s}(M)}.
    \end{equation}
    Moreover, $g_k\in \rho^{\delta}H_0^\infty(M) $.

    If  $f_k\in {\Aphg^E(M)} $, where $\Re(E)>-|k|+1$, the same statement holds with $v_{k\mp 1}\in \Aphg^G(M)$, where $\Re(G)\geq\min\{|k|,\inf(E)\}^{-}$,
     and $g_{k}\in \Aphg^{G\cup E}(M)$.
\end{proposition}

\begin{remark}
	In the polyhomogeneous case, if $f_k\in \Aphg^E(M)$ satisfies $\eta_\mp (f_ke^{ik\theta})=0 $, $\pm k\geq 1$, then by uniqueness we have $g_k=f_k$ and $v_{k\mp1}=0$.
	Since then $G=\emptyset$,  the space $\Aphg^E(M)$ where $f_k$ lies stays unchanged.
\end{remark}

Since $\overline{\eta_-} = \eta_+$, we first observe that the case $k\le -1$ can be deduced by complex conjugation from the case $k\ge 1$: assuming the latter is known, and if $f_k e^{ik\theta} \in \Omega_k^{\cal F}$ for some $k\le -1$, then $\overline{f_k e^{ik\theta}} \in \Omega_{-k}^{\cal F}$ and as such can be written as 
\begin{align*}
    \overline{f_k e^{ik\theta}} = g_{-k} e^{-ik\theta} + \eta_+ \left( v_{-k-1} e^{-i(k+1)\theta} \right), \qquad \eta_- (g_{-k} e^{-ik\theta}) = 0.
\end{align*}
Setting $g_k=\overline{g_{-k}}$ and $v_{k+1}=\overline{v_{-k-1}}$ will yield decomposition \eqref{eq:decomposition}. Hence it suffices to focus on an elliptic decomposition for $\eta_+$ (case $k>0$) from now on. Applying $\eta_-$ to \eqref{eq:decomposition} gives the equation $\eta_- (f_k e^{ik\theta}) = \eta_- \eta_+ (v_{k-1} e^{i(k-1)\theta})$, hence determining $v_{k-1}$ calls for a solvability result for the operator $\eta_- \eta_+ \colon \Omega_{k-1}^{\cal F}\to \Omega_{k-1}^{\cal F}$ in appropriate spaces, akin to \cite[Lemma 6.5.4]{Paternain2023} in the compact case. To this end, using \eqref{eq:eta_pm_isoth}, we write for $k\ge 0$, 
\begin{align}
    \begin{split}
	\eta_-\eta_+(f(z)e^{ik\theta}) &= e^{ik\theta}P_k f, \\
	\text{where}\quad P_k &\coloneqq \rho^2\partial_{\bar{z}}\partial_z-k\rho\partial_{\bar{z}}\rho \,\partial_{z}+k\rho\partial_z\rho\,\partial_{\bar{z}}+(k\rho\partial_{\bar{z}z}^2\rho-k(k+1)\partial_{\bar{z}}\rho\partial_z\rho). 	
    \end{split}    
    \label{eq:Pmp}
\end{align}
Note that $P_k\in \Diff_0^2(M)$ (see Appendix \ref{sec:the_0_calculus}) and is formally self-adjoint with respect to the measure $\rho^{-2}\frac{1}{2}|\d z\wedge \d \bar{z}|$. Also observe that if $k=0$, $P_0=\frac{1}{4}\Delta_g$, where $\Delta_g$ is the negative Laplace-Beltrami operator of $g$, for which \eqref{eq:H2-L^2} is known, see e.g. \cite[Theorem C and Proposition F]{Lee2006}, as well as the references there.

\begin{proposition}\label{prop:Ppm}
    Suppose that $M$ is a simply connected, oriented AH surface.
    Let $k\ge 0$ be an integer. For every $s\in \Rm$ and $\delta \in (-k-1/2,k+1/2)$, the map
    \begin{equation}\label{eq:H2-L^2}
        P_k \colon \rho^\delta H_0^s(M)\to \rho^\delta H^{s-2}_0(M) 
    \end{equation}
    is an isomorphism.
    
    Further, let $h\in \Aphg^{E}(M)$ with $\Re({E})>-k$. For $\delta\in (-k-1/2,k+1/2)\cap (-\infty,\inf(E)-1/2)$, let $f\in \rho^\delta H_0^2(M)$ be the unique solution to $P_k f = h$. Then $f\in \Aphg^{G}(M)$, where ${G}\geq \min\{k+1,\inf(E)\}^{-}$.
\end{proposition}

\begin{proof}
    We will use Theorem \ref{thm:fredholm}. For definitions of the objects below, see Appendix \ref{sub:the_0_calculus}. The  operator $P_k\in \mathrm{Diff}_0^2(M)$ is elliptic, since its 0-principal symbol is
    \begin{equation}
	\sigma_0^2(\zeta)=-\frac{1}{4}|\xi|_{g(z)}^2\neq 0, \quad \text { for all } \quad \zeta =(z,\xi_j \rho^{-1}\d u^j)\in {}^0T^*M\setminus 0.
    \end{equation}

    To check full ellipticity, we compute the indicial operator and the normal operator at a fixed $p\in \partial M$.
    Recall that the metric is of the form \eqref{eq:isothermal} in isothermal coordinates near $p$. Because it is AH, $\|\d \rho\big|_{\partial M}\|_{\rho^2g}=1$, so upon fixing $p\in \partial M$ we can translate and rotate  the coordinates so that $(u^1,u^2)=0$ and $\d u^1=\d \rho$ at $p$, which also implies $\d \rho (\partial_{u^2})=0$ there.
    Thus $\rho=u^1+O(|u|^2)$ and we can use coordinates $(\rho,u^2)$, which allows us to  write \eqref{eq:Pmp} as
    \begin{equation}
        P_k = \frac{1}{4}\left(\rho^2(\partial_{\rho}^2+\partial_{u^2}^2)+2ik\rho \partial_{u^2}- k( k+1)\right)+ Q,
    \end{equation}
    where $Q\in \Diff_0^2(M)$ has smooth  coefficients which vanish at $(\rho,u^2)=0$.
       Freezing coefficients at $p$, taking Fourier transform in $u^2$ with dual variable $\eta$, and finally rescaling by setting $t=\rho|\eta|$, $\hat{\eta}=\eta/|\eta|$, we find that the reduced normal operator is given by 
    \begin{align}
       \hat{ N}(P_k,p,\hat{\eta})=&\frac{1}{4}\left(t^2\partial_t^2-t^2-2kt \hat{\eta}-k(k+1)\right)\\
       =&\frac{1}{4}\left(t^2\partial_t^2-(t+k \hat{\eta})^2-k\right)\text{ on }[0,\infty)_t, \quad  p\in \partial M,\quad \hat{\eta}\in \Sm^0.
    \end{align}   
    The {indicial operator} is 
    \begin{equation}
        I(P_k,p)=\frac{1}{4}\left(t^2\partial_t^2- k( k+1)\right),\quad  p\in \partial M,
    \end{equation}
     hence the {indicial family} is 
    \begin{equation}
        I(P_k,p,\zeta)=\frac{1}{4}\left(\zeta^2-\zeta- k(k+1)\right),\quad \zeta\in \Cm.
    \end{equation}    
    The {boundary spectrum} is therefore given by
    \begin{align}\label{eq:b-spectrum}
        \mathrm{Spec}_b(P_k,p )= \{1+k,-k\},
    \end{align}
    and we observe that it is independent of the basepoint $p$.

    To obtain Fredholm properties between the spaces in \eqref{eq:H2-L^2}, we could prove the requirements in Definition \ref{elliptic at weight} for $\alpha \in (-k,k+1)$.
    It is actually easier to restrict attention to the case $\alpha =1/2$  (this corresponds to the formally self-adjoint functional setting with $\delta = 0$ in \eqref{eq:H2-L^2}), and prove Fredholmness and invertibility.
    Since Theorem \ref{thm:fredholm} provides us with a detailed understanding of the inverse, we can  use Proposition~\ref{prop:Sobolev_mapping} to show that this inverse is well defined as an operator $\rho^\delta H_0^{s-2}(M)\to \rho^\delta H_0^s(M)$ for all $\delta\in (-k-1/2,k+1/2)$, and that it inverts $P_k$. 
   
    We check the requirements of Definition \ref{elliptic at weight} for $\alpha = 1/2$. Item \ref{item1}  holds by \eqref{eq:b-spectrum}. For \ref{item2}, suppose that 
    \begin{equation}\label{eq:Nu=0}
        \left(t^2\partial_t^2-(t+k \hat{\eta})^2-k\right)f=0,\quad f\in  L^2\big([0,\infty),t^{-2}{\d t}\big).
    \end{equation}
    As mentioned in the Appendix, $f\in C^\infty((0,\infty))$ and is Schwartz as $t\to \infty$.
    Moreover, by ellipticity of $ \hat{ N}(P_k,p,\hat{\eta})$ in the 0-calculus sense and \eqref{eq:0-ellipticity}, and the fact that $u$ is rapidly decaying at infinity, we also have $ {f}\in H^2_0([0,\infty),\frac{\d t}{t^2})$, where the Sobolev regularity near infinity is defined using powers of $\partial_t$.
    Upon multiplying \eqref{eq:Nu=0} by $\bar{f}$ and integrating with respect to $t^{-2}\d t$,
    \begin{equation}
    \label{eq:formally_sadj}
    \begin{aligned}
        0=&\int_0^\infty \left(t^2\partial_t^2-(t+k \hat{\eta})^2- k\right){f}\bar{{f}}t^{-2}\d t\\
        =&-\| t\partial_t {f}\|^2_{L^2([0,\infty),\frac{\d t}{t^2})}-\int_0^\infty \left((t+k \hat{\eta})^2+k\right)|{f}|^2t^{-2}\d t,
    \end{aligned}
    \end{equation}
    so $f=0$ since $k\geq 0$. 
    This computation also proves item \ref{item3} because $\hat{ N}(P_k,p,\hat{\eta})$ is formally self-adjoint with respect to the measure $\d t/t^2$.

    By Theorem \ref{thm:fredholm}, for any $s\in \Rm$ the operator
    \begin{equation}
      	P_k:H_0^s(M)\to H_0^{s-2}(M) \label{eq:d=0}
    \end{equation}   
    is Fredholm and in particular has closed range. We will check injectivity in the case $s=2$. This suffices, because if it does hold then it will hold for any $s$ by the ellipticity of $P_k$. So suppose that $f\in H_0^2(M)$ satisfies $P_k f=0$. Since ${C}_c^\infty(M^\circ)$ is dense in $H_0^2(M)$ , we can consider ${C}_c^\infty(M^\circ)\ni f_j\to f\in H_0^2(M)$.
    In terms of global isothermal coordinates identifying $M$ with $\Dm$ we compute, using everywhere the measures induced by $g$, 
    \begin{align}
	\big( P_k {f}_j,{f}_j\big)_{L^2(\Dm,\d V_g)} &= (2\pi)^{-1} \big( e^{-ik\theta}\eta_{-}\eta_{+}e^{ik\theta}{f}_j,{f}_j\big)_{L^2({}^0\overline{S\Dm})} \\
	&\quad= -(2\pi)^{-1} \big( \eta_{+}e^{ik\theta}{f}_j,\eta_{+}e^{ik\theta}{f}_j\big)_{L^2({}^0\overline{S\Dm})} \label{eq:IBP}
        \overset{\eqref{eq:eta_pm_isoth}}{=} -\|\rho^{-k+1}\partial_z(\rho ^{k }{f}_j)\|_{L^2(\Dm,\d V_g)}^2.
    \end{align}
    For the integration by parts we used \cite[Lemma 6.1.5]{Paternain2023}.
    Note that $\rho^{-k+1}\partial_z(\rho ^{k }\cdot )\in \Diff_0^1(\Dm)$, so $ \rho^{-k+1}\partial_z(\rho ^{k }{f}_j)\to  \rho^{-k+1}\partial_z(\rho ^{k }{f})$ in $L^2(\Dm)$. Thus taking a limit we find that $\partial_z(\rho ^{k }f)=0$.
    In particular, $\rho^{k}\Re(f)$ is harmonic with respect to the Laplacian of the constant curvature $-1$ hyperbolic metric on $\Dm$, and since $\d V_g$ is a smooth, non-vanishing multiple of $\d V_H$, $\rho^{k}\Re(f)$ is also in $L^2(\Dm,\d V_H)$. This implies that $\Re(f)$ must be 0, since the hyperbolic Laplacian does not have eigenvalues in $L^2(\Dm,\d V_H)$. Since $\rho ^{k }f$ is antiholomorphic and purely imaginary, $f=0$. Hence we obtain injectivity in \eqref{eq:d=0} for $s=2$ and for general $s\in \Rm$ by ellipticity.
    
    Since $P_k$ has closed range, surjectivity in \eqref{eq:d=0} is equivalent to injectivity of 
    \begin{equation}
    	P_k^*:H_0^{2-s}(M)\to H_0^{-s}(M),
    \end{equation}
    where $P_k^*$ is the adjoint in the functional setting \eqref{eq:d=0}; since $C_c^\infty(M^{\circ})$ is dense in $H^s_0(M)$ for $s\geq 0$, $P_k^*$  coincides with the $L^2(M)$-formal adjoint of $P_k$ on $L^2(M)$. %
     Since $P_k$ is formally self-adjoint and injective, it is surjective.

    By \eqref{eq:b-spectrum}, the index sets $\mathcal{E}_\pm $ in Definition \ref{def:index_sets} corresponding to $P_k$ are given by $\mathcal{E}_+=\{(1+l,0)\in \Nm\times\Nm_0 :\ l\geq k\}$ and $\mathcal{E}_-={\{(l,0)\in  \Nm_0\times\Nm_0:l\geq k\}}$. Therefore, Theorem \ref{thm:fredholm} implies that the inverse $G_k: L^2(M)\to H_0^2(M) $ of $P_k$ satisfies
    \begin{equation}\label{eq:Pinv}
        G_k\in \Psi^{-2}_0(M)+\Psi^{-\infty,({E}_{{\ell}},{E}_{{f}},{E}_{{r}})} _0(M),\qquad  {E}_{{\ell}},{E}_{{r}}\geq 1+k,\quad {E}_{{f}}\geq 0.
    \end{equation}    
    This implies that one has 
    \begin{equation}\label{eq:Pinv2}
        P_k G_k f=f,\qquad G_k P_k h = h
    \end{equation} 
    for all $f\in  L^2(M)$ and $h\in H_0^{2}(M)$, in particular if $f$, $h\in C_c^\infty(M^\circ)$.
    By \eqref{eq:Pinv} and Proposition \ref{prop:Sobolev_mapping}, for any $s\in \Rm$, \eqref{eq:Pinv2} can be extended by continuity to $f$ and $h$ in $ \rho^\delta H_0^{s}(M)$ and $\rho^\delta H_0^{s-2}(M)$ respectively, provided $\delta\in (-k-1/2,k+1/2)$. This finishes the proof of the first statement of the proposition.

    Now if $h\in \Aphg^E(M)$ with $\Re(E)>-k$, and $f\in \rho^\delta H_0^s(M)$ (with $\delta $ as in the statement) solves $P_k f = h$, we have  $f=G_k h$. Therefore by Proposition \ref{prop:phg_mapping}, since $\Re({E}_{r}+{E})>1$,
    \begin{equation}
        f\in \Aphg^{G}(M), \qquad {G}=({E}+{E}_{{f}})\overline{\cup}{E}_{{l}},\quad \text{so } \Re(G)\geq \min\{k+1,\inf(E)\}^{-}.     
    \end{equation} 
    The ``$-$'' superscript is explained by the fact that  the leading order term of the expansion of $h$ could contain logarithmic factors (or it could be a smooth  multiple of $\rho^{k+1}$, causing the extended union to have a logarithmic factor at leading order).
\end{proof}

We are now ready to prove the main result of the section.

\begin{proof}[Proof of Proposition \ref{prop:decomp}, case $k>0$]
    Applying $\eta_-$ to  \eqref{eq:decomposition} (case $k>0$) we find
    \begin{equation}
	\eta_-( f_ke^{ik\theta})=\eta_-\eta_+(v_{k-1}e^{i(k-1)\theta})\quad\implies\quad (\rho \partial_{\bar{z}}-k\partial_{\bar{z}}\rho)f_k=P_{k-1}v_{k-1}.
	\label{eq:Ppmeq}
    \end{equation}
    Since $(\rho \partial_{\bar{z}}-k\partial_{\bar{z}}\rho)\in \Diff_0^1(M),$ the left hand side is in $\rho^{\delta}H_0^{s-1}(M)$. Therefore, by Proposition \ref{prop:Ppm} there exists a unique solution $v_{k-1}$ to \eqref{eq:Ppmeq} in $\rho^\delta H_0^{s+1}(M)$ and one has the estimate
    \begin{equation}
	\|v_{k-1}\|_{\rho^\delta H_0^{s+1}(M)}\lesssim \|f_k\|_{\rho^\delta H_0^{s}(M)}.\label{eq:estimate_v_f}
    \end{equation}
    Upon setting
    \begin{equation}
	g_k=f_k-e^{-ik\theta}\eta_+(v_{k-1}e^{i(k-1)\theta})=f_k-(\rho\partial_{z}+(k-1)\partial_z\rho)v_{k-1}\in \rho^\delta H_0^{s}(M),\label{eq:g_k}
    \end{equation}
    we obtain \eqref{eq:decomposition} with the estimate 
    \begin{equation}
	\|g_k\|_{\rho^\delta H_0^{s}(M)}\lesssim (\|f_k\|_{\rho^\delta H_0^{s}(M)}+\|v_{k-1}\|_{\rho^\delta H_0^{s+1}(M)})\overset{\eqref{eq:estimate_v_f}}{\lesssim } \|f_k\|_{\rho^\delta H_0^{s}(M)}.
    \end{equation}
    This, together with \eqref{eq:estimate_v_f}, yields \eqref{eq:estimate_v_k_f}. For uniqueness, if $g_k'$, $v_{k-1}'$ as in the statement satisfy \eqref{eq:decomposition}, it is not hard to see that
    \begin{equation}
	\eta_-\eta_+\left((v_{k-1}-v_{k-1}')e^{i(k-1)\theta}\right)=0\implies P_{k-1} (v_{k-1}-v_{k-1}')=0\implies v_{k-1}-v_{k-1}'=0,
    \end{equation}
    by Proposition \ref{prop:Ppm}. This in turn implies $g_k=g_k'$.
    Regarding the regularity of $g_k$, note that $\eta_+\eta_-(g_ke^{ik\theta})=0\implies \overline{P_k} g_k=0$, so the ellipticity of $\overline{P_k}$ yields $g_k\in \rho^\delta H^\infty_0(M)$.
    
    The proof of the polyhomogeneous case follows the same steps: if $f_k\in \Aphg^E(M)$ with $\Re(E)>-k+1$, $(\rho \partial_{\bar{z}}-k\partial_{\bar{z}}\rho)f_k\in \Aphg^E(M)$ in \eqref{eq:Ppmeq}.
    By Proposition \ref{prop:Ppm} we obtain $v_{k-1}\in \Aphg^G(M)$ with $\Re(G)\geq \min\{k,\inf(E)\}^{-}$ and by \eqref{eq:g_k} we obtain 
    $g_k\in \Aphg^{G\cup E}(M).$
\end{proof}

\subsubsection{Proof of Lemma \ref{lem:oneforms}, Theorem \ref{thm:gauge_maybe} and Proposition \ref{prop:inj_on_repr_maybe2}}

\begin{proof}[Proof of Lemma \ref{lem:oneforms}] Let $f$ be as in the statement, and write, according to \eqref{eq:hat_notation}, \eqref{eq:isomorphism_lm}, $\ell_1 f = \hat{f}_1 + \hat{f}_{-1}$, where ${f}_{\pm 1}\in {\rho^\delta H_0^s(M)}$. Applying Proposition \ref{prop:decomp} to $\hat{f}_1$ and $\hat{f}_{-1}$ separately, we obtain
    \begin{align*}
	   \hat{f}_1 = \eta_+ \hat{v}_{0,+} + \hat{g}_1, \qquad	\hat{f}_{-1} = \eta_- \hat{v}_{0,-} + \hat{g}_{-1},
    \end{align*}
    where $\hat{v}_{0,\pm}\in \rho^\delta H_0^{s+1}(M)$ and $\hat{g}_{\pm 1}\in \rho^\delta H_0^{s}(M) \cap \ker \eta_\mp$. Using that $\eta_\pm = \frac{1}{2} (X \pm i X_\perp)$, %
    \begin{align*}
        \ell_1 f = \hat{f}_1 + \hat{f}_{-1} &= \hat{g}_1 + \hat{g}_{-1} + X \left( \frac{\hat{v}_{0,+}+\hat{v}_{0,-}}{2}\right) - X_\perp \left( \frac{\hat{v}_{0,+}-\hat{v}_{0,-}}{2i}\right) \\
        &= \ell_1 \left( \tilde{f}_1  + \d \left( \frac{\hat{v}_{0,+}+\hat{v}_{0,-}}{2}\right) + \star\d  \left( \frac{\hat{v}_{0,+}-\hat{v}_{0,-}}{2i}\right) \right) 
    \end{align*}
    where we have used \eqref{eq:XXperp_intertwiners}, and set $\tilde{f}_1 := \ell_1^{-1}(\hat{g}_1 + \hat{g}_{-1})$ by virtue of isomorphism \eqref{eq:ttdecom}. Decomposition \eqref{eq:oneform_decomp} follows upon defining $f_s := \frac{\hat{v}_{0,+}-\hat{v}_{0,-}}{2i}$ and $q_0 := \frac{\hat{v}_{0,+}+\hat{v}_{0,-}}{2}$.
\end{proof}

\begin{proof}[Proof of Theorem \ref{thm:gauge_maybe}] Let $f$ be as in the statement, and write, according to \eqref{eq:isomorphism_lm}, $\ell_m f = \hat{f}_m + \hat{f}_{-m} + \sum_{k=1}^{m-1} \hat{f}_{-m+2k}$, where ${f}_k\in {\rho^\delta H_0^s(M)}$ (recall \eqref{eq:hat_notation}). We now apply Proposition \ref{prop:decomp} to $\hat{f}_m$ and $\hat{f}_{-m}$ separately: decompose uniquely 
    \begin{align*}
	\hat{f}_m &= \eta_+ \hat{v}_{m-1} + \hat{g}_m = X \hat{v}_{m-1} + \hat{g}_m - \eta_- \hat{v}_{m-1}, \\
	\hat{f}_{-m} &= \eta_- \hat{v}_{-m+1} + \hat{g}_{-m} = X \hat{v}_{-m+1} + \hat{g}_{-m} - \eta_+ \hat{v}_{-m+1}.
    \end{align*}
    Then set $q \coloneqq  \ell_{m-1}^{-1}(\hat{v}_{m-1} + \hat{v}_{-m+1})$, $\tilde{f} \coloneqq  \ell_m^{-1}(\hat{g}_m + \hat{g}_{-m})$ using isomorphisms \eqref{eq:lmtraceless} and \eqref{eq:ttdecom}, and $\lambda \coloneqq  \ell_{m-2}^{-1}(\sum_{k=1}^{m-1} \hat{f}_{-m+2k} - \eta_- \hat{v}_{m-1} - \eta_{+} \hat{v}_{-m+1})$ using the isomorphism \eqref{eq:isomorphism_lm}.
    From the coordinate expression \eqref{eq:lmf} and from \eqref{eq:Hs} it is not hard to see that $q$ and $\tilde{f}$ have the same regularity as $v_{\pm (m-1)}\in \rho^\delta H_0^{s+1}(M)$ and $g_{\pm m}\in \rho^\delta H_0^{s}(M)$ respectively. 
    The polyhomogeneous case follows in the same way, and  the construction of $ q$, $\tilde{f}$, $\lambda$ together with \eqref{eq:estimate_v_k_f} imply that they satisfy the estimate \eqref{eq:cont_dep_maybe}.

    To see that they are unique, suppose that $\d^s q + L \lambda + \tilde f=0$. Applying $\ell_m$, we find $X\ell_{m-1}q+\ell_{m-2}\lambda +\ell_m \tilde f=0$. Using the isomorphisms \eqref{eq:lmtraceless} and \eqref{eq:ttdecom}, as above we have $\ell_{m-1} q= \hat{v}_{m-1}+\hat{v}_{-m+1},$ ${v}_j\in{\rho^{\delta}H_0^1}(M)$, and $\ell_m \tilde f=\hat{g}_m+\hat{g}_{-m} $ with $\eta_\pm \hat{g}_{\mp m}=0$.
    Thus isolating the $\pm m$ modes we obtain $q=0 $ and $\tilde f=0$ from the uniqueness statement in Proposition \ref{prop:decomp}.
    Then $\lambda=0$ by injectivity of $L$ (this follows, e.g., by \cite[Lemma 2.2]{Dairbekov2011}).
\end{proof}

Before proving Proposition \ref{prop:inj_on_repr_maybe2}, we state a lemma whose proof is a straightforward adaptation of the proofs of Lemma 3.14 and Proposition 3.15 in \cite{Graham2019}. 

\begin{lemma}\label{lm:boundary_det}Let $(M^\circ,g)$ be a non-trapping AH manifold.
Let $f\in \Aphg^E(M;S^m({}^0T^*M))$, $m\geq 0$, with $\Re(E)>0$ and $I_mf=0$. Then there exists $q\in \Aphg^E(M;S^{m-1}({}^0T^*M))$ such that $f-\d ^{s}q\in \dot{C}^{\infty}(M;S^m({}^0T^*M))$. (In the case $m=0$ the statement is that $f\in \dot C^\infty(M)$.)
\end{lemma}

We now prove Proposition \ref{prop:inj_on_repr_maybe2}.

\begin{proof}[Proof of Proposition \ref{prop:inj_on_repr_maybe2}]
    By Corollary \ref{cor:gauge},  $f = \d^s q + f^{\textrm{itt}}$, where $f^{\itt} \in \Aphg^G(M;{{S}}^{m}({}^0T^*M))$ with $\Re(G)>0$. By assumption, since Lemma \ref{lm:kernel} implies $0 = I_{m} f = I_{m} f^{\itt}$, by Lemma \ref{lm:boundary_det} there exists a tensor $u\in \Aphg^{G}(M;S^{m-1}(^0T^*M))$ such that $h \coloneqq  f^{\itt} -\d ^s u\in \dot{C}^{\infty}(M;S^{m}(^0T^*M))$ and with $I_{m} h = I_{m} f^{\itt} =0$. Now the proof of \cite[Theorem 1.1]{Graham2019}, which relies on Pestov identities on a compact exhaustion of $M^{\circ}$, goes through and yields the existence of $u'\in \dot{C}^{\infty}(M;S^{m-1}(^0T^*M))$ such that $h = \d^s u'$. Hence we have $\d^s u' = f^{\itt} - \d^s u$, i.e. $0 = \d^s (u+u') - f^{\itt}$. Since there exists $\delta\in (-1/2,1/2)$ such that $h\in \rho^{\delta }L^2(M;S^{m}(^0T^*M))$ and $u+u'\in \rho^{\delta}H_0^{1}(M;S^{m-1}(^0T^*M))$, the uniqueness statement in Corollary \ref{cor:gauge} yields $u+u'=0$ and $f^{\itt}=0$.
\end{proof}

\section{Proofs on the Poincar\'e disk} \label{sec:Poincare}

\subsection{Preliminaries}\label{sec:prelim_poinc}

Following notation in \cite{Eptaminitakis2024}, the interior $\Dm^\circ$ is equipped with the metric $g(z) = c^{-2}(z) |dz|^2$, $c(z) \coloneqq  \frac{1-|z|^2}{2}$. %
We will use for bdf the function $x$ defined in \eqref{eq:bdfs}, and in the coordinates $(x, \omega)$, $\omega =\arg(z)$, the Riemannian volume form is given by $\d V_H = x^{-2} \d x\ \d\omega$. In what follows, we write $L^2(\Dm)\coloneqq L^2(\Dm,\d V_H)$. Unit-speed hyperbolic geodesics in $S\Dm^\circ$ take the form $(\gamma_{\beta,a}(t), \dot\gamma_{\beta,a}(t)) = (z_{\beta,a}(\x(t)), c(z_{\beta,a}(\x(t)) e^{i\theta_{\beta,a}(\x(t))})$, where $\x(t) = \tanh(t/2)$, $(\beta,a)\in \Gh$, and 
\begin{align}
    z_{\beta,a}(\x) = e^{i\beta} \frac{(2+ia)\x + ia}{ia\x-2+ia}, \qquad \theta_{\beta,a}(\x) = \beta +\pi + 2\tan^{-1} \left( a \frac{\x+1}{2} \right).
    \label{eq:hypgeo}
\end{align}
We will use the shorthand $\pih$ for the map $\pi_{\Gh}$ defined in \eqref{eq:piG}. Its expression is given in \cite[Lemma 3.7]{Eptaminitakis2024}, and its pullback map induces a map from distributions on $\Gh$ to geodesically invariant distributions on $S\Dm^\circ$. The backprojection operator of $h\in C^\infty(\Gh)$ is given by
\begin{equation}
	I_0^\sharp h (z) \coloneqq \int_{S_z\Dm^{\circ}} h(\pih(z,w))\ \d S_z(w), \qquad z\in \Dm^\circ,
	\label{eq:backproj}
\end{equation}
and it is the formal adjoint of $I_0$ in the functional setting $L^2(\Dm)\to L^2(\Gh)$.

\subsubsection{The transport boundary \texorpdfstring{$\Gamma$}{Gamma}} \label{sec:transport_boundary}

Recall that a natural ``transport" boundary for $S\Dm^\circ$ as defined in \cite{Eptaminitakis2024} is $\Gamma := \Gamma_+ \cup \Gamma_-$, where $\Gamma_{+/-} = \Gh \times \{\pm 1\}$ is thought of as the inward/outward-pointing boundary. In coordinates, we have, for all $(\beta,a)\in \Gh$,
\begin{align*}
    \lim_{t\to -\infty} (\gamma_{\beta,a}(t), \dot \gamma_{\beta,a}(t)) = (\beta,a) \in \Gamma_+, \qquad \lim_{t\to +\infty} (\gamma_{\beta,a}(t), \dot \gamma_{\beta,a}(t)) = (\beta + \pi + 2\tan^{-1}a,a)\in \Gamma_-.     
\end{align*}
Both $\Gamma_\pm$ carry the measure $\d\beta\d a$, and for $f,g\in L^2(\Gamma)$, we denote 
\begin{align*}
    (f,g)_{L^2(\Gamma)} = \sum_{\sigma = \pm 1} \int_{\Gh} f(\beta,a,\sigma) \overline{g(\beta,a,\sigma)}\d \beta\d a.
\end{align*}
The space $\Gamma$ has three natural involutions, the scattering relation 
\begin{align*}
    S^H\colon \Gamma_\pm \to \Gamma_\mp, \qquad (\beta,a,\pm1)\mapsto (\beta+\pi \pm 2\tan^{-1}a, a ,\mp 1), \qquad (\beta,a)\in (\Rm/2\pi\Zm)\times \Rm, 
\end{align*}
the antipodal map 
\begin{align*}
    \sfa_H \colon \Gamma_\pm \to \Gamma_\mp, \qquad (\beta,a,\pm1)\mapsto (\beta, -a ,\mp 1), \qquad (\beta,a)\in (\Rm/2\pi\Zm)\times \Rm, 
\end{align*}
and the antipodal scattering relation $S_A^H := \sfa_H\circ S^H$, also defined in \eqref{eq:antipodal_sc}. 

\subsubsection{Bases for the data space \texorpdfstring{$L^2(\Gh)$}{L2(G)}} 

Recalling the splitting \eqref{eq:L2pm} of $L^2(\Gh)$, we now define bases of $L^2_\pm(\Gh)$ which will turn out useful in what follows. Specifically, following \cite[Eq. (50)]{Eptaminitakis2024} we define\footnote{A typo in \cite[Eq. (50)]{Eptaminitakis2024} states that $\phi_{n,k}$ is only defined for $n\ge 0$.}, on $(\Rm/2\pi\Zm)_\beta \times [-\pi/2,\pi/2]_\alpha$
\begin{align}
    \begin{split}
        \psi_{n,k} &= \frac{(-1)^n}{2\pi} e^{i(n-2k)(\beta+\alpha)} (e^{i(n+1)\alpha}+(-1)^n e^{-i(n+1)\alpha}), \qquad n\ge 0,\  k\in \Zm, \\
        \phi_{n,k} &= \frac{(-1)^n}{2\pi} e^{i(n-2k)(\beta+\alpha)} (e^{i(n+1)\alpha}-(-1)^n e^{-i(n+1)\alpha}), \qquad n\ge -1, \ k\in \Zm.     
    \end{split}
    \label{eq:phipsiE}
\end{align}
Then,  recalling \eqref{eq:bdfs}, we have 
\begin{align}
    L^2_+(\Gh) &= \text{span}_{L^2}\left\{\psi_{n,k}^{H}(\beta,a) := \muh \psi_{n,k}(\beta,\tan^{-1}a), \ n\in \Nm_0,\ k\in \Zm\right\},
    \label{eq:basis1} \\
    L^2_-(\Gh) &= \text{span}_{L^2}\left\{\phi_{n,k}^{H}(\beta,a) := \muh \phi_{n,k}(\beta,\tan^{-1}a), \ n\ge -1,\ k\in \Zm\right\}.
    \label{eq:basis2}    
\end{align}
Both bases above are orthonormal,
and can be seen to be eigenfunctions of the two operators 
\begin{align}
    D_\beta = \frac{1}{i}\partial_\beta, \qquad {\cal T}_0 = - (\partial_\beta-(1+a^2)\partial_a)^2 + 2a(\partial_\beta-(1+a^2)\partial_a) - (2a^2+1)Id,
    \label{eq:T0}
\end{align} 
specifically,
\begin{align}
    ({\cal T}_0, D_\beta) \psi_{n,k}^{H} &= ( (n+1)^2, n-2k) \psi_{n,k}^{H}, \qquad n\ge 0,\quad k\in \Zm,  \label{eq:T0spectrum}\\
    ({\cal T}_0, D_\beta) \phi_{n,k}^{H} &= ( (n+1)^2, n-2k) \phi_{n,k}^{H}, \qquad n\ge -1,\quad k\in \Zm.
    \label{eq:T0spectrum2}    
\end{align}

\subsubsection{The boundary operators \texorpdfstring{$P_\pm^H$}{Ppm}}\label{sec:dataSpace}

Recall the operators of even and odd extension relative to the scattering relation, given by 
\begin{align}
    A^H_\pm \colon L^2(\Gamma_+) \to L^2(\Gamma), \qquad A^H_\pm u|_{\Gamma_+} = u, \qquad A^H_\pm u|_{\Gamma_-} = \pm u \circ S^H,
    \label{eq:AHpm}
\end{align}
with $L^2-L^2$ adjoints 
\begin{align}\label{eq:L2adjoints}
    (A^H_\pm)^* \colon L^2(\Gamma) \to L^2(\Gamma_+), \qquad (A^H_\pm)^* u := (u|_{\Gamma_+}) \pm (u|_{\Gamma_-})\circ S^H.
\end{align}
Recall the $\Rm$-Hilbert transform 
\begin{align}
    H_\Rm f(a) := \frac{1}{\pi} \text{p.v.}\int_\Rm \frac{f(a')}{a-a'}\ \d a', \qquad a\in \Rm,
    \label{eq:HR}
\end{align}
skew-adjoint relative to $L^2(\Rm, \d a)$. We extend it to $L^2(\Gamma_\pm)$ as a fiberwise operator. Out of $H_\Rm$, we then define fiberwise even and odd Hilbert transforms ${\cal H}_\pm \colon L^2(\Gamma)\to L^2(\Gamma)$, given by 
\begin{align}
    {\cal H}_\pm u (\beta,a, 1) = H_\Rm \left( \frac{u \pm u\circ \sfa_H}{2} \right) (\beta,a,1), \qquad (\beta,a,1)\in \Gamma_+,
    \label{eq:Hpm}
\end{align}
and extended to $\Gamma_-$ in such a way as to satisfy the symmetry property $\sfa_H^* {\cal H}_\pm = \pm {\cal H}_\pm$.

Finally, we define the analogues of the Pestov-Uhlmann operators
\begin{align}
    P_\pm^H \colon L^2(\Gh) \to L^2(\Gh), \qquad P_\pm^H := (A_-^H)^* {\cal H}_\pm A_+^H.
    \label{eq:PU}
\end{align}
By symmetry considerations, it is easy to find that
\begin{align}
    P_+^H (L^2_- (\Gh)) = P_-^H (L^2_+ (\Gh)) = 0, \quad P_+^H (L^2_+ (\Gh)) \subset L^2_-(\Gh), \quad P_-^H (L^2_- (\Gh)) \subset L^2_+(\Gh).\qquad 
    \label{eq:Psymmetries}
\end{align}
The operator $P_-^H$ was defined in \cite{Eptaminitakis2024} and was seen to be intertwined to the classical Pestov-Uhlmann operator (as defined in \cite{Pestov2004}) for the Euclidean unit disk, via a diffeomorphism relating fan-beam coordinates in the Euclidean disk and $\Gh$ coordinates on the Poincar\'e disk. While it can be seen that a similar relation does not occur as simply for the $P_+^H$ operator, an important relation specific to the Poincar\'e disk occurs here: 
\begin{lemma}\label{lem:Ppmadjoints}\phantom{a}
\begin{enumerate}[(a)]
    \item\label{item_pmadj1} We have the relation $(P_-^H)^* = - P_+^H$ on $L^2_+(\Gh)$. 
    \item\label{item_pmadj2} We have the alternative expressions
    \begin{align}
        P^H_- = (A_+^H)^* \H_+ A_-^H, \qquad P^H_+ = (A_+^H)^* \H_- A_-^H. 
        \label{eq:PU2}
    \end{align}
\end{enumerate}
\end{lemma}

It should be noted that the relation $(P_-^H)^* = - P_+^H$ cannot hold for the original Pestov-Uhlmann operators, as can be observed from their Singular Value Decompositions in the Euclidean disk \cite{Monard2015a}.

\begin{proof} To prove \ref{item_pmadj1}, we aim to show $(P_-^H f,g)_{L^2(\Gh)} = -(f,P_+^H g)_{L^2(\Gh)}$ for all $f,g\in L^2(\Gh)$, and by virtue of \eqref{eq:Psymmetries}, it is enough to assume $f\in L^2_- (\Gh)$ and $g\in L^2_+ (\Gh)$. In particular, this means that $A^H_+ f$ and $A^H_- g$ are fiberwise odd, and $A^H_- f$ and $A^H_+ g$ are fiberwise even on $\Gamma$ (i.e., relative to $\sfa_H^*$). In particular, from \eqref{eq:Hpm}, we have 
    \begin{align*}
        ({\cal H}_- A_+^H f)|_{\Gamma_+} = H_\Rm f, \qquad ({\cal H}_+ A_+^H g)|_{\Gamma_+} = H_\Rm g,
    \end{align*}
    the former extended by oddness to $\Gamma$, the latter by evenness. We then write
    \begin{align*}
        (P_-^H f,g)_{L^2(\Gh)} = ({\cal H}_- A_+^H f, A_-^H g)_{L^2(\Gamma)} &\stackrel{(\star)}{=} 2 (({\cal H}_- A_+^H f)|_{\Gamma_+}, g)_{L^2(\Gamma_+)} = 2 (H_\Rm f, g)_{L^2(\Gh)}, \\
        (f,P_+^H g)_{L^2(\Gh)} = (A_-^H f, {\cal H}_+ A_+^H g)_{L^2(\Gamma)} &\stackrel{(\star)}{=} 2(f, ({\cal H}_+ A_+^H g)|_{\Gamma_+})_{L^2(\Gamma_+)} = 2 (f, H_\Rm g)_{L^2(\Gh)},
    \end{align*}
    where step $(\star)$ follows by fiberwise evenness of the integrand under consideration. The result follows by the skew-adjointness of $H_\Rm$.  

    We now prove \ref{item_pmadj2}, focusing on \eqref{eq:PU2} for $P_-^H$, as the proof for $P_+^H$ is similar. To this end, notice that for $\phi\in L^2_-(\Gh)$, $A_-^H \phi$ is fiberwise even on $\Gamma$ and $A_+^H \phi$ is fiberwise odd on $\Gamma$. By the definition of $\H_\pm$, $\H_+ A_-^H \phi$ is fiberwise even, with $(\H_+ A_-^H \phi)|_{\Gamma_-} = (\H_+ A_-^H \phi)|_{\Gamma_+}\circ \sfa_H = H_\Rm \phi \circ \sfa_H$. Similarly, $\H_- A_+^H \phi$ is fiberwise odd, with $(\H_- A_+^H \phi)|_{\Gamma_-} = -(\H_- A_+^H \phi)|_{\Gamma_+}\circ \sfa_H = - H_\Rm \phi \circ \sfa_H$. Then,
    \begin{align*}
        (A_+^H)^* \H_+ A_-^H \phi = (\H_+ A_-^H \phi)|_{\Gamma_+} + (\H_+ A_-^H \phi)|_{\Gamma_-} \circ S^H &= H_\Rm \phi + H_\Rm \phi \circ \sfa_H \circ S^H \\
        &= (A_-^H)^* \H_- A_+^H \phi = P_- \phi,
    \end{align*}
    hence the result. 
\end{proof}

Combining Lemma \ref{lem:Ppmadjoints} with \cite[Corollary 5.3]{Eptaminitakis2024} (augmented with the additional observation that $P_-^H \phi_{-1,k}^H=0$ for all $k\in \Zm$), we deduce the Singular Value Decomposition of $P_+^H$:

\begin{corollary}
    Relative to bases \eqref{eq:basis1}-\eqref{eq:basis2}, the operator $P_+^H \colon L^2_+ (\Gh)\to L^2_- (\Gh)$ satisfies: 
    \begin{align}
        P_+^H \psi_{n,k}^{H} = (-2i) 1_{0\le k\le n}\ \phi_{n,k}^H, \quad n\ge 0,\ 0\le k\le n.
        \label{eq:P+HSVD}
    \end{align}   
\end{corollary}

\subsection{The transform \texorpdfstring{$I_\perp$}{Iperp}}\label{sec:Iperp}

We study the map $I_\perp \colon C_c^\infty(\Dm^\circ) \to C_c^\infty (\Gh)$, given by $I_\perp f = -I_1 (\star \d f)$. We first observe that by \eqref{eq:XXperp_intertwiners}, \eqref{eq:Xperpconf} and Lemma~\ref{lem:surjective}, $I_\perp$ extends to a bounded operator for any $\delta >0$
\begin{align}
    I_\perp \colon x^\delta H^1_0 (\Dm) \to L^2_-(\Gh). 
    \label{eq:IperpExtended}
\end{align}

We first derive a relation akin to \cite[Eq. (4.3)]{Pestov2004} adapted to the AH setting. 

\begin{proposition}\label{prop:intertwiner}
    We have the relation
    \begin{align*}
        \frac{1}{2\pi}I_\perp I_0^\sharp = P_+^H \qquad \text{on }\qquad  C_c^\infty(\Gh).
    \end{align*}
\end{proposition}

\hide{
\begin{remark}[Temporary: using this for every $\gamma>-1$] Using \cite{Eptaminitakis2024}, since $(I_0 x^{2+2\gamma})^* = x^{-1} I_0^\sharp \muh^{-2\gamma}$, this gives (denoting $N_\gamma = (I_0 x^{2+2\gamma})^* I_0 x^{2+2\gamma}$)
\begin{align*}
    I_\perp x N_\gamma = P_+^H \muh^{-2\gamma} I_0 x^{2+2\gamma},
\end{align*}
and with $N_\gamma\colon x^{-\gamma-3/2} L^2 \to H^{1,\gamma}_w$ being an isomorphism, this gives us $I_\perp$ defined on $xH_w^{1,\gamma}(\Dm)$, as the composition
\begin{align*}
    P_+^H \muh^{-2\gamma} I_0 x^{2+2\gamma} N_\gamma^{-1} \colon H_w^{1,\gamma}(\Dm) \stackrel{(\diamond)}{\longrightarrow} L^2(\Gh).
\end{align*}
The question is for which $\gamma$ is the map ``$(\diamond)$" continuous. Observing that $\muh^{-2\gamma} I_0 x^{2+2\gamma}$ lands into $\muh^{-\gamma} L^2_+ (\Gh) = L^2_+ (\Gh, \muh^{2\gamma})$, the question is then: for which values of $\alpha$ is the Hilbert transform bounded on $L^2(\Rm, (1+x^2)^{\alpha/2}\ dx)$, and the answer seems to be (using $A_2$ Muckenhoupt weight theory) for $|\alpha|<1$, i.e. for $|\gamma|<1/2$, which is the right range to consider for us, further restricting to $\gamma \in (-1/2,0)$ so that $\muh^{-\gamma}L^2(\Gh) \subset L^2(\Gh)$. In principle this gives us mapping properties for $I_\perp (xH^{1,\gamma}_w(\Dm))$ as well, although might not directly give a solid SVD (need to check), nor does it serve any purpose for the paper, such as understanding the spaces $I_\perp (x^\delta H^1_0(\Dm))$.     
\end{remark}

}

The proof of Proposition \ref{prop:intertwiner} relies on the following preliminary lemma relating the fiberwise Hilbert transform $H$ on $S\Dm^\circ$ (see, e.g. \cite{Paternain2023}) and traces 
at the boundary $\Gamma$. 

\begin{lemma}
    \label{lem:hilbert-limit}
    For any $w\in C_c^\infty(\Gh)$, we have 
    \begin{align}
        \lim_{t\to -\infty} H(w\circ \pih)(\gamma_{\beta,a}(t), \dot{\gamma}_{\beta,a}(t)) = H_\Rm w (\beta,a), \qquad (\beta,a)\in \Gh,
        \label{eq:Hlimit}
    \end{align}
    where $H_\Rm$ is the $\Rm$-Hilbert transform defined in Eq. \eqref{eq:HR}.
\end{lemma}

\begin{proof}[Proof of Lemma \ref{lem:hilbert-limit}]
First note that the limit as $t\to -\infty$ can be thought of as a limit as $\x = \tanh (t/2) \to -1$, and below, we will identify $(\gamma_{\beta,a}(t), \dot{\gamma}_{\beta,a}(t))$ with  $(z_{\beta,a}(\x), \theta_{\beta,a}(\x))$ as defined in \eqref{eq:hypgeo}. By definition of the fiberwise Hilbert transform, we have 
\begin{align}
    H (w\circ\pih)(z_{\beta,a}(\x),\theta_{\beta,a}(\x)) &= \text{p.v.}\int_{-\pi}^{\pi} h(\theta)\ w \circ \pih (z_{\beta,a}(\x),\theta_{\beta,a}(\x)-\theta)\ \d\theta \nonumber \\*
    &= -\int_0^\pi h(\theta)\sum_{\pm}\pm w\circ\pih(z_{\beta,a}(\x),\theta_{\beta,a}(\x)\pm\theta) \ \d\theta, \label{eq:pv-sym}
\end{align}
where $h(\theta) := \frac{1+\cos\theta}{2\pi \sin\theta}$. The singularity of $h$ at $\theta=0$ warrants a principal value integral which, to study the parameter-dependent limits  below, is rewritten as the regular integral \eqref{eq:pv-sym}.
We now need to understand the behavior of $\pih(z_{\beta,a}(\x),\theta_{\beta,a}(\x)\pm\theta)$. We show
\begin{lemma}
\label{lem:tildeba-asymp}
If $(\tilde\beta,\tilde{a}) = \pih(z_{\beta,a}(\x),\theta_{\beta,a}(\x)+\theta)$ where $|\theta|<\pi$, then
\begin{align}
    \tilde\beta = \beta + O((1+\x)|\sin\theta|), \qquad   \tilde{a} = a +\frac{1}{1+\x}\sin\theta+O(|\sin\theta|),
    \label{eq:abasym}
\end{align}
where the big-O estimates are uniform as $\x\to-1$ on any compact subset of $(-\pi,\pi)$.
\end{lemma}
Assuming Lemma \ref{lem:tildeba-asymp}, we proceed as follows. First, we note, for a fixed $\theta\in(-\pi,\pi)\backslash\{0\}$, that if $(\tilde\beta,\tilde{a}) = {\pih(z_{\beta,a}(\x),\theta_{\beta,a}(\x)+\theta)}$, then $|\tilde{a}|\to\infty$  as $\x\to-1$. In particular, the integrand of the right-hand side of \eqref{eq:pv-sym} eventually equals $0$ as $\x\to-1$ for any $\theta\in(0,\pi)$ since $w$ has compact support. Thus, for any $\epsilon>0$, we have
\[\lim_{\x\to-1}-\int_{\epsilon}^\pi h(\theta)\sum_{\pm}\pm w\circ\pih(z_{\beta,a}(\x),\theta_{\beta,a}(\x)\pm\theta) \ \d\theta = 0\]
by the Dominated Convergence Theorem (note that $h(\theta)$ is bounded on $(\epsilon,\pi)$, so the integrand can be bounded by $\|w\|_{L^\infty}h(\theta)$, which is integrable on $(\epsilon,\pi)$).

It suffices to study the integral on the right-hand side of \eqref{eq:pv-sym}, with the limits of integration changed to $(0,\epsilon)$ for some small $\epsilon>0$.
In that integral, we make the change of variables
\[s=\frac{1}{1+\x}\sin\theta,\quad \d s=\frac{1}{1+\x}\cos\theta\ \d\theta.\]
Note that $h(\theta) = \tilde h(\theta)\frac{\cos(\theta)}{\sin(\theta)}$, where $\tilde h(\theta) = \frac{1+\cos\theta}{2\pi\cos\theta}$; note that $\tilde h(\theta) = \frac{1}{\pi}+O(\sin^2\theta)$. We then have
\[h(\theta)\ d\theta = \tilde h(\theta)\frac{\frac{1}{1+\x}\cos\theta}{\frac{1}{1+\x}\sin\theta}\d \theta = \hat{h}(s;\x)\frac{\d s}{s}\]
where $\hat{h}(s;\x):=\tilde{h}(\arcsin((1+\x)s)) = \frac{1}{\pi}+O((1+\x)^2s^2)$.
In addition, if we let
\begin{align*}
    (\tilde\beta_{\pm}(s;\x),\tilde a_{\pm}(s;\x)) := \pih \left(z_{\beta,a}(\x),\theta_{\beta,a}(\x)\pm\arcsin((1+\x)s) \right),  
\end{align*}
then Lemma \ref{lem:tildeba-asymp} gives $\tilde\beta_{\pm} = \beta + O((1+\x)^2|s|)$, and $\tilde a_{\pm} = a \pm s + O((1+\x)|s|)$. It follows that
\begin{align*}
    -\int_0^\epsilon h(\theta)\sum_{\pm}\pm w\circ\pih(z_{\beta,a}(\x),\theta_{\beta,a}(\x)\pm\theta) \ \d\theta = -\int_0^{\frac{\sin(\epsilon)}{1+\x}}\hat{h}(s;\x)\sum_{\pm}\pm w(\tilde\beta_{\pm}(s;\x),\tilde a_{\pm}(s;\x)) \frac{\d s}{s}.    
\end{align*}
We now apply Dominated Convergence Theorem to the right-hand side. First we note that the integrand is uniformly bounded: indeed, $\tilde h(\theta)$ is uniformly bounded on $(0,\epsilon)$, and hence so is $\hat h(s;\x)$ on $(0,\sin(\epsilon)/(1+\x))$. Moreover, we have
\[|\tilde\beta_+(s;\x)-\tilde\beta_-(s;\x)|\le C_\beta(1+\x)^2|s|,\quad |\tilde a_+(s;\x)-\tilde a_-(s;\x)|\le 2|s|+C_a(1+\x)|s|,\]
so by the Mean Value Theorem we have
\[\left|\frac{w(\tilde\beta_+(s;\x),\tilde a_+(s;\x))-w(\tilde\beta_-(s;\x),\tilde a_-(s;\x))}{s}\right|\le \|w\|_{C^1}(C_\beta(1+\x)^2+2+C_a(1+\x)).\]
Finally, since $\tilde a_\pm(s;\x) = a\pm s+O((1+\x)|s|)$, and $w$ is compactly supported, we see that $w(\tilde\beta_{\pm}(s;\x),\tilde a_{\pm}(s;\x))$ is supported in a uniform compact set as $\x\to -1$. Thus, the integrand is uniformly bounded and supported in a uniform compact set, and hence dominated by an integrable function. For each $s>0$, we have
\[\lim_{\x\to -1}\hat{h}(s;\x)\sum_{\pm}\pm w(\tilde\beta_{\pm}(s;\x),\tilde a_{\pm}(s;x)) \frac{1}{s} = \frac{1}{\pi}\sum_{\pm} \pm w(\beta,a\pm s)\frac{1}{s}.\]
Hence, by the Dominated Convergence Theorem, 
\begin{align*}
\lim_{\x\to-1} -\int_0^{\frac{\sin(\epsilon)}{1+\x}}\hat{h}(s;\x)\sum_{\pm}\pm w(\tilde\beta_{\pm}(s;\x),\tilde a_{\pm}(s;x)) \frac{\d s}{s} = -\int_0^\infty\frac{1}{\pi s}\sum_{\pm}\pm w(\beta, a\pm s)\ \d s.
\end{align*}
Changing variable $a' = a\pm s$ in the $\pm$ term recovers $H_\Rm w (\beta,a)$. Lemma \ref{lem:hilbert-limit} is proved.

\begin{proof}[Proof of Lemma \ref{lem:tildeba-asymp}]
    Recalling \eqref{eq:hypgeo}, a preliminary computation shows that 
    \begin{align}
        e^{-i\theta_{\beta,a}}z_{\beta,a} %
        &= -1 - 2w, \qquad w := - (\x+1) \frac{1+ia}{2+ia (\x+1)}, \label{eq:prelim1} \\
        \frac{2e^{-i\theta_{\beta,a}} z_{\beta,a} }{1-|z_{\beta,a}|^2} &= \frac{4\x + a^2 (\x+1)^2}{2(1-\x^2)} + ia.
        \label{eq:prelim2}
    \end{align}
    In particular, $e^{-i\theta_{\beta,a}}z_{\beta,a} \to -1$ as $\x\to-1$. Recall from \cite[Lemma 3.7]{Eptaminitakis2024} that $\pih(z,\tilde\theta) = (\tilde\beta,\tilde a)$ where 
    \begin{equation}
    \label{eq:tildebetaa}
    \tilde\beta(z,\tilde\theta) = \tilde\theta + \pi + 2 \arg (1-z e^{-i\tilde\theta}), \qquad \tilde a(z,\tilde\theta) = \text{Im} \left(\frac{2ze^{-i\tilde\theta}}{1-|z|^2} \right).
    \end{equation}
    Thus, writing $(\beta,a) = \pih (z_{\beta,a},\theta_{\beta,a})$ and $(\tilde\beta,\tilde a) = \pih(z_{\beta,a},\theta_{\beta,a}+\theta)$, we find, with $\tilde \theta = \theta_{\beta,a}+\theta$,
    \begin{align*}
        \tilde\beta - \beta &=  \left(\tilde\theta + \pi + 2 \arg (1-z_{\beta,a} e^{-i\tilde\theta})\right) - \left(\theta_{\beta,a}+\pi+2\arg(1-z_{\beta,a}e^{-i\theta_{\beta,a}})\right)\\
        &= 2\arg \left( (e^{i\theta/2}-z_{\beta,a}e^{-i\theta_{\beta,a}}e^{-i\theta/2})/(1-z_{\beta,a}e^{-i\theta_{\beta,a}}) \right) \\
        &\!\!\! \stackrel{\eqref{eq:prelim1}}{=} 2\arg\left(\cos (\theta/2) -i\frac{w}{1+w}\sin (\theta/2)\right) \\
        &= O(|w||\sin\theta|).
    \end{align*}
    The estimate for $\tilde\beta$ in \eqref{eq:abasym} follows since $|w| = O(1+\x)$. For $\tilde a$, we write
    \begin{equation}
    \label{eq:tildea}
    \begin{aligned}
        \tilde a \stackrel{\eqref{eq:tildebetaa}}{=} \text{Im}\left(\frac{2z_{\beta,a}e^{-i\tilde\theta}}{1-|z_{\beta,a}|^2}\right) &\!\!\stackrel{\eqref{eq:prelim2}}{=} \text{Im} \left(  e^{-i\theta} \left( \frac{4\x + a^2 (\x+1)^2}{2(1-\x^2)}+ia      \right) \right) \\
        &= a \cos\theta - \frac{4\x + a^2 (\x+1)^2}{2(1-\x^2)} \sin\theta \\
        &= a + \frac{1}{1+\x}\sin\theta + \left(a(\cos\theta-1)-\frac{2+a^2(\x+1)}{2(1-\x)}\sin\theta\right).
    \end{aligned}
    \end{equation}
    Then \eqref{eq:abasym} follows after noting the last term is $O(|\sin\theta|)$ as $\x\to-1$. Lemma \ref{lem:tildeba-asymp} is proved.
\end{proof}

\end{proof}

\begin{proof}[Proof of Proposition \ref{prop:intertwiner}] Take $w\in C_c^\infty(\Gh) \cap \ker (Id - S_A^*)$ and extend it as a fiberwise even invariant distribution $w\circ\pih$ on $S\Dm^\circ$. We now write 
\begin{align*}
    HX(w\circ \pih) - X H (w\circ \pih) = [H,X] (w\circ \pih) \stackrel{(\diamond)}{=} X_\perp \P_0 (w\circ \pih) + \P_0 (X_\perp (w\circ\pih)),    
\end{align*}
where $\P_0$ is defined in \eqref{eq:projector} and $(\diamond)$ can be found in, e.g., \cite[Proposition 6.2.2]{Paternain2023}. In the above display, the leftmost term vanishes because $w\circ\pih$ is geodesically invariant, and the rightmost term vanishes because $X_\perp (w\circ\pih)$ is fiberwise odd. Integrating the remaining terms along the geodesic flow and using the fundamental theorem of calculus, we arrive at the relation
\begin{align*}
    \frac{1}{2\pi} I_\perp I_0^\sharp w \overset{\eqref{eq:backproj}}= I X_\perp \P_0 (w\circ \pih) = - IX H (w\circ\pih) = (A_-^H)^* (H (w\circ\pih))|_{\Gamma}. 
\end{align*}
Now using Lemma \ref{lem:hilbert-limit}, we have $(H (w\circ\pih))|_{\Gamma_+} = H_\Rm w$. Moreover, since $w\circ\pi_h$ is fiberwise even, at the transport boundary we have $(H (w\circ\pih))|_{\Gamma} \in \ker (Id - \sfa_H^*)$, and hence $(H (w\circ\pih))|_{\Gamma_-} = (H (w\circ\pih))|_{\Gamma_+} \circ \sfa_H$, so we have $(H (w\circ\pih))|_{\Gamma} = \H_+ A_+ ^Hw$. Proposition \ref{prop:intertwiner} is proved. 
\end{proof}

As a consequence of Proposition \ref{prop:intertwiner} and known SVD results about $I_0$, we deduce in Proposition \ref{prop:IperpSVD} a form of generalized SVD  for the unbounded operator
\begin{equation}\label{eq:Iperp_unb}
    iI_\perp:xC_\ev^\infty(\Dm)\subset x^{-1/2}L^2(\Dm)\to L^2_-(\Gh)
\end{equation} and a range characterization (Corollary \ref{cor:Iperp_range}) for $I_\perp\colon x H^{1,0}_w(\Dm) \to L^2_-(\Gh)$. 
In the statements below, $\Phi(z) = \frac{2z}{1+|z|^2}$, and $\{Z_{n,k}^0\}_{n\ge 0, 0\le k\le n}$ refer to the Zernike polynomials as defined in \cite{Eptaminitakis2024}. The functions $\{\widehat{\Phi^{*}Z_{n,k}^0}\}_{n\ge 0, 0\le k\le n}$ are in $C_\ev^\infty(\Dm)$ and form an orthonormal basis of $x^{-3/2}L^2(\Dm)$ (the notation $\widehat \cdot$ stands for scaling to achieve norm 1). 
Moreover, 
for $m\geq 0,$ we write $\ell^{*}_{m}$ for the $L^2(\Dm;S^{m}({}^0T^*\Dm))\to L^{2}(\OSD)$-adjoint of $\ell _{m}$.

\begin{proposition}\label{prop:IperpSVD}
The   Hilbert space adjoint of the unbounded, densely defined operator \eqref{eq:Iperp_unb} is given by $(iI_\perp)^*=-ix^{-1}\star \d\ell_1^* I^\sharp$, with domain 
\begin{equation}\label{eq:dom_adj}
    \mathcal{D}((iI_\perp)^*)=\Big\{u=\sum_{n=-1}^\infty\sum_{ k\in \mathbb{Z} } a_{n,k} \, \phi_{n,k}^H:\sum_{n=-1}^\infty\sum_{ k\in \mathbb{Z} }|a_{n,k}|^2<\infty \ \ \text{and}\ \ \sum _{n\geq 0}\sum_{0\leq k\leq n}|a_{n,k}|^2(n+1)<\infty\Big\}
\end{equation}
and nullspace 
\begin{equation}
    \ker ((iI_\perp)^*) = \mathrm{span} _{L^2(\Gh)}\{\phi_{n,k}^H:n\geq -1 ,k<0\text{ or } k>n\}.
\end{equation}
Further, we have the spectral decompositions
    \begin{align}
       i I_\perp (x  \widehat{\Phi^* Z_{n,k}^0} ) = {2\sqrt{\pi}}{\sqrt{n+1}} \,\phi_{n,k}^H,\qquad n\geq 0, \quad 0\leq k\leq n,
        \label{eq:IperpSVD}\\
        (iI_\perp)^*\phi_{n,k}^H = -ix^{-1}\star \d \ell_1^* I^\sharp  \big( \phi_{n,k}^H\big) = 2\sqrt \pi\sqrt{n+1}  \, x\widehat{\Phi^* Z_{n,k}^0} ,\qquad n\geq 0, \quad 0\leq k\leq n.\label{eq:spectral_Iperp_star}
    \end{align}
\end{proposition}

\begin{proof} We first derive \eqref{eq:IperpSVD}. Recall from \cite[Theorem 2.6]{Eptaminitakis2024} that $I_0 x^2 \colon x^{-3/2}L^2(\Dm, \d V_H) \to L^2_+(\Gh)$ is bounded with adjoint $(I_0 x^2)^* = x^{-1} I_0^\sharp$. Combining this with Proposition \ref{prop:intertwiner}, if $f\in C_c^\infty(\Dm^\circ)$, then $I_0 x^2 f \in C_c^\infty(\Gh)$, and applying $I_\perp I_0^\sharp = I_\perp x (I_0 x^2)^*$, we obtain the relation
    \begin{align}
        I_\perp x (I_0 x^2)^* (I_0 x^2) f = 2\pi P_+ I_0 x^2 f, \qquad f\in C_c^\infty(\Dm^\circ). 
        \label{eq:intertwiner2}
    \end{align}
    Since the right hand side extends by boundedness to $f\in x^{-3/2}L^2(\Dm,\d V_H)$, so does the left hand side. Moreover, from \cite[Theorem 2.17]{Eptaminitakis2024}, the operator 
    \begin{align*}
        (I_0 x^2)^* (I_0 x^2) \colon x^{-3/2}L^2(\Dm,  \d V_H)\to H^{1,0}_w (\Dm),
    \end{align*}
    is a homeomorphism, hence $I_\perp x \colon H^{1,0}_w (\Dm) \to L^2_-(\Gh)$ is bounded. From \cite[Theorem 2.17]{Eptaminitakis2024} again, we have, for every $n\ge 0$ and $0\le k\le n$, the relations 
    \begin{align}\label{eq:N0_sp_dec}
        (I_0 x^2)^* (I_0 x^2) \widehat{\Phi^* Z_{n,k}^0} = \frac{4\pi}{n+1} \widehat{\Phi^* Z_{n,k}^0}, \qquad (I_0 x^2) \widehat{\Phi^* Z_{n,k}^0} = \frac{2\sqrt{\pi}}{\sqrt{n+1}} \psi_{n,k}^{H}.
    \end{align} 
    Using those in \eqref{eq:intertwiner2} along with \eqref{eq:P+HSVD}, we obtain \eqref{eq:IperpSVD}.

    To find the domain of the Hilbert space adjoint, recall that
    \begin{equation}
        \D((iI_\perp)^*) = \{u\in L^2_-(\Gh): (iI_\perp \varphi,u)_{L^2(\Gh)}\leq C \|\varphi\|_{x^{-1/2}L^2(\Dm)}\quad \forall \varphi\in xC_\ev^\infty(\Dm)\}.
    \end{equation}
    Expanding $u=\sum_{n=-1}^\infty\sum_{k\in \Zm}a_{n,k} \phi_{n,k}^H$ and $\varphi=\sum_{n\geq 0}\sum_{0\leq k\leq n}b_{n,k}x\widehat {\Phi^* Z_{n,k}^0}$, where $\sum_{n,k}|a_{n,k}|^2<\infty$ and $\sum_{n,k} (n+1)^s|b_{n,k}|^2<\infty$ for any $s$ (see \cite[Theorem 2.17]{Eptaminitakis2024}), we have by \eqref{eq:IperpSVD}
    \begin{equation}
        (iI_\perp \varphi,u)_{L^2(\Gh)}= 2\sqrt\pi \sum_{n\geq 0}\sum_{0\leq k\leq n}\sqrt{n+1}b_{n,k}\overline{a_{n,k}}.
    \end{equation}
    Thus the Cauchy-Schwarz inequality yields the ``$\supset$'' inclusion in \eqref{eq:dom_adj}. For the ``$\subset$'' inclusion, if $u\in \D((iI_\perp)^*)$, by the Riesz representation theorem there exists $z\coloneqq (iI_\perp)^*u\in x^{-1/2}L^2(\Dm)$ such that $(iI_\perp \varphi,u)_{L^2(\Gh)}=( \varphi,z)_{x^{-1/2}L^2(\Dm)}$ for all $\varphi \in xC^\infty_\ev (\Dm)$. Thus if we expand $u,$ $\varphi$ as above, and also write $z = \sum_{n\geq 0}\sum_{0\leq k\leq n}c_{n,k} x\widehat{\Phi^*Z_{n,k}^0}$, since 
        $(\varphi ,z)_{x^{-1/2}L^2(\Dm)}=\sum_{n\geq 0}\sum _{0\leq k\leq n}b_{n,k}\overline{c_{n,k}},$
    we see that $a_{n,k}= \frac{1}{2\sqrt{\pi}\sqrt{n+1}}c_{n,k}$, which yields the estimate $\sum_{n\geq 0}\sum_{0\leq k\leq n}|a_{n,k}|^2(n+1)<\infty$.
    Similarly, $z=0$ exactly when all coefficients of the $\phi_{n,k}^H$-expansion of $u$ vanish for $0\leq n$ and $0\leq k \leq n$.

    To compute the adjoint, for $n,n'\geq 0,$ and $0\leq k\leq n$, $0\leq k'\leq n'$ we have
    \begin{equation}
    \begin{aligned}
       {2\sqrt{\pi}}{\sqrt{n+1}}\delta_{nn'}\delta_{kk'}\overset{\eqref{eq:IperpSVD}}=-\big ( iI \ell_1\star \d   \big(  x\widehat{\Phi^*  Z_{n,k}^0} \big),\phi_{n',k'}^H\big)_{L^2(\Gh)}\hspace{1.7 in}\\*
       =-i\big (  \ell_1\star \d   \big(  x\widehat{\Phi^*  Z_{n,k}^0} \big),I^\sharp \phi_{n',k'}^H\big)_{L^2(\OSD)}
       =\big (      x\widehat{\Phi^*  Z_{n,k}^0} , -i\star \d\ell_1^*I^\sharp \phi_{n',k'}^H\big)_{L^2(\Dm)},%
        \label{eq:Iperpstarpair}
    \end{aligned}
    \end{equation}
    hence for any $u\in \D((iI_\perp)^*)$ and any $\varphi \in x C_\ev^\infty(\Dm)$
    \begin{equation}
        \big (     \varphi,-  ix^{-1}\star \d\ell_1^*I^\sharp u\big)_{x^{-1/2}L^2(\Dm)}= (iI_\perp \varphi,u)_{L^2(\Gh)}\label{eq:last_svd},
    \end{equation}
    and this yields the claimed formula for the adjoint.
    The second equality in \eqref{eq:Iperpstarpair} holds by \eqref{eq:Santalo}, Lemma~\ref{lem:surjective}, and density of $C_c^\infty(S\Dm^\circ)$ in $ x^\delta L^2$, because $\ell_1
    \star \d (x\widehat{\Phi^* Z^0_{n,k}})\in xC^\infty(\OSD)\subset x^\delta L^2(\OSD)$ for any fixed  $\delta\in (0,1/2)$.
    Finally, \eqref{eq:spectral_Iperp_star} follows from \eqref{eq:IperpSVD} and \eqref{eq:Iperpstarpair}.
\end{proof}

\begin{corollary}\label{cor:Iperp_range}
    The operator $I_\perp \colon xH^{1,0}_w(\Dm)\to L^2(\Gh) $ is bounded, and in fact
        \begin{align}
        I_\perp (xH^{1,0}_w(\Dm)) = \bigg\{ \sum_{n\ge 0}\sum_{k=0}^n \frac{a_{n,k}}{\sqrt{n+1}} \phi_{n,k}^H, \qquad \sum_{n,k}|a_{n,k}|^2<\infty \bigg\}. 
        \label{eq:IperpRange}
    \end{align}
\end{corollary}

\subsection{Data space decomposition - proof of Theorem \ref{thm:dataDecomp}}\label{sec:dataDecomp}

We now sketch a proof of Theorem \ref{thm:dataDecomp}, based on intermediate results whose proofs are covered in the next subsections. We first locate the closures of the ranges of $I_0$ and $I_\perp$ in bases \eqref{eq:basis1} and \eqref{eq:basis2} respectively, namely we observe that 
\begin{align}
    \overline{I_0 (x^{1/2}L^2(\Dm))}^{L^2(\Gh)} &= \text{span}_{L^2} \left\{\psi_{n,k}^{H},\ n\ge 0,\ 0\le k\le n\right\}, 
    \label{eq:rangeI0_closure} \\
    \overline{I_\perp (x H^{1,0}_w(\Dm))}^{L^2(\Gh)} &= \text{span}_{L^2} \left\{\phi_{n,k}^{H},\ n\ge 0,\ 0\le k\le n\right\}.
    \label{eq:rangeIperp_closure}
\end{align}
Eq. \eqref{eq:rangeI0_closure}  can be easily derived from the SVD of $I_0$ as presented in, e.g., \cite[Theorem 2.8]{Eptaminitakis2024}, from which it  follows that 
\begin{align*}
    I_0 (x^{1/2}L^2(\Dm)) = I_0 x^2 (L^2(\Dm, x^3 \d V_H)) = \left\{ \sum_{n\ge 0} \frac{2\sqrt{\pi}}{\sqrt{n+1}} \sum_{k=0}^n a_{n,k} \psi_{n,k}^{H}, \quad \sum_{n,k} |a_{n,k}|^2<\infty\right\},
\end{align*}
That \eqref{eq:rangeIperp_closure} is true follows directly from \eqref{eq:IperpRange}. From \eqref{eq:rangeI0_closure} and \eqref{eq:rangeIperp_closure}, we thus deduce
\begin{align}
\begin{split}
    \left(\overline{I_0 (x^{1/2}L^2(\Dm))}^{L^2(\Gh)}\right)^\perp &= E_+ \stackrel{\perp}{\oplus} E_-, \qquad (\text{within } L^2_+ (\Gh))\\
\text{where}\quad E_+ \coloneqq  \text{span}_{L^2} &\left\{\psi_{n,k}^{H},\ n\ge 0,\ k<0 \right\}, \quad E_- \coloneqq  \text{span}_{L^2} \left\{\psi_{n,k}^{H},\ n\ge 0,\ k>n \right\}, 	    
\end{split}
\label{eq:Epm}	\\
\begin{split}
    \left(\overline{I_\perp (xH^{1,0}_w(\Dm))}^{L^2(\Gh)}\right)^\perp &= F_+ \stackrel{\perp}{\oplus} F_-, \qquad (\text{within } L^2_- (\Gh))\\
\text{where}\quad F_+ \coloneqq  \text{span}_{L^2} &\left\{\phi_{n,k}^{H},\ n\ge -1,\ k<0 \right\}, \quad F_- \coloneqq  \text{span}_{L^2} \left\{\phi_{n,k}^{H},\ n\ge -1,\ k>n \right\}, 	    
\end{split}
\label{eq:Fpm}
\end{align}
see Figure \ref{fig:data} for a schematic. Using that $\overline{\psi_{n,k}^{H}} = (-1)^n \psi_{n,n-k}^{H}$ and $\overline{\phi_{n,k}^{H}} = (-1)^{n+1}\phi_{n,n-k}^{H}$ for all $n,k$, we observe that $E_- = \{\overline{\psi},\ \psi \in E_+\}$ and $F_- = \{\overline{\psi},\ \psi \in F_+\}$. Then Theorem \ref{thm:dataDecomp} boils down to proving the following two lemmas: 

\begin{lemma}\label{lem:spans}
The following equalities hold: 
\begin{align}
    E_+ = \mathrm{span}_{L^2} \left\{\widehat{I_{p,2q}},\ p\ge 0,\ q\ge 1 \right\}, \qquad F_+ = \mathrm{span}_{L^2} \left\{\widehat{I_{p,2q+1}},\ p\ge 0,\ q\ge 0 \right\}.
    \label{eq:EpFp}
    \end{align}
\end{lemma}

\begin{lemma}\label{lem:orthonormal}
The two families appearing on the right sides of \eqref{eq:EpFp} are orthonormal families.      
\end{lemma} 

The remainder of the section is organized as follows: Lemmas \ref{lem:spans} and \ref{lem:orthonormal} both rely on a crucial algebraic relation between the $I_{p,q}$ integrals and certain distinguished invariant functions on $S\Dm^\circ$. We first explore these relations in Section \ref{sec:xipm}. We then use them to prove Lemma \ref{lem:spans} in Section \ref{sec:spans}. Finally, in Section \ref{sec:orthonormal}   we cover the proof of Lemma \ref{lem:orthonormal}, using considerations of fiberwise Fourier content on $S\Dm^\circ$ to prove the orthogonality of the ${I_{p,q}}$.

\subsubsection{Distinguished first integrals and an expression for \texorpdfstring{$I_{p,q}$}{Ipq}}\label{sec:xipm}

Consider the geodesically invariant functions $\xi_{\pm} \colon S\Dm^\circ\to \Cm$ defined as the forward $(+)$ and backward $(-)$ endpoints of the geodesic emanating from $(z,\theta)\in S\Dm^\circ$ respectively (in the identification \eqref{eq:osm_param}).
Hence in terms of the identification $(z_{\beta,a}(t), \dot{z}_{\beta,a}(t)) \stackrel{\eqref{eq:osm_param}}{\leftrightarrow} (z_{\beta,a}(t), \theta_{\beta,a}(t))$, and with the help of \eqref{eq:antipodal_sc}, they satisfy
\begin{align*}
    \xi_- (z_{\beta,a}(t), \theta_{\beta,a}(t)) = e^{i\beta}, \qquad \xi_+ (z_{\beta,a}(t), \theta_{\beta,a}(t)) = e^{i(\beta+\pi+2\tan^{-1} a)}, \qquad ( (\beta,a),t)\in \Gh\times \Rm. 
\end{align*}
Using, e.g., \cite[Lemma 3.7]{Eptaminitakis2024} it is easy to check that
\begin{align}
    \xi_\pm (z,\theta) = \frac{z \pm e^{i\theta}}{1\pm \zbar e^{i\theta}}, \qquad (z,\theta)\in \Dm^\circ\times (\Rm/2\pi\Zm) \stackrel{\eqref{eq:osm_param}} \approx S\Dm^\circ,
    \label{eq:xipm}
\end{align}
which shows in particular that $\xi_{\pm}$ are fiberwise holomorphic. Out of $\xi_\pm$, let us define the functions $w,\xi\colon S\Dm^\circ\to \Cm$ via  
\begin{align}
    w \coloneqq  \frac{\xi_+ + \xi_-}{2} = \frac{z-e^{2i\theta} \zbar}{1-\zbar^2 e^{2i\theta}}, \qquad \xi \coloneqq  \frac{\xi_+ - \xi_-}{4} = \frac{c e^{i\theta}}{1-\zbar^2 e^{2i\theta}},
    \label{eq:wxi}
\end{align} 
satisfying for all $t\in \Rm$
\begin{align}
    \begin{split}
	w (z_{\beta,a}(t), \theta_{\beta,a}(t)) &= e^{i (\beta+ \tan^{-1}a + \pi/2)} \cos(\pi/2+\tan^{-1}a) = e^{i (\beta+ \tan^{-1}a+\pi/2)} \frac{-a}{\sqrt{1+a^2}}, \\
	\xi (z_{\beta,a}(t), \theta_{\beta,a}(t)) &= \frac{i}{2} e^{i (\beta+ \tan^{-1}a+\pi/2)} \sin (\pi/2 + \tan^{-1}a) =  e^{i (\beta+ \tan^{-1}a+\pi/2)} \frac{i}{2\sqrt{1+a^2}}.
    \end{split}
    \label{eq:trace_w_xi}
\end{align}

We now show that for any $p\ge 0$ and $q\ge 1$, we can write $I_{p,q}$ (defined in \eqref{eq:Ipq}) as a polynomial of $w,\xi$. 

\begin{proposition}\label{prop:Ipq}
    For any $q\ge 1$ and $p\ge 0$, 
    \begin{align*}
	I_{p,q} = 2\xi^q \sum_{\ell = 0}^{\lfloor p/2 \rfloor} \binom{p}{2\ell} B(\ell+\frac{1}{2},q) (2\xi)^{2\ell} w^{p-2\ell}.
    \end{align*}    
\end{proposition}

To prove Proposition \ref{prop:Ipq}, we will use the following 
\begin{lemma}\label{lem:geodesic-diffeq}
    If $z(t)$ is a unit-speed hyperbolic geodesic in $\Dm^\circ$, with $z_{\pm} = z(\pm\infty)\in \Sm^1$, then
    \[\dot z = -\frac{(z-z_-)(z-z_+)}{z_+-z_-}.\]
\end{lemma}

\begin{proof}[Proof of Lemma \ref{lem:geodesic-diffeq}]
    Let $F(z) = \frac{z-z_-}{z-z_+}$. Then $F$ is a M\"obius transformation sending $z_-$ to $0$ and $z_+$ to $\infty$, so it sends any generalized circle in the complex plane passing through $z_-$ and $z_+$ to a generalized circle passing through $0$ and $\infty$, i.e. a straight line through $0$. Thus, it sends $\mathbb{D}$ to the upper half-plane, up to rotation.
    
    This implies $F$ must send $z(t)$ to the geodesic in the (rotated) upper half-plane starting at $0$ and ending at $\infty$. Thus, if $w(t) = F(z(t))$, then $w(t)$ must have the form $w(t) = w_0e^t$ for some $w_0\in\mathbb{C}$, and hence we must have $\dot w(t) = w(t)$. Written in terms of $z(t)$, this yields $F'(z(t))\dot{z}(t) = F(z(t))$, which gives the desired conclusion. 
\end{proof}

\begin{proof}[Proof of Proposition \ref{prop:Ipq}] By definition \eqref{eq:Ipq}, we have
    \begin{align}
	I_{p,q}(\beta,a) \coloneqq  \int_{\Rm} z_{\beta,a}^p(t) \dot z_{\beta,a}^q(t)\ dt, \quad (\beta,a)\in \Gh.
	\label{eq:Ipq2}
    \end{align}
    Note that for $z(t) = z_{\beta,a}(t)$ as in \eqref{eq:hypgeo}, we have $z_- = e^{i\beta}$ and $z_+ = e^{i(\beta + \pi + 2\tan^{-1} a)}$ so that, with $(w,\xi)$ defined in \eqref{eq:wxi}, we have $w = \frac{z_++z_-}{2}$ and $\xi = \frac{z_+-z_-}{4}$. Using Lemma \ref{lem:geodesic-diffeq}, in \eqref{eq:Ipq2} we replace $\dot{z}^{q-1}$ while keeping the last $\dot{z}$ factor to change variable $t\to z$. This gives 
    \begin{align*}
	I_{p,q} &= \frac{(-1)^{q-1}}{(z_+-z_-)^{q-1}}\int_{z_-}^{z_+} z^p((z-z_-)(z-z_+))^{q-1}\, \d z \\
	&= \frac{1}{(z_+-z_-)^{q-1}}\int_{-1}^1 \left(\frac{z_++z_{-}}{2} + u \frac{z_+-z_-}{2}\right)^p \left( \left(\frac{z_+-z_-}{2}\right)^2 (1-u^2)\right)^{q-1} \frac{z_+-z_-}{2}\ \d u,
    \end{align*}
    where the complex integral is computed using the parameterization $z(u) = \frac{z_++z_{-}}{2} + u \frac{z_+-z_-}{2}$. In terms of the functions $(w,\xi)$, this reads
    \begin{equation}
		\label{eq:ipqwxi}
	I_{p,q} = 2 \xi^q \int_{-1}^1 (w+ 2u \xi)^p (1-u^2)^{q-1}\ \d u = 2\xi^q \sum_{k = 0}^p \binom{p}{k} (2\xi)^k w^{p-k} \int_{-1}^1 u^k (1-u^2)^{q-1}\ \d u.
    \end{equation}
    All coefficients for $k$ odd vanish and hence upon setting $k=2\ell$, 
    \begin{align*}
	I_{p,q} = 2\xi^q \sum_{\ell = 0}^{\lfloor p/2 \rfloor} \binom{p}{2\ell} (2\xi)^{2\ell} w^{p-2\ell} \int_{-1}^1 u^{2\ell} (1-u^2)^{q-1}\ \d u.
    \end{align*}
    The result follows upon observing that $\int_{-1}^1 u^{2\ell} (1-u^2)^{q-1}\ \d u = B(\ell+\frac{1}{2},q)$. 
\end{proof}

\subsubsection{Proof of Lemma \ref{lem:spans}}\label{sec:spans}

Lemma \ref{lem:spans} is proved by exploring each eigensubspace of the operator $D_\beta$. Indeed, since $D_\beta \psi_{n,k}^{H} = (n-2k) \psi_{n,k}^{H}$ and $D_\beta \phi_{n,k}^{H} = (n-2k) \phi_{n,k}^{H}$, we have the decompositions
\begin{align}
    E_+ &= \bigoplus_{r\ge 2} E_r,\ \text{where}\  E_r \coloneqq  E_+ \cap \ker_{L^2} (D_\beta-r) = \text{span} \{\psi_{r-2p,-p}^{H},\ 1\le p\le \lfloor r/2 \rfloor\},
    \label{eq:Er} \\
    F_+ &= \bigoplus_{r\ge 1} F_r,\  \text{where}\  F_r \coloneqq  F_+ \cap \ker_{L^2} (D_\beta-r) = \text{span} \{\phi_{r-2p,-p}^{H},\ 1\le p\le \lfloor (r+1)/2 \rfloor\}, \qquad
    \label{eq:Fr}    
\end{align}
in particular, $\dim E_r = \lfloor r/2\rfloor$ and $\dim F_r = \lfloor (r+1)/2\rfloor$. See Figure \ref{fig:data} for a visualization of these decompositions.

\begin{figure}[htpb]
    \begin{center}
	\includegraphics[width=0.48\textwidth]{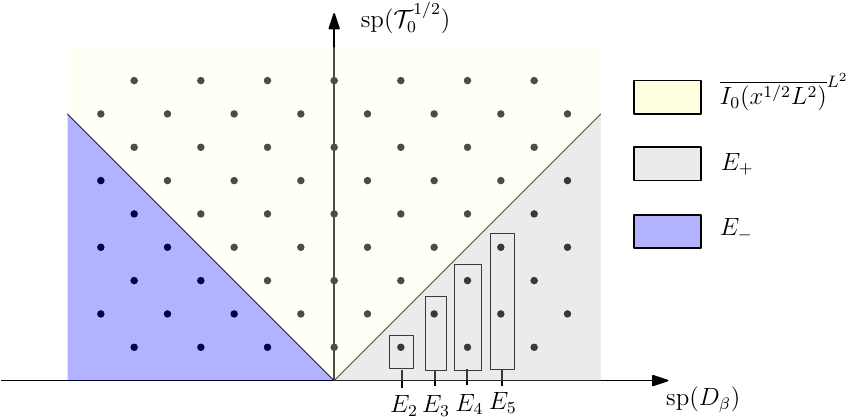}
    \includegraphics[width=0.48\textwidth]{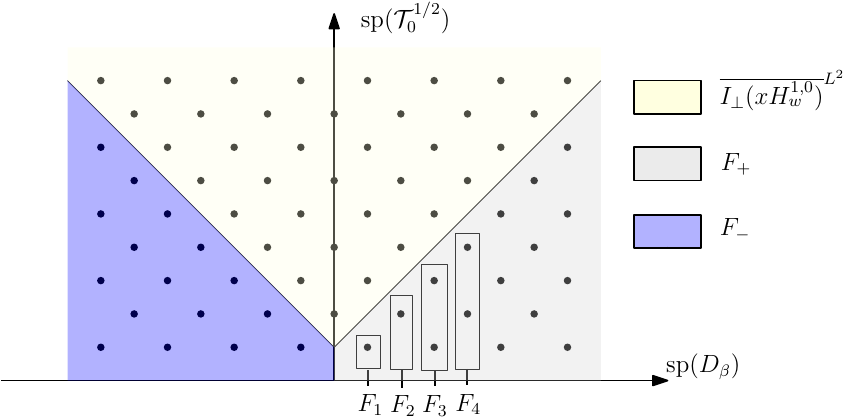}
    \end{center}
        \captionsetup{width=0.8\textwidth}
    \caption{A spectral representation of the data spaces $L^2_+(\Gh)$ (left) and $L^2_-(\Gh)$ (right). Each ``$\bullet$'' represents an element $\psi_{n,k}^{H}$ (left) or $\phi_{n,k}^H$ (right). The spaces $E_{\pm}$, $E_r$, $F_\pm$ and $F_r$ are defined in \eqref{eq:Epm}, \eqref{eq:Er}, \eqref{eq:Fpm} and \eqref{eq:Fr}, respectively. 
    }

    \label{fig:data}
\end{figure}

\begin{lemma}\label{lem:psink}\phantom{a}
\begin{enumerate}[(a)]
    \item \label{item_psink1}For any $n\in \Nm_0$ and $k>0$, $\psi_{n,-k}^{H}$ is a homogeneous polynomial of degree $n+2k$ in $(\xi,w)$, even in $\xi$, and factoring out a $\xi^2$. 
    
    \item \label{item_psink2}For any $n\ge -1$ and $k>0$, $\phi_{n,-k}^{H}$ is a homogeneous polynomial of degree $n+2k$ in $(\xi,w)$, odd in $\xi$.     
\end{enumerate}
\end{lemma}

\begin{proof} Below, we use the shorthands $\xi = \frac{i}{2} e^{i\omega} \muh$, $\frac{a}{\sqrt{1+a^2}} = -w e^{-i\omega}$ and $e^{2i\omega} = w^2-4\xi^2$, where $\omega = \beta + \tan^{-1}a + \pi/2$. Then an important observation is that, up to scalar multiples, we have 
\begin{align*}
    \psi_{n,-k}^H \propto \muh^2 e^{i(n+2k)\omega} U_n \left(\frac{a}{\sqrt{1+a^2}}\right), \qquad \phi_{n,-k}^H \propto \muh e^{i(n+2k)\omega} T_{n+1} \left(\frac{a}{\sqrt{1+a^2}}\right),
\end{align*}
where $T_n,U_n$ denote the Chebyshev polynomials of the first and second kind\footnote{Uniquely defined by the relations $T_n(\cos\theta) = \cos(n\theta)$ and $U_{n}(\cos\theta)\sin\theta = \sin ((n+1)\theta)$ for $n\in \Nm_0$.}, respectively. To prove \ref{item_psink1}, recall that $U_n$ is a polynomial of degree $n$ with the same parity as $n$, i.e. of the form $U_n (t) = \sum_{q=0}^{\lfloor n/2 \rfloor} a_{n,q} t^{n-2q}$ for some coefficients $a_{n,q}$. Then the above expression of $\psi_{n,-k}^{H}$ becomes 
    \begin{align*}
	\psi_{n,-k}^{H} \propto \muh^2 e^{i(n+2k)\omega} \sum_{q=0}^{\lfloor n/2 \rfloor} a_{n,q} \left( -w e^{-i\omega} \right)^{n-2q} = -4(-1)^n \xi^2 \sum_{q=0}^{\lfloor n/2 \rfloor} a_{n,q} w^{n-2q} (w^2-4\xi^2)^{q+k-1},
    \end{align*}
    hence the result. The proof of \ref{item_psink2} is similar. 
\end{proof}

\begin{lemma}
    Fix $r = n+2k>0$. Then 
    \begin{enumerate}[(a)]
        \item\label{item_l_4131} the four linear spans below are equal and describe the subspace $E_r$ defined in \eqref{eq:Er}: 
        
    (1) $\psi^{H}_{r-2p,-p}$ for $1\le p\le \lfloor r/2\rfloor$. 

    (2) $I_{r-2p,2p}$ for $1\le p\le \lfloor r/2\rfloor$. 

    (3) $\xi^{2p} w^{r-2p}$ for $1\le p\le \lfloor r/2\rfloor$. 

    (4) $\xi^{2} w^{r-2p} (w^2-4\xi^2)^{p-1}$ for $1\le p\le \lfloor r/2\rfloor$. 

    \item\label{item_l_4132} The four linear spans below are equal and describe the subspace $F_r$ defined in \eqref{eq:Fr}: 

    (1$\,'$) $\phi^{H}_{r-2p,-p}$ for $1\le p\le \lfloor (r+1)/2\rfloor$. 

    (2$\,'$) $I_{r+1-2p,2p-1}$ for $1\le p\le \lfloor (r+1)/2\rfloor$. 

    (3$\,'$) $\xi^{2p-1} w^{r+1-2p}$ for $1\le p\le \lfloor (r+1)/2\rfloor$. 

    (4$\,'$) $\xi w^{r-2p+1} (w^2-4\xi^2)^{p-1}$ for $1\le p\le \lfloor (r+1)/2\rfloor$. 
    \end{enumerate}
    \label{lem:linspans}
\end{lemma}

\begin{proof} We prove \ref{item_l_4131}, then the proof of \ref{item_l_4132} is similar.
    That (3)=(4) follows from the fact that (4) admits a triangular representation in the span (3). 

    (2)=(3) follows from Proposition \ref{prop:Ipq}, notably that $\{I_{r-2p,2p},\ 1\le p\le \lfloor r/2\rfloor\}$ admits a triangular representation in the span (3). 

    (1)=(4): follows from Lemma \ref{lem:psink}, notably that $\{\psi^{H}_{r-2p,-p},\ 1\le p\le \lfloor r/2\rfloor\}$ admits a triangular representation in the span (4).     
\end{proof}

We can then conclude with the proof of Lemma \ref{lem:spans}.

\begin{proof}[Proof of Lemma \ref{lem:spans}]
From Lemma \ref{lem:linspans}, we see that for every $p\ge 0$ and $q\ge 1$, $I_{p,2q}$ is a finite linear combination over those $\psi^{H}_{n,k}$'s that satisfy $p+2q = n-2k$ with $k<0$, and vice-versa. Similarly, for every $p\ge 0$ and $q\ge 0$, $I_{p,2q+1}$ is a finite linear combination over those $\phi^{H}_{n,k}$'s that satisfy $p+2q+1 = n-2k$ with $k<0$, and vice-versa. Hence the result.
\end{proof}

\subsubsection{Orthogonality properties - proof of Lemma \ref{lem:orthonormal}} \label{sec:orthonormal}

 While the explicit expression computed in Proposition \ref{prop:Ipq} could give a direct (yet tedious) computation of the inner products $\int_{\Gh} I_{p,q} \overline{I_{p',q'}}\ \d \beta \d a$, we show below that such pairings can be computed by moving them from $\Gh$ to $S\Dm^\circ$ via Santal\'o's formula and exploiting fiberwise Fourier content. To do this in the next result below, Proposition \ref{prop:Ipq} provides the key insight that $I_{p,q}$ is a homogeneous polynomial of degree $p+q$ in $(w,\xi)$ with the same parity in $\xi$ as that of $q$. Since $\xi = c e^{i\theta} + O(e^{3i\theta})$ and $w = z + O(e^{2i\theta})$, where for $k\in \Zm$, $O(e^{ik\theta})$ denotes a smooth function on $S\Dm^\circ$ all of whose Fourier modes of order $<k$ vanish, Proposition \ref{prop:Ipq} gives us that
\begin{align}
    I_{p,q}\circ\pih = 2B\Big(\frac{1}{2},q\Big) z^p c^q e^{iq\theta} + O(e^{i(q+2)\theta}).  
    \label{eq:Ipq_leading}
\end{align}
The following result in particular justifies that the $I_{p,q}$'s belong to $L^2(\Gh)$, and that the conclusion of Lemma \ref{lem:orthonormal} holds.

\begin{proposition}\label{prop:Ipq_ortho}
    Fix $p,p'\ge 0$, $q,q'\ge 1$. Then
\begin{enumerate}[(a)]
    \item\label{ortho_item_1} If $(p,q)\ne (p',q')$, then $\int_{\Gh} I_{p,q} \overline{I_{p',q'}}\ \d \beta \d a = 0$.

    \item\label{ortho_item_2} If $(p,q) = (p',q')$, we have $\int_{\Gh} |I_{p,q}|^2 \d \beta\d a  = 4\pi^2 B\left( \frac{1}{2},q \right) \frac{1}{4^{q-1}} B(p+1, 2q-1)$.

    \item \label{ortho_item_3}For any $\delta>0$ and $f\in x^\delta L^2(\Dm)$, $\left(I_{p,q}, I_0 f \right)_{L^2(\Gh)}= 0$. 

    \item \label{ortho_item_4} For any $\delta>0$ and $f\in x^\delta H^1_0(\Dm)$, $\left(I_{p,q}, I_1 (\star \d f) \right)_{L^2(\Gh)}= 0$. 
    
\end{enumerate}
\end{proposition}

\begin{proof} \ref{ortho_item_1} 
    If $q=q'$, then $p\ne p'$. In that case, noticing that $D_\beta I_{p,q} = (p+q) I_{p,q}$ (this follows for instance from Prop. \ref{prop:Ipq} and the fact that $D_\beta w = w$ and $D_\beta \xi = \xi$) and similarly for $I_{p',q'}$, since $p+q\ne p'+q'$, orthogonality follows directly from the $\beta$ integral.

    Now suppose $q<q'$ without loss of generality. Applying Santal\'o's formula to the expression $z^p c^q e^{iq\theta} \overline{I_{p',q'}\circ \pih}$, we obtain %
    \begin{align*}
    	\int_{\Gh} I_{p,q} \overline{I_{p',q'}}\ \d \beta\d a \stackrel{\eqref{eq:Santalo}}{=} \int_{S\Dm^\circ} z^p c^q e^{iq\theta} \overline{I_{p',q'}\circ \pih}\ d\Sigma^3 = \int_{\Dm^\circ} z^p c^q \int_{0}^{2\pi} e^{iq\theta} \overline{I_{p',q'}\circ \pih}\ d\theta\ dV_H, 
        \end{align*} 
    and by virtue of \eqref{eq:Ipq_leading}, the last $\theta$ integral is zero since $I_{p',q'}\circ \pih$ only has Fourier content in the modes greater than or equal to $q'$. 

    \ref{ortho_item_2} Following a similar argument as in \ref{ortho_item_1}, if $(p,q)=(p',q')$, we compute
    \begin{align*}
	\int_{\Gh} |I_{p,q}|^2 \d \beta\d a \stackrel{\eqref{eq:Santalo}}{=} \int_{\Dm^\circ} z^p c^q \int_{0}^{2\pi } e^{iq\theta} \overline{I_{p,q}\circ \pih}\ d\theta\ \d V_H \stackrel{\eqref{eq:Ipq_leading}}{=} 4\pi B\left( \frac{1}{2},q \right) \int_{\Dm^\circ} |z|^{2p} c^{2q} \d V_H,
    \end{align*}
    and the result follows upon observing that 
    \begin{align}
	\int_{\Dm^\circ} |z|^{2p} c^{2q} \d V_H = \frac{2\pi}{4^{q-1}} \int_0^1 r^{2p} (1-r^2)^{2(q-1)}\ r \d r = \frac{\pi}{4^{q-1}} B(p+1, 2q-1).
	\label{eq:temp}
    \end{align}

    \ref{ortho_item_3} With $\delta, f$ as in the statement, Lemma \ref{lem:surjective} guarantees that $I_0 f\in \muh^\delta L^2(\Gh)\subset L^2(\Gh)$ and hence the pairing $\left(I_{p,q}, I_0 f \right)_{L^2(\Gh)}$ makes sense. Similarly to \ref{ortho_item_1}, apply Santal\'o's formula to $f\ \overline{I_{p,q}\circ \pih}$ to make appear
    \begin{align*}
	\int_{\Gh} I_0 f\ \overline{I_{p,q}}\ \d\beta\d a \stackrel{\eqref{eq:Santalo}}{=} \int_{S\Dm^\circ} f\ \overline{I_{p,q}\circ \pih}\ \d\Sigma^3 = \int_{\Dm^\circ} f \int_{0}^{2\pi } \overline{I_{p,q}\circ \pih}\ d\theta\ \d V_H = 0,
    \end{align*}
    again, by considerations of fiberwise Fourier content. 

    \ref{ortho_item_4} With $\delta, f$ as in the statement, Lemma \ref{lem:surjective} guarantees that $I_1(\star\d f)\in \muh^\delta L^2(\Gh)\subset L^2(\Gh)$ and hence the pairing $\left(I_{p,q}, I_1(\star\d f) \right)_{L^2(\Gh)}$ makes sense. Applying Santal\'o's formula to $-X_\perp f\ \overline{I_{p,q}\circ \pih}$, we arrive at 
    \begin{align*}
        \int_{\Gh} I_1(\star \d f) \overline{I_{p,q}}\ \d\beta\d a \stackrel{\eqref{eq:Santalo}}{=} -\int_{S\Dm^\circ} (X_\perp f) \overline{I_{p,q}\circ \pih}\ \d\Sigma^3 &= i (\eta_+ f, \mathcal{P}_1(I_{p,q}\circ\pih))_{L^2(\OSD)} \\
        &\stackrel{(\diamond)}{=} -i(f, \eta_- \mathcal{P}_{1}(I_{p,q}\circ\pih))_{L^2(\OSD)} = 0,
    \end{align*}
    where step $(\diamond)$ is an integration by parts with zero boundary term, which can be justified beginning with $f\in C_c^\infty(\Dm^\circ)$, then extending to $f\in x^\delta H^1_0(\Dm)$ by density.
\end{proof}

\subsection{Range characterization: proof of Theorem \ref{thm:moments}}\label{sec:rangeCharac}

\subsubsection{The X-ray transform of \texorpdfstring{$x^\delta L^2(\Dm; \Stt^{k} ({}^0 T^* \Dm))$}{tt tensors}: proof of Proposition \ref{prop:IStt_2k}} 

Recall from \eqref{eq:ttdecom} that for $k\ge 1$ and $\delta>0$,
\begin{align}
    \ell_k \colon x^\delta L^2(\Dm; \Stt^{k} ({}^0 T^* \Dm)) \longrightarrow (\Omega_{-k}^{x^\delta L^2}\cap \ker \eta_+) \oplus (\Omega_{k}^{x^\delta L^2}\cap \ker \eta_-)
    \label{eq:lmStt}
\end{align}
is an isomorphism, and moreover that the first summand on the right side is obtained by complex-conjugating elements of the second summand. With this reduction in mind, in the result below we first characterize the spaces $\Omega_{k}^{x^\delta L^2}\cap \ker \eta_-$ for $k\ge 1$, then study their images under the X-ray transform. 

\begin{lemma}\label{lm:conv}
    Fix $k\ge 1$ and $\delta\in (0,k-1/2)$. 
    \begin{enumerate}[(a)]
        \item\label{conv_item_1} Then $\hat g_k\in \Omega_k ^{x^\delta L^2} \cap \ker \eta_-$ if and only if 
    \begin{equation}\label{eq:a_p}
    \hat g_k=c^k e^{i k\theta} \sum_{p=0}^{\infty}a_p z^{p}\qquad \text{with} \qquad \sum _{p=0}^\infty |a_p|^2 B(2k-1-2\delta,p+1)<\infty.
    \end{equation}
    \item\label{conv_item_2}  The restriction of the X-ray transform below is a homeomorphism: 
    \begin{align*}
    I\colon \Omega_k ^{x^\delta L^2} \cap \ker \eta_- \longrightarrow h^\delta ( \widehat{I_{p,k}},\ p\ge 0).
    \end{align*}

    \end{enumerate}
\end{lemma}

\begin{proof}
    \ref{conv_item_1} Since $\partial_{\zbar} (c^{-k}{g}_k) = 0$ by \eqref{eq:eta_pm_isoth},  $c^{-k}{g}_k$ is holomorphic and we have ${g}_k=c^k\sum_{p=0}^{\infty}a_p z^{p}$, where the series converges to ${g}_k$ uniformly on compact subsets of $\Dm^\circ$. 
    We thus compute
    \begin{align*}
	\infty>\int_{\Dm}c^{-2\delta} |{g}_k|^2 \d V_H = \lim_{R\to 1^-} \int_{|z|\le R} c^{-2\delta} |{g}_k|^2 \d V_H &= \lim_{R\to 1^-} \sum_{p=0}^\infty |a_p|^2 \int_{|z|\le R} c^{2k-2\delta} |z|^{2p} \d V_H \\
	&= \sum_{p=0}^\infty |a_p|^2 \int_{\Dm} c^{2k-2\delta} |z|^{2p} \d V_H,
    \end{align*}
    where the first and last equalities use monotone convergence, and the middle one uses uniform convergence. Finally, using \eqref{eq:temp}, we arrive at 
    \begin{align}
	\int_{\Dm}c^{-2\delta} |{g}_k|^2 \d V_H = 2^{-2k+2\delta+2} \pi \sum_{p=0}^\infty |a_{p}|^2 B(2k-2\delta-1,p+1). 
	\label{eq:sum_convergence}
    \end{align}

    \ref{conv_item_2} Writing $\hat{g}_k$ as the $x^\delta L^2$-convergent series \eqref{eq:a_p}, by continuity of $I\colon x^\delta L^2 (\OSM) \to L^2(\Gh)$ we have $I \hat{g}_k = \sum_{p=0}^\infty a_p\ I_{p,k}$ at least defined as an $L^2(\Gh)$-convergent series. Using Proposition \ref{prop:Ipq_ortho}, we compute the squared $h^s( \widehat{I_{p,k}},\ p\ge 0)$ norm
    \begin{align*}
	\| I\hat{g}_k \|_{s}^2 = \sum_{p=0}^\infty |a_{p}|^2 \|I_{p,k}\|_{L^2}^2 (p+1)^{2s} = \frac{4\pi^2 B(1/2,k)}{4^{k-1}} \sum_{p=0}^\infty |a_{p}|^2 B(p+1,2k-1) (p+1)^{2s}.
    \end{align*}
    We now use the asymptotics $\Gamma(x+\alpha) \sim \Gamma(x) x^{\alpha}$ as $x\to +\infty$ (valid for any fixed $\alpha\in \Cm$) to deduce that $B(x,y) = \frac{\Gamma(x) \Gamma(y)}{\Gamma(x+y)} \sim \Gamma(y) x^{-y}$ as $x\to \infty$. Then as $p\to \infty$,  
    \begin{align*}
	B(2k-2\delta-1, p+1) &\sim \Gamma(2k-2\delta-1) (p+1)^{1-2k+2\delta}, \quad (2k-2\delta-1>0), \\
	B(p+1,2k-1) (p+1)^{2s} &\sim \Gamma(2k-1) (p+1)^{1-2k+2s}, \quad (2k-1>0),
    \end{align*}
    while none of the coefficients vanish at any $p$. In particular, the two norms $\|I \hat{g}_k\|_s$ and $\|\hat{g}_k\|_{x^\delta L^2}$ (computed in \ref{conv_item_1}) are equivalent if and only if $s = \delta$. 
\end{proof}

\begin{remark}\label{rm:holodiff}
    Note that the characterization \eqref{eq:a_p} also means that $\hat{g}_k = \ell_k(g_k \d z^k)$, where $g_k \d z^k$ is a holomorphic $k$-differential. Hence \eqref{eq:lmStt} also implies that $x^\delta L^2(\Dm; \Stt^{k} ({}^0 T^* \Dm))$ is in 1:1 correspondence with sums of $x^\delta L^2$-integrable holomorphic $k$-differentials and their complex conjugates.     
\end{remark}

\begin{proof}[Proof of Proposition \ref{prop:IStt_2k}] This is now a direct consequence of \eqref{eq:lmStt}, Lemma \ref{lm:conv}.\ref{conv_item_2}, and the fact that the images of each summand in the right-hand side of \eqref{eq:lmStt} are orthogonal by virtue of Theorem \ref{thm:dataDecomp}.
\end{proof}

\subsubsection{Proof of Theorem \ref{thm:moments}}

We cover the case $m$ even, as the proof for odd $m$ is completely similar. Suppose $m=2n$ for $n$ a natural integer, and let $\delta\in (0,3/2)$.

($\implies$) Suppose $u = I_{2n} f$ for some $f\in x^\delta L^2(\Dm, S^{2n}({}^0 T^* \Dm))$, which we may assume in iterated-tt form \eqref{eq:gauge_decomp}. Then by Proposition \ref{prop:Ipq_ortho}, condition \ref{range_item1} follows. Condition \ref{range_item2} is a direct consequence of Proposition \ref{prop:IStt_2k}. Finally, since $\tilde{f}_0\in x^\delta L^2 (\Dm)$,  by Proposition \ref{prop:Ipq_ortho}.\ref{ortho_item_3}, $I_0 \tilde{f}_0 \perp \sum_{k=1}^n I_{2k} \tilde{f}_{2k}$ and hence $\Pi_0 h = I_0 \tilde{f}_0$.

($\impliedby$) Suppose $h\in L^2_+(\Gh)$ satisfies \ref{range_item1}, \ref{range_item2} and \ref{range_item3}. Then, 
\begin{align*}
    h = \Pi_0 h + \sum_{k=1}^\infty \Pi_{2k}h \stackrel{\text{\ref{range_item1}}}{=} \Pi_0 h + \sum_{k=1}^n \Pi_{2k}h.
\end{align*}
By \ref{range_item2}, for each $1\le k\le n$, $\Pi_{2k} h$ belongs to $h^\delta (\widehat{I_{p,2k}}, p\ge 0)\oplus h^\delta (\widehat{I_{p,2k}}, p\ge 0)$ and by Proposition \ref{prop:IStt_2k} there exists a unique $\tilde{f}_{2k}\in x^\delta L^2(\Dm, \Stt^{2k}({}^0 T^* \Dm))$ such that $I_{2k} \tilde{f}_{2k} = \Pi_{2k} h$. By \ref{range_item3}, $\Pi_0 h = I_0 \tilde{f}_0$ for some $\tilde{f}_0 \in x^\delta L^2(\Dm)$. Finally, by linearity, $h = I_0 \tilde{f}_0 + \sum_{k=1}^n I_{2k} \tilde{f}_{2k} = I_{2n} f$ where $f \coloneqq  \sum_{k=0}^n L^{n-k} \tilde{f}_{2k}$, an element of $x^\delta L^2(\Dm, S^{2n}({}^0 T^* \Dm))$. Theorem \ref{thm:moments} is proved.

\subsection{Reconstructions} \label{sec:recons}

\subsubsection{Proof of Theorems \ref{thm:reconsXray} and \ref{thm:recons_m=1} and discussion} \label{sec:reconsXray}

\begin{proof}[Proof of Theorem \ref{thm:reconsXray}] With notations as in the statement, we first reconstruct $f_{m,+} = \sum_{p=0}^\infty b_{p} z^p$, then the reconstruction formula for $f_{m,-}$ can be deduced by complex conjugation. 

The recovery of $b_{p}$ will be done by computing the integral ${\cal J} := \int_{\OSD} \ell_m f\cdot \overline{w^p \xi^m} \d\Sigma^3$ in two ways, where $w,\xi\in C^\infty(S\Dm^\circ)$ are defined in \eqref{eq:wxi}, and we abuse notation slightly by not distinguishing between $w,\xi$ and their restrictions to $\Gh$. To this end, we first contend that $\ell_m f \cdot \overline{w^p\xi^m} \in L^1(\OSD,\d\Sigma^3)$. Indeed, using that $\ell_m f \in x^\delta L^2(\OSD)$ and using Cauchy-Schwarz inequality, we write 
\begin{align*}
    \int_{\OSD} |\ell_m f\cdot \overline{w^p\xi^m}| \d\Sigma^3 &\le \| x^{-\delta}\ell_m f\|_{L^2(\OSD)}
    \left(\int_{\OSM} x^{2\delta} |w^p\xi^m|^2 \d\Sigma^3\right)^{1/2} \\
    &\le \| x^{-\delta}\ell_m f\|_{L^2(\OSD)} \left(\int_{\Gh} (Ix^{2\delta}) |w|^{2p} |\xi|^{2m} \d\beta\d a\right)^{1/2}  \\
    &\le \| x^{-\delta}\ell_m f\|_{L^2(\OSD)} B(\delta,1/2)^{1/2} \left(\int_{\Gh} |w|^{2p} |\xi|^{2m} \muh^{2\delta}\d\beta\d a\right)^{1/2},
\end{align*}
where the last inequality follows using Remark~\ref{rmk:phg_hyp}, and the last integral is finite since $m\ge 1$ and $|w|^{2p} |\xi|^{2m} = O(\muh^{2m})$ by \eqref{eq:trace_w_xi}. In addition, observe that, by consideration of fiberwise Fourier contents on the product $w^p \xi^m$, we have that ${\cal P}_m (w^p \xi^m) = e^{im\theta} c^m z^p$ and ${\cal P}_k (w^p \xi^m) = 0$ for all $k<m$. Thus, since ${\cal P}_k (\ell_m f) = 0$ for $k>m$, 
\begin{align*}
    {\cal P}_0 (\ell_m f\cdot \overline{w^p \xi^m}) = \sum_{k\in \Zm} \mathcal{P}_k(\ell_m f)\ \overline{\mathcal{P}_k(w^p \xi^m)} = \mathcal{P}_m(\ell_m f)\ \overline{\mathcal{P}_m(w^p \xi^m)} = c^m f_{m,+} \cdot c^m \bar{z}^p.
\end{align*}
On to the first computation of ${\cal J}$, using the last display, we also have, if $m\geq 2$,
\begin{equation}
    \begin{aligned}
        {\cal J} &= 2\pi \int_{\Dm} {\cal P}_0 (\ell_m f\cdot \overline{w^p \xi^m})\ \d V_H \stackrel{(\star)}{=} 2\pi \int_{\Dm} (\sum_{q=0}^\infty b_q z^q)c^{2m} \zbar^p\ \d V_H \\
        &= 2\pi b_p \int_{\Dm} c^{2m} |z|^{2p}\ \d V_H \stackrel{\eqref{eq:temp}}{=} \frac{\pi^2}{2^{2m-3}} B(p+1,2m-1) b_{p} = \frac{\pi^2}{2^{2m-3}} \frac{(2m-2)! p!}{(2m + p-1)!} b_{p}.
    \end{aligned}\label{eq:ibp}
    \end{equation}
    On the other hand, by Santal\'o's formula, we have
    \begin{align*}
        {\cal J} = \!\! \int_{\Gh} \! (I f)\ \overline{w^p \xi^m}\ \d\beta\d a \stackrel{\eqref{eq:trace_w_xi}}{=}\!\! \int_{\Gh}\! (I f)   \frac{e^{-i(p+m)(\beta + \tan^{-1}a+\pi/2)} (-a)^p (-i)^m}{2^m (1+a^2)^{(p+m)/2}} \d\beta\d a.
    \end{align*}
    Combining the last two displays, we obtain 
    \begin{align}
        b_{p} = \frac{2^{m-3}}{(2m-2)!\pi^2} \frac{(2m+p-1)!}{p!} \int_{\Gh} (I f)  \frac{ e^{-i(p+m)(\beta + \tan^{-1}a+\pi/2)}(-a)^p (-i)^m}{(1+a^2)^{(p+m)/2}} \d\beta\d a.
	\label{eq:bprecons}
    \end{align}    
    Then \eqref{eq:reconsfk2} follows by plugging \eqref{eq:bprecons} back into the definition of $f_{m,+}$ and defining $G_{m}(\beta,a;z)$ along the way.
    The case $m=1$ follows similarly, observing in addition that  
    $(\star)$ in \eqref{eq:ibp} holds because $(\ell_1 \star \d f_s, w^p \xi)_{L^2(\OSD)}= ( I_1(\star \d f),I_{p,1})_{L^2(\Gh)} =0,$ by Proposition \ref{prop:Ipq_ortho}\ref{ortho_item_4}.
    \end{proof}

\begin{remark}
    The reconstruction formula \eqref{eq:reconsfk2} is not unique. Indeed, in the argument above, one could perturb the first integral $w^p \xi^{m}$ with a first integral whose Fourier support lies in the modes $\{k\colon k>n\}$. This may change the integral kernel without changing the outcome. 
\end{remark}

\begin{proof}[Proof of Theorem~\ref{thm:recons_m=1}]

By virtue of Lemma \ref{lem:Ppmadjoints}.\ref{item_pmadj2} and \cite[Corollary 5.3]{Eptaminitakis2024}, we have
\begin{align*}
    (A_{+} ^{H})^{*}\mathcal{H}_+ A_-^{H} \phi_{n,k}^H = P_-^H \phi_{n,k}^H = -2i \psi_{n,k}^H,\qquad n\geq 0,\quad 0\leq k\leq n.
\end{align*}
 Combining this with \eqref{eq:IperpSVD} and \cite[Theorem 2.8]{Eptaminitakis2024} yields
 \begin{equation}
 \begin{aligned}
    -\frac{1}{8\pi}I_0^{\sharp
    } (A_+^H)^{*}\mathcal{H}_+A_-^{H}I_\perp x  \widehat{\Phi^* Z_{n,k}^0}
    =
    \frac{\sqrt{n+1}}{2\sqrt{\pi}} I_0^{\sharp
    } \psi_{n,k}^{H}
    =\frac{\sqrt{n+1}}{2\sqrt{\pi}} \frac{2\sqrt \pi}{\sqrt{n+1}}x   \widehat{\Phi^* Z_{n,k}^0} 
    =x   \widehat{\Phi^* Z_{n,k}^0} .
 \end{aligned}\label{eq:on_basis_elements}
 \end{equation}
Since $\big\{x\widehat{\Phi^* Z_{n,k}^0}\big\}_{n\geq 0,\ 0\leq k\leq n}$ are an orthonormal basis of $x^{-1/2}L^2(\Dm)$, \eqref{eq:PU_inversion} holds there.
In particular, it holds a.e. and for elements of  $x^\delta  H_0^{s+1}(\Dm)$ for $\delta\geq -1/2$ and $s \geq -1$.
\end{proof}

\paragraph{Reconstruction of $\tilde f_0$.}

For any $\delta>0$, the reconstruction of $\tilde{f}_0\in x^\delta L^2(\Dm)$ from $I_0 \tilde{f}_0$ (or equivalently, from $\Pi_0 I_{2n} f$ where $f\in x^\delta L^2(\Dm, S^m ({}^0 T^* \Dm))$ with $\delta\in (0,3/2)$ and $\tilde{f}_0$ is the $0$-component of $f^\itt$) can generally leverage the Singular Value Decompositions for $I_0$ derived in \cite{Eptaminitakis2024}. 

\begin{proposition}
    For any $\delta >0$ and $\tilde{f}_0\in x^\delta L^2(\Dm)$, we have
    \begin{align}
        \tilde f_0 = x^{2+2\gamma} \sum_{n\ge 0} \sum_{k=0}^n (\sigma_{n,k}^\gamma)^{-1} (I_0 \tilde{f}_0, \muh^{-2\gamma} \psi_{n,k}^{\gamma,H})_{L^2(\Gh)} \widehat{\Phi^* Z_{n,k}^\gamma}, \qquad \gamma = \delta-\frac{1}{2},
        \label{eq:f0recons}
    \end{align}    
    where $(\Phi^* Z_{n,k}^\gamma, \psi_{n,k}^{\gamma,H}, \sigma_{n,k}^\gamma)$ are defined in \cite[Theorem 2.8]{Eptaminitakis2024}.    
\end{proposition}

\begin{proof}
    To bridge notation, we set $\gamma = \delta-1/2$, and notice that $\tilde{f}_0 \in x^\delta L^2(\Dm)$ if and only if $x^{-2-2\gamma} \tilde{f}_0\in L^2 (\Dm, x^{2\gamma+3}\d V_H)$, while $I_0 \tilde{f}_0 = I_0 x^{2+2\gamma} (x^{-2-2\gamma} \tilde{f}_0)$. This justifies that we should leverage the SVD $\left(\widehat{\Phi^* Z_{n,k}^\gamma}, \psi_{n,k}^{\gamma,H}, \sigma_{n,k}^\gamma\right)_{n\ge 0,\ 0\le k\le n}$ of 
    \[ I_0 x^{2+2\gamma} \colon L^2 (\Dm, x^{2\gamma+3}\d V_H) \to L^2(\Gh, \muh^{-2\gamma} \d\Sigma_\partial).  \]
    To wit, $x^{-2-2\gamma} \tilde{f}_0$ is reconstructed by inverting this SVD, namely writing
    \begin{align*}
        x^{-2-2\gamma} \tilde{f}_0 = \sum_{n\ge 0} \sum_{k=0}^n (\sigma_{n,k}^\gamma)^{-1} (I_0 x^{2+2\gamma}(x^{-2-2\gamma} \tilde{f}_0), \psi_{n,k}^{\gamma,H})_{L^2(\Gh,\muh^{-2\gamma}\d\Sigma_\partial)} \widehat{\Phi^* Z_{n,k}^\gamma},
    \end{align*}
    then \eqref{eq:f0recons} follows upon multiplying both sides by $x^{2+2\gamma}$, and using that 
    \begin{align*}
        (I_0 x^{2+2\gamma}(x^{-2-2\gamma} \tilde{f}_0), \psi_{n,k}^\gamma)_{L^2(\Gh,\muh^{-2\gamma}\d\Sigma_\partial)} = (I_0 \tilde{f}_0, \muh^{-2\gamma} \psi_{n,k}^{\gamma,H} )_{L^2(\Gh)}.
    \end{align*}
\end{proof}

The case $\delta=1/2$ presents some functional advantages, notably the existence of intertwining differential operators, which allow for the derivation of a new reconstruction  given below. This in particular gives a pointwise reconstruction approach for $\tilde f_0\in x^\delta L^2(\Dm)$ for any $\delta\ge 1/2$. In the statement below, $\T_0$ refers to the self-adjoint realization of the operator defined in \eqref{eq:T0}, with spectral decomposition \eqref{eq:T0spectrum}.
\begin{proposition}
    For every $f\in x^{1/2}L^2(\Dm)$, 
    \begin{align}
	f = \frac{1}{4\pi} x I_0^\sharp \circ \T_0^{1/2} \circ I_0 f.
	\label{eq:f0recons2}
    \end{align}
\end{proposition}

\begin{proof} We will use Proposition 2.14 and Theorem 2.15 in \cite{Eptaminitakis2024}, which is best done by setting $g = x^{-2}f\in x^{-3/2}L^2(\Dm) = L^2(\Dm, x^3\d V_H)$ so that $I_0 f = I_0 x^2 g$. Then we have 
    \begin{align*}
	\frac{1}{4\pi} x^{-1}  {I_0^\sharp}  \T_0^{1/2} I_0 x^2 g &= \frac{1}{4\pi} {(I_0 x^2)^{*}}  \T_0^{1/2} I_0 x^2 g \stackrel{(*)}{=} \frac{1}{4\pi} (\L_0^H)^{1/2} {(I_0 x^2)^{*}} I_0 x^2 g \stackrel{(\bullet)}{=} g,
    \end{align*}
    where step $(*)$ follows from \cite[Proposition 2.14]{Eptaminitakis2024}  (more specifically, that the intertwining of the operators $\T_0$ and $\L_0^H$ implies an intertwining of their functional calculi) and step $(\bullet)$ follows from \cite[Theorem 2.15]{Eptaminitakis2024}. Above, {$(I_0 x^2)^{*}$} denotes the  $ L^2(\Dm_H, x^{3}\, \d V_H) \to L_+^2(\Gh)$-adjoint of $I_0 x^{2}$ and the differential operator $\L_0^{H}$ is defined in \cite[Section 2.4]{Eptaminitakis2024}. %
    Multiplying both sides of this equality by $x^2$ completes the proof. 
\end{proof}

\begin{remark}\label{rem:f0recons}
    To reconstruct $\tilde f_0$, other approaches have been derived in the literature in the case where $\tilde f_0$ is either compactly supported or of Schwartz class: 
    \begin{itemize}
	\item Equation \eqref{eq:f0recons2} resembles \cite[Theorem 1]{Lissianoi1997} and \cite[Eq.\ (29)]{Bal2005}, formulated there for $C_c^\infty$ functions and derived by other means. A more thorough comparison will be the object of future work. 
	\item If $\tilde f_0$ has compact support, one may also use  the reconstruction formula of Kurusa (\cite[Theorem 4.1]{Kurusa1991}) or of  Berenstein-Casadio Tarabusi (\cite{Berenstein1991}), which is of backprojection-then-filter type, cf.  Theorem \ref{thm:inversionN0} below. One may also enclose $\text{supp } \tilde f_0$ inside a simple subset of the Poincar\'e disk, on which the Pestov-Uhlmann reconstruction formula \cite{Pestov2004} can be implemented with no error term (as those only appear in non-constant curvature). 
	\item Some formulas are written for $\tilde f_0$ of Schwartz class, e.g. Helgason \cite[Eq.\ (30) or (31)]{Helgason1999}. It is often stated in the literature that it is expected that formulas extend, although that the borderline cases are unknown to the authors.  
    \end{itemize}
\end{remark}

\subsubsection{Reconstruction from the Normal Operator - Proof of Theorems \ref{thm:inversionN0} and \ref{thm:recons_normal}}\label{ssub:reconstruction_from_the_normal_operator_proof_of_theorem_ref_thm_recons_normal}

We begin with some preliminaries to explain the notation in Theorem \ref{thm:recons_normal}. 
We write \begin{equation}
    I^\sharp \coloneqq\pih^*:C^\infty(\Gh)\to C^\infty(S\Dm^{\circ})
\end{equation} for the extension as a first integral. 
By Santal\'o's formula \eqref{eq:Santalo}, $I^\sharp $ is the formal adjoint of $I$ in the functional setting $ L^2(\OSD)\to L^2(\Gh)$. 
Thus by Lemma \ref{lem:surjective}, we have boundedness
\begin{equation}\label{eq:Istar_bd}
    I^\sharp :\muh^{-\delta}L^2(\Gh)\to {x}^{-\delta}L^2 (\OSD),\quad \delta>0.
\end{equation}
Recall that for $m\geq 0,$ we write $\ell^{*}_{m}$ for the $L^2(\Dm;S^{m}({}^0T^*\Dm))\to L^{2}(\OSD)$-adjoint of $\ell _{m}$.
By \eqref{eq:Hs} and Lemma \ref{lem:surjective}, $ I_m=I\ell_m:{x}^{\delta}L^2(\Dm;S^m({}^0T^*\Dm))\to \muh^{\delta}L^2(\Gh)$ is bounded if $\delta>0$, so 
\begin{equation}
    N_m=\ell^{*}_mI^\sharp I\ell_m:{x}^{\delta}L^2(\Dm;S^m({}^0T^*\Dm))\to{x}^{-\delta}L^2(\Dm;S^m({}^0T^*\Dm)),\quad  \delta>0,\label{eq:Nm_bd}
\end{equation}
is bounded. If $m=0$, $N_0=I_0^\sharp I_0=x(I_0x^2)^*(I_0x^2)x^{-2}$ in the notation of the proof of Proposition~\ref{prop:IperpSVD}.

Recalling Proposition \ref{prop:IperpSVD}, we write 
\begin{equation}
    N_\perp\coloneq x(iI_\perp)^*iI_\perp= -\star \d \ell_1^* I^\sharp I\ell_1 \star \d= -\star \d N_1 \star \d.
\end{equation}
Note that for $k=0,1$, $\star \d $ can be identified with an operator in  $\Diff_0^1(\Dm;\Lambda ^k({}^0T^*M),\Lambda ^{1-k}({}^0T^*M))$ by extension from the interior.
From Proposition~\ref{prop:IperpSVD} it follows that
$N_\perp:xH^{1,0}_w(\Dm)\to x^{1/2}L^2(\Dm)$ is a homeomorphism with  spectral decomposition
\begin{equation}\label{eq:Nperp}
    N_\perp \big(\frac{1}{n+1}x\widehat{\Phi^* Z_{n,k}^0}\big)={4\pi} x^2\widehat{\Phi^* Z_{n,k}^0}.
\end{equation}

\begin{proof}[Proof of Theorem \ref{thm:inversionN0}.]
    Comparing \eqref{eq:Nperp} with the spectral decomposition of $N_0:x^{1/2}L^2(\Dm)\to xH^{1,0}_w(\Dm)\subset x^{-1/2}L^2(\Dm)$ in \eqref{eq:N0_sp_dec} we obtain $ N_\perp N_0=16\pi^2 {Id}$ on $x^{1/2}L^2(\Dm)$.
\end{proof}
\noindent As in the proof of Theorem \ref{thm:inversionN0}, we also have
\begin{equation}\label{eq:Nperp_inv}
N_0N_\perp=16\pi^2{Id}\quad \text { on}\quad xH^{1,0}_w(\Dm).
\end{equation}

We will need an expression for $\ell_m^{*}$.
We first establish that for every $m\ge 0$ 
\begin{align}
A^{kt}\coloneqq	\langle c^{-m}\d z^{k} \d\bar z^{m-k}, c^{-m}\d z^{t} \d\bar z^{m-t}\rangle_g = 2^m \binom{m}{k}^{-1} \delta^{kt}, \qquad 0\le k,t\le m.
	\label{eq:inner}
\end{align}
To see this, recall that on decomposables the inner product defined in \eqref{eq:metric_on_fibers} gives
\begin{align*}
	\langle a_1\otimes \ldots \otimes a_m, b_1\otimes \ldots \otimes b_m\rangle_g  = \prod_{j=1}^m \langle a_j, b_j \rangle_g,
\end{align*}
and hence, using that $\sigma$ is symmetric and satisfies $\sigma \circ \sigma = \sigma$, 
\begin{align}
	\langle \sigma(a_1\otimes \ldots \otimes a_m), \sigma(b_1\otimes \ldots \otimes b_m)\rangle_g  = \frac{1}{m!}\sum_{\tau \in \mathrm{S}_m} \prod_{j=1}^m \langle a_j, b_{\tau (j)} \rangle_g.
	\label{eq:symdprod}
\end{align}
Now fix $m\ge 0$ and $0\le t,k\le m$, and consider $a_j = c^{-1}\d z$ for $1\le j\le k$, $a_j = c^{-1}\d\zbar$ for $k+1\le j\le m$, $b_j = c^{-1}\d z$ for $1\le j\le t$, and $b_j = c^{-1}\d\zbar$ for $t+1\le j\le m$. Since $\langle c^{-1} \d z, c^{-1} \d z \rangle_g = \langle c^{-1} \d\zbar, c^{-1} \d\zbar \rangle_g = 2$ and $\langle c^{-1} \d z, c^{-1} \d\zbar \rangle_g = 0$, we find that in the right hand side of \eqref{eq:symdprod}, the term $\prod_{j=1}^m \langle a_j, b_{\tau(j)} \rangle_g$ is nonzero and equal to $2^m$ if and only if $k=t$ and $\tau (\{1, 2, \ldots, k\}) = \{1, 2, \ldots, k\}$. The latter condition is satisfied for $k! (m-k)!$ elements of $\mathrm{S}_m$, hence \eqref{eq:inner} follows.

\begin{lemma}\label{lm:lm_adj}Let $m\geq 0$. In complex notation, upon writing $h=\sum_{t=-\infty}^{\infty} h_{t} \, e^{it \theta}\in x^{\delta}L^2(\OSD)$,
\begin{equation}
    \ell_m^{*}h=2^{1-m}\pi\sum _{k=0}^{m}\binom{m}{k} h_{-m+2k} \ \frac{\d z^k \d \bar{z}^{m-k}}{c^{m}}.
\end{equation}
\end{lemma}
\begin{proof}
Write $f=\sum_{k=0}^{m}f_{k}\frac{\d z^{k}\d \bar{z}^{m-k}}{c^{m}}\in x^{-\delta}L^2(\Dm;S^{m}({}^0T^*\Dm))$, so that $\ell_{m}f=\sum_{k=0}^{m} f_{k} e^{i(-m+2k)\theta}$.
Also write $\ell_m^* h=\sum_{t=0}^{m} H_{t}\frac{\d z^{t}\d \bar{z}^{m-t}}{c^{m}}$.
We compute 
\begin{align}
        2\pi \sum _{k=0}^{m}\sum_{t=-\infty}^{\infty}\int_\Dm  f_{k}\delta^{-m+2k,t}\overline{h_{t}}\d V_H=( \ell_m f ,h)_{L^{2}(\OSD)}=
        ( f, \ell_m^{*}h)_{L^2(\Dm)}\overset{\eqref{eq:inner}}=\!\sum _{k,t=0}^{m}\int _\Dm  f_{k} A^{kt}\overline{H_{t}}\d V_H.
\end{align}
From this, $\sum_{t=0}^{m}\overline{A^{kt}}{H_{t}}=2\pi {h_{-m+2k}}$ for each $k=0,\dots, m$,  thus \eqref{eq:inner} yields the claim.
    \end{proof}

By Lemma \ref{lm:lm_adj} and \eqref{eq:Akt2}, if $h=\sum_{t=-\infty}^{\infty} h_{t} \, e^{it \theta}$,
\begin{equation}\label{eq:lmstar}
    \ell_m\mathcal A_m\ell_m^{*}h=2\pi \sum _{k=0}^{m}h_{-m
    +2k }e^{i(-m
    +2k)\theta} =2\pi \sum _{k=0}^{m}\P_{-m
    +2k } h ,
\end{equation} recall \eqref{eq:projector}.
The main ingredient in the proof of Theorem~\ref{thm:recons_normal} is the following lemma:
\begin{lemma}\label{lm:diag_term}
Let $m\geq 1$. Suppose that $0<k\leq m$ (resp. $-m\leq k<0$) is an integer with the same parity as $m$ and let $\delta\in (0,|k|-1/2)$. Consider $\hat g_k\in \Omega_k ^{x^
\delta L^2}$ satisfying $ \eta_-  \hat g_k=0$ (resp. $\eta_+ \hat g_k=0$).
Then for any  $-m\leq k'\leq k$ (resp. $k\leq k'\leq m$) 
\begin{equation}\label{eq:triangularization}
  \frac{1}{2\pi}  \P_{k'}\ell _m\mathcal A_m \ell_m^* I^\sharp I \hat g_k= \begin{cases}
        2B\big(\frac{1}{2},|k|\big)\hat g_k, &k'=k, \\*
        0,& k'< k\ (\text{resp. }k'>k).
    \end{cases}
\end{equation}
\end{lemma}

\begin{proof}
Using that $ \overline{\mathcal{A}_m \big(c^{-m}{\d z^{k}\d \bar {z}^{m-k} }\big)}=
\mathcal{A}_m \big(c^{-m}{ \d \bar{z}^{k}\d z^{m-k} }\big)$ for $k=0,\dots,m$, we can check  that $\mathcal{A}_m$  commutes with conjugation in $S^{m}(^{0}T^*\Dm)$ (which as we recall is the complexification of a real vector bundle). 
Hence it suffices to prove \eqref{eq:triangularization} for $k>0$, as the case of $k<0$ will then follow by conjugation.
Let $k>0$ and write  $\hat g_k=c^k e^{i k\theta} \sum_{p=0}^{\infty}a_p z^{p}$  as in  \eqref{eq:a_p}.
As in the proof of Lemma \ref{lm:conv},  the infinite sum converges in $x^{\delta}L^{2}(\OSD)$. Hence by Lemma \ref{lem:surjective}, 
\begin{equation}\label{eq:interm}
    I\hat g_k=I\Big(e^{ik\phi}c^{k}\sum_{p=0}^{\infty} a_p z^p\Big)=\sum _{p=0}^{\infty }a_p I (z^p e^{ik\phi}c^{k})=\sum _{p=0}^{\infty }a_p I_{p,k} \in \muh^{\delta}L^2(\Gh) \subset L^2(\Gh).
\end{equation}
We now use \eqref{eq:lmstar}  to write, for $-m\leq k'\leq k$ with the same parity as $m$
\begin{align}
    \frac{1}{2\pi}\P_{k'}\ell _ m& \mathcal{A}_m\ell_m^{*} I^\sharp  I \hat g_k
        \overset{\eqref{eq:lmstar}}{=}\P_{k'}I^\sharp I\hat g_k
        \overset{\eqref{eq:interm}}{=} \P_{k'}I^\sharp \sum _{p=0}^{\infty }a_p I_{p,k}\\
        &\overset{(*)}{=} \sum _{p=0}^{\infty }a_p \P_{k'}I^\sharp  I_{p,k}
        {=}\sum _{p=0}^{\infty }a_p  \P_{k'}\pih^* I_{p,k}
         \overset{\eqref{eq:Ipq_leading}}{=}
        \sum _{p=0}^{\infty }a_p \P_{k'}\Big(2B\Big(\frac{1}{2},k\Big) z^p c^k e^{ik\theta} + O(e^{i(k+2)\theta})\Big)\\
        &\quad  =\begin{cases}
            \sum _{p=0}^{\infty }2B\left(\frac{1}{2},k\right)a_p e^{ik\theta}z^pc^k, &k'=k,\\*
            0, &k'<k
        \end{cases}    \quad =\begin{cases}
            2B\left(\frac{1}{2},k\right)\hat g_k, &k'=k,\\*
            0, &k'<k.
        \end{cases}
        \end{align} 
To justify $(*)$ we use the convergence of the infinite sum in $L^{2}(\Gh)$ by \eqref{eq:interm}, the  boundedness \eqref{eq:Istar_bd} as well as the fact that $\P_{k'}:x^{\delta'}L^2(\OSD)\to x^{\delta'}L^2(\OSD)$ is bounded for any $\delta'$.
In other words, the infinite sum following $(*)$  a priori makes sense in $x^{-\delta}L^2(\OSD)$. If $m$ and $k'$ have opposite parities, $\frac{1}{2\pi}\P_{k'}\ell _ m \mathcal{A}_m \ell_m^{*}I^\sharp I \hat g_k=0$ by {\eqref{eq:lmstar}}.
\end{proof}

\begin{remark}[A Euclidean-Hyperbolic comparison]\label{rmk:comparison}
    In \eqref{eq:triangularization},  we obtain in the case $k=k'$ a multiplication operator due to the fact that the coefficient in front of the bottom Fourier mode of $I_{p,q}$ does not depend on $p$, see \eqref{eq:Ipq_leading}. This is markedly different from the analogous situation in the Euclidean disk, which we now discuss for comparison. If we define $I_{p,q}^E$ similarly to \eqref{eq:Ipq} with $I_q$ replaced by the Euclidean X-ray transform (such a quantity is equally relevant to reconstruct higher-order components of tensor fields in the Euclidean disk), we now explain that the analogue of \eqref{eq:Ipq_leading} is 
\begin{align}
	\frac{I_{p,q}^E}{\cos\alpha} \circ \pi_F = \frac{2(-1)^p}{p+1} z^p e^{iq\theta} + O(e^{i(q+2)\theta}).
	\label{eq:Ipq_leadingE}
\end{align}
where $(\beta,\alpha)$ denote fan-beam coordinates for the Euclidean disk and $\pi_F$ is the analogue there of the map $\pih$. Division by $\cos\alpha$ in this case is a necessary step to ensure one-sided Fourier content after applying $\pi_F^*$. Thus in the Euclidean case, the coefficient of the bottom mode in \eqref{eq:Ipq_leadingE} does not depend on $q$ but depends on $p$. To justify \eqref{eq:Ipq_leadingE}, we first compute
\begin{align*}
	\frac{I_{p,q}^E}{\cos\alpha} \circ \pi_F = \frac{e^{iq(\beta+\alpha+\pi)}  I_0^E z^p}{\cos\alpha} \circ \pi_F = e^{iq\theta} \left(  \frac{ I_0^E z^p}{\cos\alpha} \circ \pi_F   \right).
\end{align*}
A direct calculation shows that 
\begin{align*}
	\frac{I_0^E z^p (\beta,\alpha)}{\cos\alpha} = \frac{2}{p+1} e^{ip (\beta+\alpha+\pi/2)} U_p(-\sin\alpha),
\end{align*}
where $U_p$ is the $p$-th Chebyshev polynomial of the second kind. Moreover, if $(z=x+iy,\theta)$ are coordinates on the Euclidean unit tangent bundle $S\Dm_E$ (describing a tangent vector $v = \cos\theta \partial_x + \sin\theta\partial_y$), we have that $(\beta+\alpha+\pi)\circ \pi_F = \theta$ and $-\sin\alpha \circ\pi_F = \frac{z e^{-i\theta}-\bar{z} e^{i\theta}}{2i}$, so that 
\begin{align*}
	\frac{I_0^E z^p}{\cos\alpha} \circ\pi_F = \frac{2 (-i)^{p}}{p+1} e^{ip \theta} U_p ( (z e^{-i\theta}-\bar{z} e^{i\theta})/(2i) ),
\end{align*}
The zeroth Fourier mode of the right side only comes from the top-degree term of $U_p$, whose coefficient is $2^p$, in particular $U_p ( (z e^{-i\theta}-\bar{z} e^{i\theta})/(2i) ) = \frac{1}{i^p} z^p e^{-ip\theta} + O(e^{-i(p-1)\theta})$, and hence \eqref{eq:Ipq_leadingE} follows. 
\end{remark}

We will now prove Theorem \ref{thm:recons_normal}. 
The operators $Q_{m,k}$  in its statement are given by
\begin{equation}
    Q_{m,k}\coloneqq \ell_m^{-1}\P_k\ell _m,\qquad -m\leq k\leq m \quad  \qquad k \equiv m \mod 2,
\end{equation}
 using the isomorphism property \eqref{eq:isomorphism_lm}.
 On $\mathcal{F}(\Dm;S^m({}^0T^*\Dm))$ one has $\sum_{j=0}^{m}Q_{m,-m+2j}={Id}$ and $Q_{m,k}^{2}=Q_{m,k}$.

\begin{proof}[Proof of Theorem \ref{thm:recons_normal}]
Let $f$ be as in the statement, and decompose it as in  \eqref{eq:gauge_decomp}.
By Lemma \ref{lm:kernel}, $N_m \d^{s} q=0$. 
By \eqref{eq:ttdecom}, for $1\leq k\leq m$ with $k\equiv m \mod 2$,
\begin{equation}\label{eq:tensor-to-pol}
    \ell_m L^{(m-k)/2}\tilde{f}_{k}=\ell_{k}\tilde{f}_{k}=(\P_{k} +\P_{-k})\ell_{k}\tilde{f}_{k}\eqqcolon\hat g_{k}+\hat g_{-k}, \quad\text{where}\quad \hat g_{\pm k}\in \Omega_{\pm k}^{x^{\delta}L^2}\cap \ker \eta_{\mp}.
\end{equation}
From this, it is immediate upon applying $\ell_m^{-1}\P_{\pm k}$ that
\begin{equation}\label{eq:k_kf_k}
    \ell_{m}^{-1}\hat g_{\pm k}=\ell_{m}^{-1}\P_{\pm k}\ell_{m}L^{(m-k)/2}\tilde{f}_{ k}=Q_{m,\pm k}L^{(m-k)/2}\tilde{f}_{k}.
\end{equation}
We now separate cases.

\noindent $\bullet$ \textbf{Case $m$ even.}
To prove \eqref{eq:normal_inversion}-\eqref{eq:normal_inversion2}, we compute $Q_{m,k'}\mathcal A_m \mathsf{D}$ for even $k'\in [-m,m]$.
If  $0 < k'\leq m$ even, and writing $\hat{g}_0\coloneqq \tilde{f}_0/2$, we have
\begin{align}
        Q_{m,k'}\mathcal A_m \mathsf{D}&=Q_{m,k'}\mathcal A_mN_m f =\sum _{\substack{0\leq k\leq m\\ k \text{ even}}}Q_{m,k'}\mathcal A_m\ell_m^*I^\sharp I\ell_{m} L^{(m-k)/2}\tilde f_{k}\\
        &\overset{\eqref{eq:tensor-to-pol}}{=}\sum _{\substack{0\leq  k\leq m\\ k \text{ even}}}(\ell_m^{-1}\P_{k'}\ell_m)\mathcal A_m\ell_m^*I^\sharp I( \hat{g}_{k}+\hat{g}_{-k}) \\ 
        &\!\!\!\!\overset{\text{Lem. \ref{lm:diag_term}}}{=}\sum _{\substack{0\leq k<  k'\\ k \text{ even}}}(\ell_m^{-1}\P_{k'}\ell_m)\mathcal A_m\ell_m^*I^\sharp I( \hat{g}_{k}) +4\pi B\big(\frac{1}2,k'\big)\ell_m^{-1} \hat{g}_{k'} \\ 
        &\overset{\eqref{eq:k_kf_k}}{ =}\sum _{\substack{0\leq  k< k'\\ k \text{ even}}}Q_{m,k'}\mathcal A_mN_m Q_{m,k}L^{\frac{m-k}2}\tilde{f}_k+ 4\pi B\big(\frac{1}2,k'\big)Q_{m,k'}L^{\frac{m-k'}2}\tilde{f}_{k'}. \label{eq:matrix_proof}
\end{align}

If $k'=0$, the first three equalities above are exactly the same, and then Lemma \ref{lm:diag_term} yields
\begin{equation}\label{eq:case_k_0}
    \begin{aligned}
    Q_{m,0}\mathcal{A}_m\mathsf{D}=Q_{m,0}\mathcal{A}_mN_m& L^{m/2} \tilde{f}_0
    =Q_{m,0}\mathcal{A}_mN_m \ell_m^{-1}\tilde{f}_0
    =\ell_m^{-1}\P_0(\ell_m\mathcal{A}_m\ell_m^{*})I^\sharp I\ell_0\tilde{f}_0\\ 
     &\quad \overset{\eqref{eq:lmstar}}{=}
     2\pi \ell_m^{-1}\P_0I^\sharp I\ell_0\tilde{f}_0=\ell_m^{-1}\ell_0^{*}I^\sharp I\ell_0\tilde{f}_0=\ell_m^{-1}N_0\tilde{f}_0=L^{m/2}N_0\tilde{f}_0.
\end{aligned}    
\end{equation}
From this we obtain \eqref{eq:normal_inversion}: 
\begin{equation}
    Q_{m,0}\mathcal{A}_m \mathsf{D}=L^{m/2}N_0\tilde{f}_0\implies N_0\tilde{f}_0=\P_0\ell_m \mathcal{A}_m\mathsf{D}\implies \tilde{f}_0=R_0 \P_0\ell_m \mathcal{A}_m\mathsf{D}.
\end{equation}
The inversion as stated in \cite{Berenstein1991} is for $\tilde{f}_0\in C_c^{\infty}(\Dm^\circ)$, but it also holds for $\tilde{f}_0\in x^{\delta}L^{2}(\Dm)$, by \cite[Proposition 5.2]{Eptaminitakis2022}.

Finally, note for even  $k'>0$  we have $Q_{m,-k'}=\overline {Q_{m,k'}}$.
Using that $\mathcal{A}_m$ commutes with conjugation as explained before and that $\overline{\mathsf{D}}=N_m \overline{f}$, for such $k'$ we have 
\begin{equation}\begin{aligned}
    Q_{m,-k'}{\mathcal{A}_m}{\mathsf{D}}\overset{ \phantom{\eqref{eq:matrix_proof}}}{=}&\overline {Q_{m,k'}{\mathcal{A}_m}\overline{\mathsf{D}}}\\ \overset{ \eqref{eq:matrix_proof}}{=}  &\sum _{\substack{0\leq k< k'\\ k \text{ even}}}Q_{m,-k'}{\mathcal{A}_m}N_m Q_{m,-k}L^{\frac{m-k}2}{\tilde{f}}_k+ 4\pi B\big(\frac{1}2,k'\big)Q_{m,-k'}L^{\frac{m-k'}2}{\tilde{f}}_{k'}.
    \end{aligned}\label{eq:neg_k_pr}
\end{equation}

Now to obtain \eqref{eq:normal_inversion2}, 
combine \eqref{eq:matrix_proof} with \eqref{eq:neg_k_pr} to see that for even ${k'}>0$
\begin{align}
        L^{\frac{m-k'}{2}}\tilde f_{k'}&
        =Q_{m,{k'}}L^{\frac{m-k'}{2}}\tilde f_{k'}+ Q_{m,-{k'}}L^{\frac{m-k'}{2}}\tilde f_{k'}\\
        &=\frac{Q_{m,k'}\mathcal{A}_m\mathsf{D}+Q_{m,-k'}{\mathcal{A}_m}{\mathsf{D}}}{4\pi  B\big(\frac{1}{2},k'\big)}-\!\!\!\! \sum _{\substack{0\leq k< k'\\ k \text{ even}}}\!\!\!\!\frac{Q_{m,k'}\mathcal{A}_mN_m Q_{m,k}+Q_{m,-k'}{\mathcal{A}_m}N_m Q_{m,-k}}{4\pi  B\big(\frac{1}{2},k'\big)}L^{\frac{m-k}2}\tilde{f}_k.
\end{align}
\noindent $\bullet$ \textbf{Case $m$ odd.} Following the same steps as in \eqref{eq:matrix_proof} and \eqref{eq:neg_k_pr}, we see that for $1\leq k'\leq m$ odd we have 
\begin{align}
  Q_{m,\pm k'}\mathcal A_m \mathsf{D} &=Q_{m,\pm k'}\mathcal A_mN_m L^{\frac{m-1}{2}} \star \d f_s\\*
  &+\sum _{\substack{0\leq  k< k'\\ k \text{ odd}}}Q_{m,\pm k'}\mathcal A_mN_m Q_{m,\pm k}L^{\frac{m-k}2}\tilde{f}_k+ 4\pi B\big(\frac{1}2,k'\big)Q_{m,\pm k'}L^{\frac{m-k'}2}\tilde{f}_{k'},
\end{align}
which yields \eqref{eq:normal_inversion2odd} upon adding the two equalities corresponding to each odd $1\leq k'\leq m$ and solving for $L^{\frac{m-k'}{2}}\tilde f_{k'}=(Q_{m, k'}+Q_{m,- k'})L^{\frac{m-k'}2}\tilde{f}_{k'}$.
For \eqref{eq:normal_inversion_odd}, note that in the special case $k'=1$,
\begin{align}
    (Q_{m,1}+Q_{m,-1})\mathcal A_m \mathsf{D} &=(Q_{m,1}+Q_{m,-1})\mathcal A_m N_mL^{\frac{m-1}{2}} \star \d f_s+4\pi B\big(\frac{1}2,1\big)L^{\frac{m-1}{2}}\tilde f_1\\ 
    &=\ell_{m}^{-1}(\mathcal{P}_{1}+\mathcal{P}_{-1})(\ell_m\mathcal A_m\ell _m^{*})I^\sharp I_1 \star \d f_s+4\pi B\big(\frac{1}2,1\big)L^{\frac{m-1}{2}}\tilde f_1\\
    &\overset{\eqref{eq:lmstar} }=2\pi\ell_{m}^{-1}(\mathcal{P}_{1}+\mathcal{P}_{-1})I^\sharp I_1 \star \d f_s+4\pi B\big(\frac{1}2,1\big)L^{\frac{m-1}{2}}\tilde f_1\\
    &\overunderset{\eqref{eq:lmstar}, \text{ case  }m=1}{\eqref{eq:Akt2}}=2\ell_{m}^{-1}\ell_1 N_1 \star \d f_s+4\pi B\big(\frac{1}2,1\big)L^{\frac{m-1}{2}}\tilde f_1\\
    \implies&  N_1 \star \d f_s=\frac12 \ell_1^{-1}(\P_1+\P_{-1})\ell_m \A_m \mathsf{D}-2\pi B(\frac12,1)\tilde f_1.
\end{align}
Applying $-\star \d$ to both sides, the last term vanishes and we find 
\begin{equation}\label{eq:f_s_inversion}
    N_\perp f_s=-\frac12 \star \d \ell_1^{-1} (\P_1+\P_{-1})\ell_m \A_m \mathsf{D}.
\end{equation}
Note that \eqref{eq:Nperp_inv} in particular holds on $C_c^\infty(\Dm^\circ)$.
Now \eqref{eq:Nperp_inv} can be decomposed as the following chain of  maps which are bounded if $\delta\in (0,1/2)$:
\begin{equation}
    x^\delta H_0^1(\Dm)\overset {\star \d }\longrightarrow x^\delta L^2(\Dm;{}^0T^*\Dm)\overset {N_1}{\underset {\eqref{eq:Nm_bd}}\longrightarrow}
    x^{-\delta} L^2(\Dm;{}^0T^*\Dm)
    \overset {-\star \d }{\longrightarrow}x^{-\delta} H_0^{-1}(\Dm)\overset{N_0}\longrightarrow x^{-\delta}L^2(\Dm);
\end{equation}
 the last mapping property follows by \cite[Proposition 4.4]{Eptaminitakis2022}
 and Proposition \ref{prop:Sobolev_mapping}.
Hence \eqref{eq:Nperp_inv} holds on $x^\delta H_0^1(\Dm) $ and upon applying $N_0$ to \eqref{eq:f_s_inversion} we obtain \eqref{eq:normal_inversion_odd}.
\end{proof}

\appendix

\section{Background Material}\label{sub:the_0_calculus}
\subsection{Polyhomogeneous Conormal Distributions}\label{sec:phg}

In what follows,  a \textit{smooth index set} is a discrete set $E\subset \Cm\times \mathbb{N}_0 $ with the properties $|(s_j,p_j)|\to\infty\implies\Re(s_j)\to\infty$ and 
$(s_j,p_j)\in E\implies (s_j+m,p_j-\ell)\in E$ for any $  m\in \mathbb{N}_0=\{0,1,\dots\} $ and $\ell=0,1,\dots,p_j.$

Let $M^n$ be a compact manifold with boundary with $\rho $ a bdf and let ${E}\subset \Cm\times \Nm_0$ be a smooth index set as above.
A function $f\in \rho^{s} L^\infty(M)$ is in the space $\mathcal{A}^s(M) $ of \emph{conormal functions of order $s$} if for any number of vector fields $V_1,\dots V_k$ that are tangent to $\partial M$ we have
\begin{equation}
    V_1\dots V_kf\in \rho^{s}L^\infty(M).
\end{equation}
The space $\Aphg^{E}(M)$ of  \emph{polyhomogeneous functions with index set $E$} is defined by
\begin{equation}
    f\in \Aphg^{E}(M) \quad \text{if}\quad f\sim \sum _{\Re s_j\to \infty}\sum _{\ell=0}^{p_j} a_{j\ell} \rho^{s_j}\log^{\ell}(\rho),\text{ where }a_{j\ell}\in C^\infty(\partial M), \quad (s_j,p_j)\in {E}\label{eq:phg}
\end{equation}
with respect to a product decomposition $[0,\epsilon)_\rho\times \partial M$ of a  collar neighborhood of $\partial M$.
The meaning of the asymptotic expansion is that  for every $N\in  \Rm$, if one calls $f_N$ the sum on the right hand side corresponding to $\Re s_j\leq N$, then $|f-f_N|\in \mathcal{A}^{N}(M)$ (see \cite{Melrose1992} and \cite{Grieser2001}).

If $a\in \Rm$, we write $\Re E>a$ (or ${E}> a$ if $E\subset \Rm\times \Nm_0)$ if $\Re(s)>a$ for all $(s,p)\in E$.
We also write  $\Re {E}\geq a^-$  to mean that $\Re s\geq a$ for all $(s,p)\in {E}$.
So, if  $f\in \Aphg^{E}(M)$ and $\Re E\geq a^-$, $\rho^{-a}f$ is not necessarily bounded but it grows at most logarithmically fast as $\rho\to 0$.
We write $\Re {E}\geq a$  if $\Re {E}\geq a^-$ and there is no element $(s,p)$ in ${E}$ with $\Re(s)=a$ and $p\neq 0$.
Finally, we let $\inf(E)\coloneqq \inf\{\Re(s):(s,p)\in E\}$. If $E\neq \emptyset$, the infimum is attained.
Given index sets $E$, $E'$, their sum and extended union respectively are defined as
    \begin{align}
        E+E'\coloneqq&\{(s,p)+(s',p'):(s,p)\in E, \; (s',p')\in E'\},\\
        E\overline{\cup}E'\coloneqq&E\cup E'\cup\{(s,p+p'+1):\text{there exist }(s,p)\in E, (s,p')\in E'\}.
     \end{align}

\subsection{The 0-calculus}\label{sec:the_0_calculus}

In this section we give a quick overview of the 0-calculus, in particular the results we need to prove Proposition~\ref{prop:Ppm}.
For more details see \cite{Mazzeo1987,Mazzeo1991,Hintz2021}.
Let $M^{n}$ be a compact manifold with boundary, and let $\rho$ be a bdf. The Lie algebra of 0-vector fields consists of those vector fields that vanish at $\partial M$, or, equivalently, of the smooth sections of ${}^0TM$.
We denote by $\mathrm{Diff}_0^k(M)$  the space of those differential operators of order $k$ that consist of locally finite sums of at most $k$-fold compositions of 0-vector fields.
If $P\in \mathrm{Diff}_0^k(M)$, then we have in coordinates $u=(u^1,\dots ,u^n)$  (and using multi-index notation)
\begin{equation}
    P=\sum _{|\alpha|\leq k} a_{\alpha}(u)(\rho\partial_{u})^\alpha, \quad a_\alpha\in C^\infty.
\end{equation}
Its 0-principal symbol is given by 
\begin{equation}
    \sigma^k_0(\zeta)=\sum _{|\alpha|=k} a_{\alpha}(u)(i\xi)^\alpha, \qquad  \zeta =\Big(u,\sum _{j=1}^{n}\xi_j \frac{\d u^j}{\rho}\Big)\in {}^0T^*M\setminus 0.
\end{equation}
Then $P\in \mathrm{Diff}_0^k(M)$ is elliptic if $\sigma^k_0(\zeta)\neq 0$ for all $\zeta \in {}^0T^*M \setminus 0$.
To obtain Fredholm properties, one needs the notions of the reduced normal operator and the indicial operator.

The normal operator is a differential operator in the inward tangent space at a point $p\in \partial M$, given by freezing coefficients at $p$ in terms of special coordinates. 
Fix $p\in \partial M$ and let the coordinates $(u)=(u^1,u')$ be chosen such that $u^1\eqqcolon x$ is a bdf and $u'=(u^2,\dots, u^n)$ restrict to coordinates on $\partial M$ with $u'(p)=0$.
Then if in these coordinates 
\begin{equation}
    P=\sum _{|\alpha|+j\leq k} a_{\alpha,j}(u)(x\partial_x)^j(x\partial_{u'})^\alpha,
\end{equation}the normal operator is given by 
\begin{equation}
    N(P,p)=\sum _{|\alpha|+j\leq k} a_{\alpha,j}(0)(x\partial_x)^j(x\partial_{u'})^\alpha.
\end{equation}
Taking Fourier transform in the translation invariant directions $(u^2,\dots,u^n)$ with dual variable $\eta$ and then setting $t=x|\eta|$, $\hat{\eta}=\eta/|\eta|$, one obtains the reduced normal operator, given by 
\begin{equation}
    \hat{N}(P,p,\hat{\eta})=\sum _{|\alpha|+j\leq k} a_{\alpha,j}(0)(t\partial_t)^j(ti\hat{\eta})^\alpha,
\end{equation}
a parameter-dependent ordinary differential operator on $[0,\infty)$.
Elements in its nullspace are smooth on $(0,\infty)$ and if they are temperate, they are Schwartz as $t\to \infty$.

The indicial operator is given by
\begin{equation}
    I(P,p)=\sum _{j\leq k} a_{0,j}(0)(t\partial_t)^j,
\end{equation}
and the indicial family by the Mellin transform of the latter, namely
\begin{equation}
    I(P,p,\sigma )=\sum _{j\leq k} a_{0,j}(0)\sigma^j.
\end{equation}
 The  \emph{boundary spectrum} of $P$ is given by
\begin{align}
    \mathrm{Spec}_b(P,p)=&\{(\zeta,m)\in \Cm\times\Nm_0:  I(P,p,\sigma)^{-1}\text{ has a pole of order at least }m+1\text{ at }\sigma=\zeta\}.
\end{align}

\begin{definition}[{\cite[Def. 1.2]{Hintz2021}}]\label{elliptic at weight}
    Let $P\in \mathrm{Diff}_0^k(M) $ with elliptic principal symbol and fix $p\in \partial M$.
    We say that $\hat{N}(P,p,\hat{\eta})$ is \emph{invertible at the weight $\alpha$} if all of the following are satisfied:
    \begin{enumerate}[(i)]
        \item\label{item1} $\alpha\neq \Re(\zeta) $ for all $\zeta\in \mathrm{Spec}_b(P,p)$;
        \item \label{item2}If $f\in t^{\alpha-\frac{n-1}{2}} L^2([0,\infty],\frac{\d t}{t^{n}})$ satisfies $\hat{ N}(P,p,\hat{\eta})f=0$, then $f=0$.
        \item\label{item3} If $f\in t^{-\alpha +\frac{n-1}{2}}L^2([0,\infty],\frac{\d t}{t^{n}})$ satisfies $\hat{ N}(P,p,\hat{\eta})^*f=0$, then $f=0$, where $\hat{ N}(P,p,\hat{\eta})^*f=0$ is the $\d t/t^{n}$-formal adjoint of $\hat{ N}(P,p,\hat{\eta})$. 
    \end{enumerate}
    The operator $P\in \mathrm{Diff}_0^k(M)$ is \emph{fully elliptic at the weight }$\alpha$ if it is elliptic and in addition $\hat{N}(P,p,\hat{\eta})$ is invertible at the weight $\alpha$ for every $p\in \partial M$ and $\hat{\eta}\in \mathbb{S}^{n-2}$.
\end{definition}

In what follows we let $g$ be an AH metric on $M$.
If $m\geq 0$ is an integer and $\alpha \in \Rm$, we define the 0-Sobolev space as 
\begin{equation}\label{eq:0_sob}
    \rho^\alpha H_0^m(M)=\{u\in \rho^\alpha L^2(M): Pu\in \rho^\alpha  L^2(M)\quad \forall P\in \mathrm{Diff}_0^m(M)\}.
\end{equation}
Recall that, according to the convention \eqref{eq:notation}, $L^2(M)\coloneqq L^2(M,\d V_g)$.
For $s\geq 0$, $\rho^\alpha H_0^s(M)$ is defined by interpolation and for $s<0$ by duality.
By \cite[Lemma 3.4]{Lee2006}, $C_c^{\infty}(M^\circ)$ is dense in $\rho^{\alpha}H_0^{s}(M)$ for all $s\geq 0$.
The space $\rho^\alpha H_0^s(M)$  is a module over $C^\infty(M)$, so one can define invariantly $\rho^\alpha H_0^s$-regularity sections of vector bundles.
We denote $ \rho^\alpha H_0^\infty(M,\d V_g)\coloneqq \bigcap_{s\in \Rm}\rho^\alpha H_0^s(M)$.
If $P\in \mathrm{Diff}_0^k(M)$, we have the mapping property
\begin{equation}
	P:\rho^\alpha H_0^s(M)\to \rho^\alpha H_0^{s-k}(M),
\end{equation}
for any $s$ and $\delta$.
Moreover, if $P$ is elliptic (in the sense of having non-vanishing 0-principal symbol), then
\begin{equation}\label{eq:0-ellipticity}
	Pu\in \rho^\delta H_0^s(M)\implies u\in \rho^\delta H_0^{s+k}(M).
\end{equation}

The pseudodifferential operators in the 0-calculus are defined by means of the behavior of their Schwartz kernel in the 0-stretched space $M^2_0\coloneqq [M^2;\partial \mathrm{diag}]$ (see \cite{Mazzeo1987}), that is, the space obtained by blowing up the boundary of the diagonal in the product $M^2$.
For us, only their mapping properties will matter, so we will not define them in detail; the conventions we follow are the same as in \cite{Hintz2021}, and we refer the reader there for definitions of the terms below.
The operators of order $k$ in the small calculus, acting on 0-half densities,  are denoted by $\Psi_0^k(M)$: their Schwartz kernels are sections of the kernel half density bundle $\mathrm{KD}_0$ over $M_0^{2}$ which are conormal to the lift of the diagonal and vanish to infinite order at the side faces.
The operators in the large calculus are denoted by $\Psi_0^{k,(E_\ell,E_f,E_r)}(M)$; their Schwartz kernels are  sections of $\mathrm{KD}_0$ which are conormal of order $m $ to the lifted diagonal in $M_0^2$ and have polyhomogeneous expansions with index sets $E_\ell$, $E_r$, $E_f$ in the left, right, and front face of $M_0^2$ space respectively. The small calculus is a special case, with $E_\ell=E_r=\emptyset$ and $E_f=\mathbb{N}_0\times\{0\}$. On Sobolev spaces, we have (see \cite[Cor. 3.23, Theorem 3.25]{Mazzeo1991} and \cite[Prop. 3.3]{Eptaminitakis2022})
\begin{proposition}\label{prop:Sobolev_mapping}
	Let $Q\in \Psi_0^{k,({E}_\ell,{E}_{f},{E}_r)}(M)$, $k\in \Rm$. Provided $s\in \Rm$, $\Re({E}_r)>(n-1)/2-\delta$, $\Re({E}_f)\geq \delta'-\delta$  and $\Re({E}_\ell)>\delta'+(n-1)/2$, one has that 
	\begin{equation}
		Q:\rho^\delta H_0^s(M)\to \rho^{\delta'} H_0^{s-k}(M)
	\end{equation}
	is bounded\footnote{With the help of the AH metric $g$ we can trivialize the various half density bundles and view 0-pseudodifferential operators as acting on scalar functions.}.
\end{proposition}
Regarding mapping properties on polyhomogeneous functions, we have (\cite[Lemma 2.4]{Hintz2021})
\begin{proposition}\label{prop:phg_mapping}
	Let $u\in \Aphg^F(M)$ and $P\in \Psi_0^{k, (E_\ell,E_f,E_r)}(M)$, $k\in \Rm$. If $\Re(E_r+F)>n-1$
	then $Pu\in \Aphg^{F'}(M)$, where $F'=E_\ell\overline{\cup}(E_f+F)$.

\end{proposition}

We will now state the main Fredholm theorem for differential operators in the 0-calculus, only in as much generality as we will need later.
First we define certain index sets. 
\begin{definition}\label{def:index_sets}
	Suppose $P\in \mathrm{Diff}_0^k(M)$ is fully elliptic at the weight $\alpha$ with $\mathrm{Spec}_b(P)=\mathrm{Spec}_b(P,p)$ independent of $p$. Let $\mathcal{E}_\pm$ be  the smallest smooth index sets such that 
\[       \mathcal{E}_+\supset{\{(z,m):(z,m)\in \mathrm{Spec}_b(P):\Re z>\alpha \}},\quad
            \mathcal{E}_-\supset{\{(-z,m):(z,m)\in \mathrm{Spec}_b(P):\Re z<\alpha \}}
  \]
    and define
\[	
			\widehat{\mathcal{E}}_\pm(0)\coloneqq \mathcal{E}_\pm,\qquad \widehat{\mathcal{E}}_\pm(j) \coloneqq {\mathcal{E}}_\pm\overline{\cup}\big(\widehat{\mathcal{E}}_\pm (j-1)+1\big), \quad j=1,2,\dots,\qquad \widehat{\mathcal{E}}_\pm\coloneqq \bigcup_{j=0}^\infty \widehat{\mathcal{E}}_\pm(j),\]
			\[\widehat{\mathcal{E}}_\pm^\sharp\coloneqq \left(\widehat{\mathcal{E}}_\pm+1\right)\overline{\cup}\left( \widehat{\mathcal{E}}_\pm\overline{\cup}\left(\widehat{\mathcal{E}}_\pm+1\right)\overline{\cup} \left(\widehat{\mathcal{E}}_\pm+2\right)\overline{\cup}\dots\right),\quad 
			\widehat{\mathcal{E}}_{\mathrm{ff}}=(\mathbb{N}_0\times \{0\})\overline{\cup}\left( \widehat{\mathcal{E}}_+^\sharp+\widehat{\mathcal{E}}_-^\sharp+n-1\right).
    \]
\end{definition}

The following theorem is due to Mazzeo (\cite{Mazzeo1991}), and we state it here as it appears in \cite[Corollary 3.3]{Hintz2021}.
\begin{theoremofothers}\label{thm:fredholm}%
    Suppose that $P\in \mathrm{Diff}_0^k(M)$ is fully elliptic at the weight $\alpha$ with $\mathrm{Spec}_b(P,p)$ independent of $p$. Then for any $s\in \Rm$
    \begin{equation}
        P:\rho^{\alpha-(n-1)/2}H^s_0(M)\to \rho^{\alpha-(n-1)/2}H^{s-m}_0(M)
    \end{equation}
    is Fredholm.
    If it is also invertible, then its inverse 
    \begin{equation}
        G:\rho^{\alpha-(n-1)/2}H^{s-m}_0(M)\to \rho^{\alpha-(n-1)/2}H^s_0(M)
    \end{equation}
    satisfies 
    \begin{equation}
        G\in \Psi_0^{-k}(M)+\Psi_0^{-\infty,(\widehat{\mathcal{E}}_+,\widehat{\mathcal{E}}_{\mathrm{ff}},\widehat{\mathcal{E}}_-+(n-1)) }(M),
    \end{equation}
    where the index sets in Definition \ref{def:index_sets} are used. 
\end{theoremofothers}

\section{The space \texorpdfstring{$H^{1,0}_w(\Dm)$}{H (1,0) - w} in the \texorpdfstring{$0$}{0}-Sobolev scale \texorpdfstring{$x^\delta H^1_0(\Dm)$}{}} \label{app:0Sob}

In this section, we provide a proof of Lemma \ref{lem:H10w}.

\begin{proof}
We first remark that in studying the membership of functions in any of these spaces, it suffices to focus attention near the boundary $x=0$. This is because all of these spaces are equivalent in the interior: that is, for any $\delta,\delta'\in\Rm$ we have that $x^\delta H_0^1(\Dm), x^{\delta'} H^{1,0}_w(\Dm)\subset H^1_{loc}(\Dm^\circ)$ (where $H^1_{loc}(\Dm^\circ)$ denotes the space of functions on $\Dm^\circ$ that are locally $H^1$), and moreover $x^\delta H_0^1(\Dm)\cap L^2_c(\Dm^\circ) = x^{\delta'}H^{1,0}_w(\Dm)\cap L^2_c(\Dm^\circ) = H^1_c(\Dm^\circ)$, where $L^2_c(\Dm^\circ)$ and $H^1_c(\Dm^\circ)$ denote functions in $L^2(\Dm^\circ)$ and $H^1(\Dm^\circ)$ that are compactly supported. (Indeed, the integrand in \eqref{eq:H10w} appears singular near $x=1$, i.e.\ near the origin, but one can show that the integrand is comparable to $|\partial_zu|^2+|\partial_{\zbar}u|^2+|u|^2$ near $z=0$.) Thus, to prove the desired inclusions, it suffices to assume that the functions we work with are supported away from the origin.

We recall, for $u$ supported near the boundary, that
\[
u\in H_0^1(\Dm)\iff x\partial_xu, x\partial_\omega u, u\in L^2(\Dm).
\]
Moreover, $C_c^\infty(\Dm)$ is dense in $H_0^1(\Dm)$, and if $u$ is supported away from the origin, then we can approximate it in the $H_0^1(\Dm)$ norm by a sequence of $C_c^\infty(\Dm)$ also supported away from the origin. Finally, we note that
\[x\partial_x(x^{-\delta}u) = x^{-\delta}(x\partial_xu) - \delta x^{-\delta}u,\]
which implies, for $u$ supported away from the origin, that the norms $\|u\|_{x^{\delta}H_0^1(\Dm)} = \|x^{-\delta}u\|_{H_0^1(\Dm)} = \|(x\partial_x)(x^{-\delta}u)\|_{L^2(\Dm)}+\|(x\partial_\omega)(x^{-\delta}u)\|_{L^2(\Dm)} + \|x^{-\delta}u\|_{L^2(\Dm)}$ and $u\mapsto \|x\partial_xu\|_{x^\delta L^2(\Dm)} + \|x\partial_\omega u\|_{x^\delta L^2(\Dm)} + \| u\|_{x^\delta L^2(\Dm)}$ are equivalent.

We first show the inclusion $x^{-1/2}H_0^1(\Dm)\subset H^{1,0}_w(\Dm)$, as this would then imply the inclusion $x^{1/2}H_0^1(\Dm)\subset xH^{1,0}_w(\Dm)$. To do so, it suffices to show, for $\phi\in C_c^\infty(\Dm^\circ)$ supported away from a fixed neighborhood of the origin, an estimate of the form $\|\phi\|_{H^{1,0}_w(\Dm)}\le C\|\phi\|_{x^{-1/2}H_0^1(\Dm)}$, or equivalently
\begin{equation}
\label{eq:H01-to-H01w}
\|\phi\|_{H^{1,0}_w(\Dm)}\le C\left(\|x\partial_x\phi\|_{x^{-1/2}L^2(\Dm)}+\|x\partial_\omega\phi\|_{x^{-1/2}L^2(\Dm)} + \|\phi\|_{x^{-1/2}L^2(\Dm)}\right),
\end{equation}
since then any approximating sequence for $u\in H_0^1(\Dm)$ supported away from the origin would be Cauchy in $H^{1,0}_w(\Dm)$. Recalling that $\|u\|_{x^{-1/2}L^2(\Dm)} = \|u\|_{L^2(\Dm, x\d V_H)} = \|u\|_{L^2(\Dm, x^{-1}\d x\d\omega)}$ as $\d V_H = x^{-2}\d x\d\omega$, we see that
\begin{equation}
\label{eq:H01w-upper}
\begin{aligned}
\|\phi\|_{H^{1,0}_w(\Dm)}^2 &=\int_{\Sm^1}\int_0^1 x(1-x^2)|\partial_x\phi|^2+ \frac{x}{1-x^2}|\partial_\omega \phi|^2+x|\phi|^2 \d x\d\omega \\
&=\int_{\Dm} \left((1-x^2)|x\partial_x\phi|^2 + \frac{1}{1-x^2}|x\partial_\omega\phi|^2+x^2|\phi|^2\right)x^{-1}\d x\d\omega \\
&\le \int_{\Dm} C(|x\partial_x\phi|^2+|x\partial_\omega\phi|^2+|\phi|^2) x\d V_H \\
&= C\left(\|x\partial_x\phi\|_{x^{-1/2}L^2(\Dm)}^2+\|x\partial_\omega\phi\|_{x^{-1/2}L^2(\Dm)}^2 + \|\phi\|_{x^{-1/2}L^2(\Dm)}^2\right),
\end{aligned}
\end{equation}
thus yielding an estimate equivalent to that desired in \eqref{eq:H01-to-H01w}. (Note that we may bound $\frac{1}{1-x^2}$ due to the support of $\phi$ away from $x=1$). This shows $x^{-1/2}H_0^1(\Dm)\subset H^{1,0}_w(\Dm)\implies x^{1/2}H_0^1(\Dm)\subset xH^{1,0}_w(\Dm)$. Moreover, by definition $H^{1,0}_w(\Dm)$ contains $\Cev(\Dm)$, so it contains the function $1$; however, $1\not\in x^{-1/2}H_0^1(\Dm)$, since $x^{1/2}\not\in L^2(\Dm)$ as $\d V_H = x^{-2}\d x\d\omega$. Thus, $xH^{1,0}_w(\Dm)\backslash x^{1/2}H_0^1(\Dm)$ is nonempty, as it contains the function $x$.

To show $xH^{1,0}_w(\Dm)\subset x^{1/2-\epsilon}H_0^1(\Dm)$, it suffices to show $H^{1,0}_w(\Dm)\subset x^{-1/2-\epsilon}H_0^1(\Dm)$. 
Since $\Cev(\Dm)$ is dense in $H^{1,0}_w(\Dm)$, it suffices to show an estimate of the form $\|\phi\|_{x^{-1/2-\epsilon}H_0^1(\Dm)}\le C\|\phi\|_{H^{1,0}_w(\Dm)}$ for $\phi\in C_\ev^\infty(\Dm)$. 
Once again we may assume $\phi$ is supported away from the origin. In that case, we can make an analogous lower bound to \eqref{eq:H01w-upper} to see that
\begin{equation}
\label{eq:H01w-lower}
\begin{aligned}
\|\phi\|_{H^{1,0}_w(\Dm)}^2 &= \int_\Dm\left((1-x^2)|x\partial_x\phi|^2+\frac{1}{1-x^2}|x\partial_{\omega}\phi|^2+|x\phi|^2\right)x^{-1}\d x\d\omega \\
&\ge c\left(\|x\partial_x\phi\|_{x^{-1/2}L^2(\Dm)}^2 + \|x\partial_\omega\phi\|^2_{x^{-1/2}L^2(\Dm)}+\|x\phi\|^2_{x^{-1/2}L^2(\Dm)}\right),
\end{aligned}
\end{equation}
where we can lower bound $1-x^2$ since $\phi$ is supported away from $x=1$. Since we already have $\|x\partial_x\phi\|_{x^{-1/2}L^2(\Dm)}\ge \|x\partial_x\phi\|_{x^{-1/2-\epsilon}L^2(\Dm)}$ and $\|x\partial_\omega\phi\|_{x^{-1/2}L^2(\Dm)}\ge \|x\partial_\omega\phi\|_{x^{-1/2-\epsilon}L^2(\Dm)}$, it suffices to control $\|\phi\|_{x^{-1/2-\epsilon}L^2(\Dm)}$ to get a bound on $\|\phi\|_{x^{-1/2-\epsilon}H_0^1(\Dm)}$. We will prove, for $\phi\in \Cev(\Dm)$ supported away from the origin, an estimate of the form
\begin{equation}
\label{eq:L2est-epsilon}
\|\phi\|_{x^{-1/2-\epsilon}L^2(\Dm)}\le C\|x\partial_x\phi\|_{x^{-1/2}L^2(\Dm)},
\end{equation}
since then, combined with \eqref{eq:H01w-lower}, we have our desired estimate
\begin{align*}
    \|\phi\|_{x^{-1/2-\epsilon}H_0^1(\Dm)} &= \|x\partial_x\phi\|_{x^{-1/2-\epsilon}L^2(\Dm)}+\|x\partial_\omega\phi\|_{x^{-1/2-\epsilon}L^2(\Dm)} + \|\phi\|_{x^{-1/2-\epsilon}L^2(\Dm)} \\
    &\le \|x\partial_x\phi\|_{x^{-1/2}L^2(\Dm)}+\|x\partial_\omega\phi\|_{x^{-1/2}L^2(\Dm)} + C\|x\partial_x\phi\|_{x^{-1/2}L^2(\Dm)}\\
    &\le (1+C)\left(\|x\partial_x\phi\|_{x^{-1/2}L^2(\Dm)} + \|x\partial_\omega\phi\|_{x^{-1/2}L^2(\Dm)}\right)\le (1+C)C'\|\phi\|_{H^{1,0}_w}.
\end{align*}
To prove \eqref{eq:L2est-epsilon}, we note that since $\phi|_{x=1} = 0$, the Fundamental Theorem of Calculus and the Cauchy-Schwarz inequality yields
\begin{align*}
\phi(x,\omega) &= -\int_x^1 \partial_x\phi(t,\omega)\d t= -\int_x^1t^{-1/2}t^{1/2}\partial_x\phi(t,\omega)\d t \\
\implies |\phi(x,\omega)|^2 &\le \left(\int_x^1 t^{-1}\d t\right)\left(\int_x^1|t^{1/2}\partial_x\phi(t,\omega)|^2\d t\right) \le \log(1/x) \int_0^1|(x\partial_x\phi)(t,\omega)|^2 t^{-1}\d t.
\end{align*}
Integrating this on $\Sm^1$ yields
\[\int_{\Sm^1}|\phi(x,\omega)|^2\d\omega\le \log(1/x)\int_{\Sm^1}\int_0^1|(x\partial_x\phi)(t,\omega)|^2 t^{-1}\d t\d\omega = \log(1/x)\|x\partial_x\phi\|_{x^{-1/2}L^2(\Dm)}^2. \]
Multiplying by $x^{-1+2\epsilon}$ and integrating in $x$, we get
\[\|\phi\|_{x^{-1/2-\epsilon}L^2(\Dm)}^2 = \int_{\Sm^1}\int_0^1 x^{1+2\epsilon}|\phi(x,\omega)|^2 x^{-2}\d x\d\omega \le C^2\|x\partial_x\phi\|_{x^{-1/2}L^2(\Dm)}^2,\]
where $C=\sqrt{\int_0^1 x^{-1+2\epsilon}\log(1/x)\d x} = \frac{1}{2\epsilon}$. It follows that $\|\phi\|_{x^{-1/2-\epsilon}L^2(\Dm)}\le C\|x\partial_x\phi\|_{x^{-1/2}L^2(\Dm)}$,
as desired. 
We have now proved the inclusion for all $\epsilon>0$, and it is strict because $x ^\alpha H_0^1(\Dm)\subsetneq x ^{\alpha'} H_0^1(\Dm)$ for $\alpha>\alpha'$.
\end{proof}

\paragraph{Acknowledgments} F.M. gratefully acknowledges partial funding from NSF CAREER grant DMS-1943580 and NSF grant DMS-2606949. N.E. acknowledges travel support from the Graduate Academy of Leibniz University Hannover.
The authors would like to thank C Robin Graham for helpful discussions.

 \bibliographystyle{alpha}

\end{document}